\documentclass[reqno,oneside,a4paper,11pt]{amsart}
\usepackage[a4paper, hmargin={2.8cm, 2.8cm}, vmargin={2.5cm, 2.5cm}]{geometry}
\usepackage[ngerman, english]{babel}
\usepackage[utf8]{inputenc}
\usepackage{pdfsync}
\usepackage{verbatim}
\usepackage[onehalfspacing]{setspace}
\usepackage{amsmath}
\usepackage{amsthm}
\usepackage{amssymb}
\usepackage[pagebackref, colorlinks, linkcolor=red, citecolor=blue, urlcolor=blue, hypertexnames=true]{hyperref}
 \usepackage{paralist}
\theoremstyle{theorem}
\newtheorem{satz}{Theorem}[section]

  \newtheorem{lemma}[satz]{Lemma}
  \newtheorem{kor}[satz]{Corollary}
  \newtheorem{thmx}{Theorem}

  \newtheorem{prop}[satz]{Proposition}
  
  \theoremstyle{definition}
 \newtheorem{defi}[satz]{Definition}
  \newtheorem{bem}[satz]{Remark}

    {\begin{proof}[Beweis]}
    {\end{proof}}

\newcommand{\norm}[1]{\lVert#1\rVert}   %Norm{} befehl
\newcommand{\betrag}[1]{\lvert#1\rvert}
\newcommand{\KK}{\mathrm{KK}}
\newcommand{\RKK}{\mathcal{R}\mathrm{KK}}

\newcommand{\K}{\mathrm K}
\newcommand{\EE}{\mathbb E}
\newcommand{\RR}{\mathbb R}

\newcommand{\CC}{\mathbb C}
\newcommand{\NN}{\mathbb N}
\newcommand{\ZZ}{\mathbb Z}

\newcommand{\lk}{\langle}
\newcommand{\rk}{\rangle}

\usepackage{tikz}
\usetikzlibrary{matrix,arrows}
\allowdisplaybreaks

\title[A Going-Down Principle for ample groupoids and the Baum-Connes conjecture]{A Going-Down Principle for ample groupoids and the Baum-Connes conjecture}
\author{Christian B\"onicke}
\address{Mathematisches Institut der WWU M\"unster,
	\newline Einsteinstrasse 62, 48149 M\"unster, Germany}
\email{christian.bonicke@glasgow.ac.uk}

\curraddr{School of Mathematics and Statistics, University of Glasgow,
	\newline University Gardens, Glasgow, G12 8QQ, UK}

\thanks{Supported by Deutsche Forschungsgemeinschaft (SFB 878).}
\date{\today}
\subjclass[2010]{19K35, 22A22, 46L80}
\keywords{Ample groupoids, Baum-Connes conjecture, KK-theory}
\begin{document}
\maketitle

\begin{abstract}
We study a Going-Down (or restriction) principle for ample groupoids and its applications. The Going-Down principle for locally compact groups was developed by Chabert, Echterhoff and Oyono-Oyono and allows to study certain functors, that arise in the context of the topological K-theory of a locally compact group, in terms of their restrictions to compact subgroups.
We extend this principle to the class of ample Hausdorff groupoids using Le Gall's groupoid equivariant version of Kasparov's bivariant $\KK$-theory. Moreover, we provide an application to the Baum-Connes conjecture for ample groupoids which are strongly amenable at infinity. This result in turn is then used to relate the Baum-Connes conjecture for an ample groupoid group bundle which is strongly amenable at infinity to the Baum-Connes conjecture for the fibres.
\end{abstract}

\tableofcontents
\section{Introduction}
One important step in the study of $\mathrm{C}^*$-algebras is the computation of its $\K$-theory. This is a notoriously difficult problem, especially for group $\mathrm{C}^*$-algebras and crossed products. Baum, Connes, and Higson present in \cite{MR1292018} a general method to attack this problem:

If $G$ is a locally compact, second countable group and $A$ is a $\mathrm{C}^*$-algebra equipped with a strongly continuous action of $G$ by $\ast$-auto\-morphisms, the topological $\K$-theory of $G$ with coefficients in $A$ is defined as
$$\K_*^{\text{top}}(G;A):=\lim\limits_{X\subseteq \mathcal{E}(G)}\KK^G_*(C_0(X),A),$$
where $X$ runs through the $G$-compact (i.e. the quotient space $X/G$ is compact) subspaces of a universal proper $G$-space $\mathcal{E}(G)$ ordered by inclusion, and $\KK_*^G$ denotes Kasparov's equivariant $\KK$-theory.
The authors in \cite{MR1292018} then proceed to construct a group homomorphism
$$\mu_A:\K^{\mathrm{top}}_*(G;A)\rightarrow \K_*(A\rtimes_rG).$$
This map is usually called the \textit{assembly map} and the Baum-Connes conjecture asserts, that $\mu_A$ is an isomorphism. By work of Higson, Lafforgue and Skandalis (see \cite{MR1911663}) it is now known, that the conjecture is false in this generality. It has however been proven to be true for large classes of groups including the class of amenable groups and the conjecture with trivial coefficients (i.e. $A=\CC$) is still open.

In his thesis \cite{LeGall}, Le Gall introduced a groupoid equivariant version of Kasparov's $\KK$-theory, which was subsequently used to define a version of the Baum-Connes assembly map for groupoids (see \cite{MR1798599} for a survey). The question, when this map is an isomorphism has been investigated by Tu in \cite{Tu98,Tu99}. He proves that the Baum-Connes conjecture is true for every locally compact, $\sigma$-compact Hausdorff groupoid acting continuously and isometrically
on a continuous field of affine Euclidean spaces. The latter condition is fulfilled in particular by all amenable groupoids. On the other hand, the groupoid version of the Baum-Connes conjecture is known to be false even in the case of trivial coefficients (again by results in \cite{MR1911663}).

In the case of locally compact groups, Chabert started in \cite{MR1752299} to study permanence properties of the Baum-Connes conjecture for the case of semi-direct products. In subsequent work of Chabert and Echterhoff (see \cite{MR1836047}) these methods were refined and it was proved that the class of groups satisfying the conjecture is stable under taking subgroups, Cartesian products, and certain group extensions. A similar approach was used in \cite{MR1966758} to prove that the topological $\K$-theory of a transformation groupoid $G\ltimes X$ does not depend on $X$, i.e. that the canonical forgetful map $\K_*^{\text{top}}(G\ltimes X;A)\rightarrow \K_*^{\text{top}}(G;A)$ is an isomorphism. Finally, in \cite{CEO}, the authors formalize the methods used to prove the main results in all of the above mentioned work and abstractly develop the so called Going-Down principle, which allows to analyse certain functors connected to the topological $\K$-theory of a locally compact group in terms of their restrictions to compact subgroups.
The Going-Down principle has turned out to be very useful in the computation of the $\K$-theory of certain $\mathrm{C}^*$-algebras, for example crossed products of the irrational rotation algebras by finite subgroups of $SL_2(\ZZ)$ (see \cite{MR2608195}) or the $\mathrm{C}^*$-algebras associated to a large class of semigroups (see \cite{MR3094498,MR3323201}).

The starting point of this paper is the work of Tu, who proves in \cite{Tu12} an analogue of the main result of \cite{MR1966758} for second countable, locally compact, étale groupoids and uses it to show that satisfying the Baum-Connes conjecture passes to subgroupoids (within this class). Inspired by the ideas in this work we set out to develop a general Going-Down principle in the spirit of \cite{CEO} for the class of ample groupoids. Although it seems plausible that similar results can be obtained for all étale groupoids, there are a lot of topological difficulties yet to overcome. In the case of ample groupoids however these difficulties disappear and the theory can be developed beautifully. Many interesting examples studied in the literature fall naturally into the class of ample groupoids.

Let us summarize our main results and simultaneously give an overview of how this paper is organized.

After reviewing some preliminaries on groupoids and proper actions we focus on a detailed study of induced algebras. One way to look at the induced algebras we are interested in is to use the picture of pullbacks along generalized morphisms of groupoids as developed by Le Gall in \cite{LeGall}. We however chose to develop the theory in analogy to the classical approach in the group case, which seems to be more useful for our purposes. To the best of our knowledge this approach has not been carried out before in the literature.

We then turn to the study of Le Gall's groupoid equivariant version of Kasparov's $\KK$-theory. We prove a generalization of a result of Meyer (see \cite{Meyer00equivariantkasparov}) on when the operator in an equivariant Kasparov triple can be chosen to be equivariant with respect to the action of the groupoid. We then proceed to prove one of the main technical ingredients in the proof of the Going-Down principle. It says that a canonically defined compression homomorphism $\textsf{comp}_H^G$ is an isomorphism:
\begin{thmx}(see Theorem \ref{CompressionIsomorphism})
	Let $G$ be an étale, locally compact Hausdorff groupoid with a clopen, proper subgroupoid $H\subseteq G$. Let $X:=G_{H^{(0)}}$. If $A$ is an $H$-algebra and $B$ is a $G$-algebra, then
	$$\textsf{comp}_H^G:\KK^G(Ind_H^XA,B)\rightarrow \KK^H(A,B_{\mid H})$$
	is an isomorphism.
\end{thmx}

Section 7 focuses solely on the proof of the Going-Down principle for ample groupoids. For convenience, this paper focuses on the following case of the Going-Down principle, which is the main technical result of this paper:
\begin{thmx}(see Theorem \ref{MainTheorem})
	Suppose $G$ is an ample, locally compact Hausdorff groupoid and $A$ and $B$ are $G$-algebras. Suppose there is an element $x\in \KK^G(A,B)$ such that 
	\[ \KK^H(C(H^{(0)}),A_{\mid H})\stackrel{\cdot\otimes res_H^G(x)}{\longrightarrow}\KK^H(C(H^{(0)}), B_{\mid H})\]
	is an isomorphism for all compact open subgroupoids $H\subseteq G$. Then the Kasparov-product with $x$ induces an isomorphism
	$$ \cdot\otimes x:\K_*^{\mathrm{top}}(G;A)\rightarrow \K_*^{\mathrm{top}}(G;B).$$
\end{thmx}
%The proof proceeds by reducing the statement in three steps:
%\begin{enumerate}
%	\item In the first step we use a specific approximation of the universal proper $G$-space by finite dimensional $G$-simplicial complexes.
%	\item Using the existence of long exact sequences in $\KK^G$-theory and Bott periodicity we reduce our problem to the case of zero-dimensional $G$-simplicial complexes.
%	\item In the zero-dimensional case one can use a Mayer-Vietoris type argument to restrict attention to the even more special case that the space in question is the $G$-saturation of a compact open set. To this compact open set we can canonically associate a compact open subgroupoid $H$ of $G$ such that $X$ is the induced space of an $H$-space. Finally, we apply the compression isomorphism to obtain the result.
%\end{enumerate}

The final section is dedicated to illustrate an application of the Going-Down principle.
It revolves around the recent notion of (strong) amenability at infinity for étale groupoids as introduced by Lassagne in \cite{Lassagne} (see also \cite{Delaroche}). Based on ideas of Higson (see \cite{Higson}) we prove the following result:
\begin{thmx}(see Theorem \ref{Theorem:BC Injectivity  for amenable at infty grpds})
	Let $G$ be a second countable ample groupoid, which is strongly amenable at infinity and let $A$ be a $G$-algebra. Then the Baum-Connes assembly map
	$$\mu_A:\K_*^{\mathrm{top}}(G;A)\rightarrow \K_*(A\rtimes_r G)$$
	is split injective.
\end{thmx}

The counterexamples to the Baum-Connes conjecture for groupoids presented in \cite{MR1911663} are in fact ample groupoid group bundles. Consequently, it is in turn very natural to ask when such a group bundle does satisfy the Baum-Connes conjecture.
As an application of our results we are able to relate the Baum-Connes conjecture for such a bundle to the Baum-Connes conjecture for each of the fibre groups. More precisely, we prove the following:
\begin{thmx}(see Theorem \ref{Thm:BC for group bundles})
	Let $G$ be a second countable ample group bundle, which is strongly amenable at infinity. Suppose $A$ is a $G$-algebra such that the associated $C^*$-bundle is continuous, and $G_u^u$ satisfies the Baum-Connes conjecture with coefficients in $A_u$ for all $u\in G^{(0)}$. Then $G$ satisfies the Baum-Connes conjecture with coefficients in $A$. 
\end{thmx}

Given the length of the current version of this paper further applications will appear in separate articles. One is a joint work with Clément Dell'Aiera on the Künneth formula for crossed products by ample groupoids \cite{Preprint2} and the second will deal with the $\K$-theory of twisted groupoid $\mathrm{C}^*$-algebras (see \cite{Preprint1}).

\section{Preliminaries on groupoids and proper actions}
Recall, that a \textit{groupoid} is a set $G$ together with a subset $G^{(2)}\subseteq G\times G$, called the set of \textit{composable pairs,} a \textit{product} map $(g,h)\mapsto gh$ from $G^{(2)}$ to $G$ and an \textit{inverse} map $g\mapsto g^{-1}$ from $G$ onto $G$, such that the following hold:
	\begin{enumerate}
		\item The product is associative: If $(g_1,g_2),(g_2,g_3)\in G^{(2)}$ for some $g_1,g_2,g_3\in G$, then we also have $(g_1g_2,g_3),(g_1,g_2g_3)\in G^{(2)}$ and 
		\begin{equation*}
		(g_1g_2)g_3=g_1(g_2g_3).
		\end{equation*}
		\item The inverse map is involutive, i.e. $(g^{-1})^{-1}=g$ for all $g\in G$.
		\item $(g,g^{-1})\in G^{(2)}$ for all $g\in G$ and if $(g,h)\in G^{(2)}$, then 
		\begin{equation*}
		g^{-1}(gh)=h \ \textup{ and }\ (gh)h^{-1}=g.
		\end{equation*}
	\end{enumerate}
The fact that multiplication is partially defined implies that multiple elements may act as (partial) units:
	The set $G^{(0)}:=\lbrace g\in G\mid g=g^{-1}=g^2\rbrace$ is called the \textit{set of units} in $G$. There are canonical maps $d:G\rightarrow G^{(0)}$ given by $d(g)=g^{-1}g$ and $r:G\rightarrow G^{(0)}$ given by $r(g)=gg^{-1}$, called the \textit{domain} and \textit{range} \textit{map} respectively.

For subsets $A,B\in G^{(0)}$ we will write $G_A:=d^{-1}(A)$, $G^B:=r^{-1}(B)$ and $G_A^B:=G_A\cap G^B$. If $A$ (and/or $B$) consists just of a single unit $u\in G^{(0)}$ we will omit the braces (e.g.: we will write $G^u:=r^{-1}(\lbrace u\rbrace)$).

In this paper we will be concerned with topological groupoids:
We say that $G$ is a locally compact Hausdorff groupoid, if $G$ is a groupoid, which is equipped with a locally compact Hausdorff topology, such that the multiplication and inversion map are continuous. The fact that $G$ is Hausdorff ensures that the unit space $G^{(0)}$ is closed in $G$.

We will mainly deal with étale groupoids. Recall, that a locally compact groupoid is called \textit{étale}, if $d:G\rightarrow G$ is a local homeomorphism, i.e. every point $g\in G$ has an open neighbourhood $U\subseteq G$, such that $d(U)$ is open in $G$ and $d_{\mid U}:U\rightarrow d(U)$ is a homeomorphism.
It follows easily from the definition that for an étale groupoid $G$
the unit space $G^{(0)}$ is open in $G$ and for each $u\in G^{(0)}$ the sets $G^u$ and $G_u$ are discrete (in the subspace topology).

An \textit{open bisection} in a locally compact Hausdorff groupoid $G$ is an open subset $U\subseteq G$ such that the domain map $d$ and the range map $r$ are homeomorphisms onto open subsets of $G$ respectively. The set of all open bisections will be denoted by $G^{op}$.
It is well-known, that $G$ is étale if and only if $G^{op}$ contains a basis for the topology of $G$.

One of the most powerful tools in the study of locally compact groups is the existence of the Haar measure. There is an analogous notion for groupoids. Recall, that a (\textit{left}) \textit{Haar system} for a locally compact Hausdorff groupoid $G$ is a collection $(\lambda^u)_{u\in G^{(0)}}$ of positive regular Borel measures on $G$ such that the following hold:
	\begin{enumerate}
		\item The support of each $\lambda^u$ is $G^u$.
		\item \label{Def:HaarSystem:Continuity} For any $f\in C_c(G)$ the function $\lambda(f):G^{(0)}\rightarrow \CC$, given by 
		\begin{equation*}
		\lambda(f)(u):=\int\limits_{G^u}f\ d\lambda^u
		\end{equation*}
	is continuous (and hence belongs to $C_c(G^{(0)})$).
		\item \label{Def:HaarSystem:Invariance} For any $g\in G$ and $f\in C_c(G)$ we have the equation
		\begin{equation*}
		\int\limits_{G^{d(g)}} f(gh)d\lambda^{d(g)}(h)=\int\limits_{G^{r(g)}}f(h)d\lambda^{r(g)}(h).
		\end{equation*}
	\end{enumerate}

In the case of a locally compact group the above definition reduces to the definition of (the) Haar measure. One should note that in contrast to the group case, locally compact groupoids neither necessarily admit a Haar measure (see \cite{MR813822} for a counterexample), nor is it unique.

As we have $(G^u)^{-1}=G_u$ and the inversion map is a homeomorphism from $G$ onto itself, we associate with $\lambda^u$ the measure $\lambda_u:=(\lambda^u)^{-1}$ on $G_u$, given by $\lambda_u(A)=\lambda^u(A^{-1})$ for a Borel subset $A\subseteq G_u$.
Consequently, we get the formula
\begin{equation*}
\int\limits_{G_u}f(g)d\lambda_u(g)=\int\limits_{G^u}f(g^{-1})d\lambda^u(g).
\end{equation*}

The existence of a Haar system on a locally compact groupoid $G$ has strong topological consequences. Indeed Renault shows in \cite[Proposition~2.4]{Renault1980}, that whenever $G$ admits a Haar system, then the range and the domain map are necessarily open maps.

The domain and range maps being open is reminiscent of étale groupoids, which always have this property. Indeed, every étale group\-oid admits a particularly nice canonical Haar system given by the family of counting measures (see \cite[Proposition~2.2.5]{Paterson1999} for a proof).

\textbf{Convention: From now on, when talking about étale group\-oids, we will always take this family of counting measures as the canonical Haar system.}

The following well-known basic result will be needed later:
\begin{lemma}\label{measure of compact set is bounded}
	Let $G$ be a locally compact Hausdorff groupoid with a Haar system $\lbrace \lambda^u\rbrace_{u\in G^{(0)}}$. If $K\subseteq G$ is compact, the set $\lbrace \lambda^u(K)\mid u\in G^{(0)}\rbrace$ is bounded.
\end{lemma}
\begin{proof}
	Let $f\in C_c(G)$ with $0\leq f\leq 1$ and $f=1$ on $K$. Then $$\lambda^u(K)\leq \int\limits_{G^u} f(x)d\lambda^u(x)$$ for all $u\in G^{(0)}$. The result follows from axiom $(2)$ of the definition of a Haar system.
\end{proof}
For later purposes it will also be important to note, that the set of functions $f$ for which $\lambda(f)$ as in the definition of the Haar system is continuous, is not limited to functions with compact support. 
\begin{defi}\label{Def:ProperlySupportedFunction}
	A function $\varphi\in C(G)$ is said to have \textit{proper support}, if for every compact subset $K\subseteq G^{(0)}$ the intersection $supp(\varphi)\cap r^{-1}(K)$ is compact.
\end{defi}

\begin{lemma}\label{Lem:proper support}
	If $\varphi\in C(G)$ has proper support, then $\lambda(\varphi):G^{(0)}\rightarrow \CC$ given by
	$$\lambda(\varphi)(u)=\int\limits_{G^u}\varphi(x)d\lambda^u(x)$$
	is continuous and bounded.
\end{lemma}
\begin{proof}
	We will show that $\lambda(\varphi)$ looks like a continuous function locally. More precisely given any $u\in G^{(0)}$ we can pick a relatively compact neighbourhood $V$ of $u$. Then choose a function $\psi\in C_c(G^{(0)})$ such that $\psi= 1$ on $\overline{V}$.
	Then $f(x):=\varphi(x)\psi(r(x))$ is a continuous function with compact support
	since $supp(f)\subseteq supp(\varphi)\cap r^{-1}(supp(\psi))$ and $\varphi$ has proper support.
	Thus $\lambda(f)$ is continuous. But for all $v\in V$ we clearly have
	$$\lambda(f)(v)=\int\limits_{G^v}\varphi(x)\psi(v)d\lambda^v(x)=\lambda(\varphi)(v).$$
	Thus $\lambda(\varphi)_{\mid V}$ is continuous. Since $u$ was chosen arbitrary $\lambda(\varphi)$ must be continuous.
\end{proof}

There is an important subclass of the class of étale groupoids, which is of particular interest to us:
\begin{defi}
	A locally compact Hausdorff groupoid $G$ is called \textit{ample}, if the set $G^a:=\lbrace A\in G^{op}\mid A\textit{ is compact}\rbrace$ forms a basis for the topology of $G$.
\end{defi}
It follows directly from the definition, that every ample groupoid is étale. Recall, that a topological space $X$ is called \textit{totally disconnected}, if the connected components in $X$ are the one-point sets. It is easy to see that every ample groupoid has totally disconnected unit space.
Exel noted in \cite{Exel10}, that this characterizes the ample groupoids among the étale groupoids, i.e. an étale groupoid $G$ is ample if and only if $G^{(0)}$ is totally disconnected.
Many interesting groupoids fall into this class:
\begin{itemize}
	\item Groupoids associated to aperiodic tilings and quasicrystals (see \cite{MR2658985}).
	\item Groupoids associated to directed graphs (see \cite{MR1432596}) and higher-rank graphs (see \cite{KumjianPask00}).
	\item Groupoids associated to inverse semigroups (see \cite{Paterson1999}).
	\item The coarse groupoid studied in large scale geometry (see \cite{STY02}).
\end{itemize}

Let us now turn to actions of groupoids. Recall, that a (left) action of a locally compact Hausdorff groupoid $G$ on a locally compact Hausdorff space $X$ consists of a continuous map $p:X\rightarrow G^{(0)}$, called \textit{anchor map} and a continuous map $G\ast X\rightarrow X$, $(g,x)\mapsto gx$, where $G\ast X=\lbrace (g,x)\mid d(g)=p(x)\rbrace$, such that the following holds:
	\begin{enumerate}
		\item If $(g,h)\in G^{(2)}$ and $(h,x)\in G\ast X$, then $(g,hx)\in G\ast X$ and $(gh)x=g(hx)$.
		\item For all $x\in X$ we have $p(x)x=x$. 
	\end{enumerate}
In this case we will also say that $X$ a (left) $G$-space.
Similarly, one can define right actions in the obvious way.
%\begin{exs} Let $G$ be a locally compact Hausdorff groupoid.
%	\begin{enumerate}
%		\item $G$ acts on itself, where the anchor map $G\rightarrow G^{(0)}$ is the range map and the action is given by the usual multiplication of the groupoid.
%		\item If $H\subseteq G$ is a closed subgroupoid, then $H$ acts from the right on $X=d^{-1}(H^{(0)})$, where the anchor map is the restriction of the domain map $d_{\mid X}:X\rightarrow H^{(0)}$ and the action is given by the multiplication in $G$.
%		\item $G$ acts from the left on its unit space, where the anchor map is the identity on $G^{(0)}$ and the action is given by $g\cdot d(g)=r(g)$.
%		\item Let $P(G)$ be the space of Borel probability measures on $G$, such that the support of each probability measure is contained in $G^u$ for some $u\in G^{(0)}$. Then $G$ acts on $P(G)$: The anchor map $P(G)\rightarrow G^{(0)}$ sends $\mu$ to $u$, where $u$ is the unit such that $supp(\mu)\subseteq G^u$. If $d(g)=u$ we define the measure $g\mu$ by $(g\mu)(A)=\mu(g^{-1}A)$.
%		\item Finally, consider the \textit{isotropy subgroupoid} $Iso(G)=\lbrace s\in G\mid d(s)=r(s)\rbrace$ of $G$. Then $G$ acts from the left on $Iso(G)$ with anchor map $p(s)=d(s)=r(s)$ and for $g\in G$ and $s\in Iso(G)$ we can define $g\cdot s=gsg^{-1}$.
%	\end{enumerate}
% (action of $G$ on $\beta_r(G)$)\\
%\end{exs}
Groupoid actions give rise to a new groupoid, usually called the transformation groupoid of the action:
If $G$ acts on $X$ we can form a new groupoid denoted $G\ltimes X$. As a set it is the subspace of $G\times X$ consisting of all pairs such that $r(g)=p(x)$. Two such pairs $(g,x),(h,y)$ are composable if $y=g^{-1}x$ and in that case we define \[(g,x)(h,y):=(gh,x).\] Furthermore we define the inverse map by \[(g,x)^{-1}:=(g^{-1},g^{-1}x).\]
The unit space of $G\ltimes X$ can be canonically identified with $X$.
Under this identification the range map becomes the projection onto $X$ and the domain map is given by $d_{G\ltimes X}(g,x)=g^{-1}x$.
One easily verifies, that if $G$ is étale, then so is $G\ltimes X$.

If $G$ acts on $X$, say from the right, we can form the space of orbits $X/G$. More specifically we can define an equivalence relation $\sim$ on $X$ by declaring $x\sim y$ if and only if there exists a $g\in G$ such that $p(y)=r(g)$ and $x=yg$. We then define $X/G:=X/\sim$ to be the quotient of $X$ by the equivalence relation $\sim$. If $G$ was a topological groupoid acting continuously on a space $X$ we equip $X/G$ with the quotient topology.
The following result is standard. A proof can be found in \cite[Lemma~2.30]{MR2117427}.
\begin{prop} \label{Prop:QuotientLocallyCompact}
	Let $G$ be a locally compact Hausdorff groupoid. Then the range and domain maps of $G$ are open if and only if the canonical quotient map $X\rightarrow X/G$ is open for every $G$-space $X$. In that case $X/G$ is locally compact (not necessarily Hausdorff), if $X$ is locally compact.
\end{prop}

Many properties of dynamical systems can easily be formulated in terms of the corresponding transformation groupoid and thus give a nice way to generalize them to arbitrary groupoids. The following is an example of this:
Recall, that a continuous map $f:X\rightarrow Y$ between locally compact Hausdorff spaces $X$ and $Y$ is called \textit{proper}, if $f^{-1}(K)$ is compact for all compact subsets $K\subseteq Y$.
If $\Gamma$ is a discrete group acting on a space $X$, the action is called proper, if $(g,x)\mapsto (x,g^{-1}x)$ is a proper map $\Gamma\times X\rightarrow X\times X$. In terms of the transformation groupoid the latter map is just the map $r\times d:\Gamma\ltimes X\rightarrow X\times X$. Thus, for general groupoids, one defines:
\begin{defi}
	A locally compact Hausdorff groupoid is called \textit{proper}, if $r\times d:G\rightarrow G^{(0)}\times G^{(0)}$ is a proper map.
	
	Similarly, we say that $X$ is a \textit{proper (left) $G$-space}, if the associated transformation groupoid $G\ltimes X$ is proper.
\end{defi}
In practice it is useful to have some more equivalent conditions to check properness. These are provided by the following proposition.
\begin{prop}\cite[Proposition~2.14]{MR2117427}\label{Prop:CharacterizationsProperAction}
Let $X$ be a locally compact $G$-space. Then the following are equivalent:
\begin{enumerate}
\item $X$ is a proper $G$-space.
\item For every compact subset $K\subseteq X$ the set $\mathcal{F}_K=\lbrace g\in G\mid gK\cap K\neq\emptyset\rbrace$ is compact.
\item If $(x_\lambda)_\lambda$ is a convergent net in $X$ and $(g_\lambda)_\lambda$ is a net in $G$ such that $d(g_\lambda)=p_X(x_\lambda)$ and $(g_\lambda x_\lambda)_\lambda$ is convergent as well, then $(g_\lambda)_\lambda$ has a convergent subnet.
\end{enumerate}
\end{prop}

\begin{bem}
It is useful to note, that the set $\mathcal{F}_K$ defined above for any compact set $K\subseteq X$ is always closed in $G$. To see this let $(g_\lambda)_\lambda$ be a net in $\mathcal{F}_K$ converging to some $g\in G$. For every $\lambda$ there exist $k_\lambda,k_\lambda'\in K$ such that $g_\lambda k_\lambda=k_\lambda'$. As $K$ is compact we can pass to a subnet if necessary to assume that $k_\lambda\rightarrow k$ and $k_\lambda'\rightarrow k'$ for some $k,k'\in K$. By continuity of the action we have $gk=\lim_\lambda g_\lambda k_\lambda=\lim_\lambda k_\lambda'=k'$. Thus, we have $g\in\mathcal{F}_K$, as desired.
\end{bem}

Note, that it follows easily from the above characterization, that every groupoid $G$ acts properly on itself.
Identifying $G$ with the transformation groupoid $G\ltimes G^{(0)}$ in the obvious way we get a similar looking result characterizing properness of the groupoid itself:
\begin{prop}
Let $G$ be a locally compact Hausdorff group\-oid. Then the following are equivalent:
\begin{enumerate}
\item $G$ is proper.
\item For every compact subset $K\subseteq G^{(0)}$ the set $G_K^K$ is compact.
\item If $(g_\lambda)_\lambda$ is a net in $G$, such that $(d(g_\lambda))_\lambda$ and $(r(g_\lambda))_\lambda$ are convergent, then $g_\lambda$ has a converging subnet. 
\end{enumerate}
\end{prop}

One of the features of proper Hausdorff groupoids is the fact, that their orbit space is again Hausdorff.
%\begin{lemma}\label{Lem:QuotientsByProperActionsHausdorff}
%Let $G$ be a proper Hausdorff groupoid with open range and domain maps. Then the quotient space $G\setminus G^{(0)}$ for the canonical left action of $G$ on $G^{(0)}$ is Hausdorff.
%\end{lemma}
%\begin{proof}
%Suppose $(Gu_\lambda)_\lambda$ is a net in the quotient $G\setminus G^{(0)}$ converging to both $Gu$ and $Gv$. We claim that $Gu=Gv$. Our assumptions together with Proposition \ref{Prop:QuotientLocallyCompact} imply, that the quotient map $G^{(0)}\rightarrow G\setminus G^{(0)}$ is open. Thus, we can pass to a subnet, relabel if necessary, and choose new representatives $u_\lambda$, to assume that $u_\lambda\rightarrow u$. Then we can use openness of the quotient map again to find elements $g_\lambda\in G$, such that $r(g_\lambda)=g_\lambda u_\lambda \rightarrow v$. Hence we can use the characterization of properness from the previous proposition to pass to another subnet and relabel, allowing us to assume that $g_\lambda\rightarrow g$ for some $g\in G$. But then
%$v=\lim g_\lambda u_\lambda=gu$, which proves the claim.
%\end{proof}

\begin{lemma}
Let $G$ be a locally compact Hausdorff groupoid and $H\subseteq G$ a subgroupoid with $H^{(0)}$ closed in $G^{(0)}$. If $H$ is proper, then $H$ is closed in $G$.
\end{lemma}
\begin{proof}
Let $(g_\lambda)_\lambda$ be a net in $H$ converging to $g\in G$. Let $K$ be a compact neighbourhood of $g$. After passing to a subnet if necessary, we can assume $g_\lambda\in K\cap H\subseteq H_{d(K)}^{r(K)}$. Since $H$ is proper, the latter set is compact and hence closed as a subset of $G$. Thus, we get $g=\lim_\lambda g_\lambda \in H_{d(K)}^{r(K)}\subseteq H$.
\end{proof}

There is a close connection between proper actions and so called induced spaces. Let us review the definition:
Let $G$ be a locally compact Hausdorff groupoid and $H\subseteq G$ a closed subgroupoid. Suppose $Y$ is a (left) $H$-space with anchor map $p:Y\rightarrow H^{(0)}$. Consider the set
$$G\times_{G^{(0)}}Y=\lbrace (g,y)\in G\times Y\mid d(g)=p(y)\rbrace$$
There is a canonical action of $H$ on $G\times_{G^{(0)}}Y$: The anchor map $P:G\times_{G^{(0)}}Y\rightarrow H^{(0)}$ is given by $P(g,y)=d(g)=p(y)$ and we define $h(g,y)=(gh^{-1},hy)$.
\begin{lemma}
The action of $H$ on $G\times_{G^{(0)}}Y$ defined above is proper.
\end{lemma}
\begin{proof}
Let $K\subseteq G\times_{G^{(0)}}Y$ be a compact subset. We need to show that $\mathcal{F}_K=\lbrace h\in H\mid hK\cap K\neq \emptyset\rbrace$ is a compact subset of $H$. If $K_1=pr_1(K)$ is the image of $K$ under the projection onto $G$ it is not hard to see that $\mathcal{F}_K\subseteq K_1^{-1}K_1\cap H$. Since the latter set is compact and $\mathcal{F}_K$ is closed in $H$, the result follows.
\end{proof}

It follows from the above Lemma combined with the fact that quotients by proper actions are Hausdorff and Proposition \ref{Prop:QuotientLocallyCompact} that the quotient space $G\times_H Y:=H\setminus (G\times_{G^{(0)}}Y)$ is a locally compact Hausdorff space. This space is called the\textit{ induced space}.
There is a canonical left action of $G$ on $G\times_H Y$, coming from the action of $G$ on itself. The anchor map $G\times_H Y\rightarrow G^{(0)}$ is given by $[g,y]\mapsto r(g)$ and we define $g_1[g_2,y]:=[g_1g_2,y]$. One easily checks, that this gives a well-defined continuous action.
%The following proposition gives a characterization, when a given $G$-space $X$ is induced from a closed subgroupoid via the procedure described above.
%\begin{prop}
%Let $X$ be a $G$-space and $H\subseteq G$ a closed subgroupoid. Then the following are equivalent:
%\begin{enumerate}
%\item There exists an $H$-space $Y$ such that $X$ is $G$-equivariantly homeomorphic to $G\times_H Y$.
%\item There exists a continuous, $G$-equivariant map $\varphi:X\rightarrow G/H$.
%\end{enumerate}
%\end{prop}
\begin{lemma} Let $G$ be a locally compact Hausdorff groupoid with open domain and range maps.
If $H\subseteq G$ is a closed subgroupoid and $Y$ is a proper $H$-space, then $G\times_H Y$ is a proper $G$-space.
\end{lemma}
\begin{proof}
We will check condition $(4)$ in \ref{Prop:CharacterizationsProperAction}. Let $([g_\lambda,y_\lambda])_\lambda$ be a convergent net in $G\times_H Y$ with limit $[g,y]$ and let $(h_\lambda)_\lambda$ be a net in $G$ with $d(h_\lambda)=r(g_\lambda)$ and such that $(h_\lambda[g_\lambda,y_\lambda])_\lambda$ is convergent as well. We have to check, that $(h_\lambda)_\lambda$ has a convergent subnet. Our assumptions imply, that the quotient map $G\times_{G^{(0)}} Y\rightarrow G\times_H Y$ is open. Hence we can pass to a subnet and relabel twice, to assume that $(g_\lambda,y_\lambda)\rightarrow (g,y)$ and $(h_\lambda g_\lambda,y_\lambda)$ converges as well. Using the fact, that $G$ acts properly on itself this implies, that $(h_\lambda)_\lambda$ has a convergent subnet, as required.
\end{proof}

\section{Induced algebras}
In this section we first review the notions of $C_0(X)$-algebras and upper-semicontinuous $\mathrm{C}^*$-bundles and groupoid dynamical systems and then define induced $C^*$-algebras.
Recall that if $X$ is a locally compact Hausdorff space and $A$ is a $\mathrm{C}^*$-algebra, then we call $A$ a $C_0(X)-algebra$ if there exists a non-degenerate $\ast$-homomorphism
$$\Phi:C_0(X)\rightarrow Z(M(A)),$$ where $Z(M(A))$ denotes the center of the multiplier algebra of $A$. For every $x\in X$ there is a closed ideal $I_x$ in $A$ defined by $I_x=\overline{C_0(X\setminus\lbrace x\rbrace)A}$ and we call the quotient $A_x:=A/I_x$ the \textit{fibre} of $A$ over $x$. We write $a(x)$ for the image of $a\in A$ in $A_x$ under the quotient map. Put $\mathcal{A}=\coprod_{x\in X} A_x$. Then $\mathcal{A}$ can be equipped with a topology such that it becomes an upper-semicontinouos $\mathrm{C}^*$-bundle over $X$ and moreover $A\cong \Gamma_0(X,\mathcal{A})$, where $\Gamma_0(X,\mathcal{A})$ denotes the continuous sections of this bundle which vanish at infinity. For further reference let us record, that a basis for the topology of $\mathcal{A}$ is defined by the sets
$$W(a,U,\varepsilon):=\lbrace b\in\mathcal{A}\mid q(b)\in U\textit{ and }\norm{b-a(q(b))}<\varepsilon\rbrace,$$
where $a\in A$, $U\subseteq X$ is an open subset and $\varepsilon>0$.
Throughout this work we will freely alternate between the bundle picture and the picture as $C_0(X)$-algebras. For convenience bundles will always be denoted by calligraphic letters.
The reader unfamiliar with the theory is referred to the expositions in \cite[Appendix C]{Williams} and \cite[Section~3.1]{Goehle}.

The following density criterion will turn out to be very useful, when working with $C_0(X)$-algebras. The proof can be adapted easily from \cite[Proposition~C.24]{Williams}.
\begin{prop}\label{Prop:DensityCriterionC(X)-algebras}
Let $A$ be a $C_0(X)$-algebra and $\Gamma\subseteq A$ be a linear subspace. Assume additionally, that 
\begin{enumerate}
\item $\Gamma$ is closed under the action of $C_0(X)$, meaning $fa\in\Gamma$ for all $f\in C_0(X)$ and $a\in\Gamma$, and
\item the image of $\Gamma$ under the quotient map $A\rightarrow A_x$ is dense in $A_x$ for all $x\in X$.
\end{enumerate}
Then $\Gamma$ is dense in $A$.
\end{prop}
An application of this result is contained in the proof of the next well-known lemma. Before we can state it, recall that a $\ast$-homomorphism $\Phi:A\rightarrow B$ between two $C_0(X)$-algebras $A$ and $B$ is called \textit{$C_0(X)$-linear} if $\Phi(f a)=f \Phi(a)$ for all $f\in C_0(X)$ and all $a\in A$.

If $\Phi:A\rightarrow B$ is a $C_0(X)$-linear homomorphism, it induces $\ast$-homo\-morphisms $\Phi_x:A_x\rightarrow B_x$ on the level of the fibres given by $\Phi_x(a(x))=\Phi(a)(x)$.
Conveniently, one can check several properties of $\Phi$ on the level of the fibres and vice versa:
\begin{lemma}\cite[Lemma~2.1]{MR2820377}\label{Lem:IsomorphismCriteriumForC(X)-linearHomomorphisms}
	Let $\Phi:A\rightarrow B$ be a $C_0(X)$-linear homomorphism. Then $\Phi$ is injective (resp. surjective, resp. bijective) if and only if $\Phi_x$ is injective (resp. surjective, resp. bijective) for all $x\in X$.
\end{lemma}

We shall need several constructions involving $C_0(X)$-algebras:

\paragraph{Pullback.}
If $A$ is a $C_0(X)$-algebra and $f:Y\rightarrow X$ a continuous map, we can define the \textit{pullback} of $A$ along $f$ as follows:
Let $q:\mathcal{A}\rightarrow X$ denote the upper-semicontinouos $\mathrm{C}^*$-bundle over $X$ associated to $A$. Then we can form the pullback bundle $f^*\mathcal{A}=\lbrace ((y,a)\in Y\times\mathcal{A}\mid f(y)=q(a)\rbrace$. The bundle $f^*\mathcal{A}$ is an upper-semicontinouos $\mathrm{C}^*$-bundle over $Y$ whose fibres $(f^*\mathcal{A})_y$ are canonically isomorphic to $A_{f(y)}$. We let $f^*A:=\Gamma_0(Y,f^*\mathcal{A})$ denote the corresponding $C_0(Y)$-algebra. Note, that we can canonically identify $(f^*A)_y=A_{f(y)}$.
It is an easy exercise to show that if $A$ is a $C_0(X)$-algebra and $f:Y\rightarrow X$ and $g:Z\rightarrow Y$ are two continuous maps, then the algebras $(f\circ g)^*A$ and $g^*(f^*A)$ are canonically isomorphic as $C_0(Z)$-algebras.

It is an easy consequence of Proposition \ref{Prop:DensityCriterionC(X)-algebras} and often helpful when working with pullbacks that $$span\lbrace \varphi\otimes a\mid \varphi\in C_c(Y), a\in A\rbrace$$
is dense in $f^*A$, where $\varphi\otimes a$ is given by $\varphi\otimes a\in \Gamma_c(Y,f^*\mathcal{A})$ by $$(\varphi\otimes a)(y):=\varphi(y)a(f(y)).$$

When working with crossed products it is often useful to consider another topology on the algebra of continuous sections $\Gamma(X,\mathcal{A})$ of an upper-semicontinuous $\mathrm{C}^*$-bundle. We say that a net $(f_\lambda)_\lambda$ of functions in $\Gamma(X,\mathcal{A})$ converges to $f\in \Gamma(X,\mathcal{A})$ with respect to the \textit{inductive limit topology}, if and only if there exists a compact subset $K$ in $X$ such that $f$ and, eventually, all the $f_\lambda$ vanish off of $K$ and $\norm{f_\lambda - f}_\infty \rightarrow 0$.
One can show (see \cite[Corollary~3.45]{Goehle}), that $span\lbrace \varphi\otimes a\mid \varphi\in C_c(Y), a\in A\rbrace$
is also dense in $\Gamma_c(Y,f^*\mathcal{A})$ with respect to the inductive limit topology.

The next lemma studies the behaviour of pullbacks with respect to $C_0(X)$-linear $\ast$-homomorphisms. The proof is straightforward:

\begin{lemma}\label{Lem:PullbackOfHomomorphisms}
	Let $A$ and $B$ be two $C_0(X)$-algebras and $f:Y\rightarrow X$ a continuous map. If $\Phi:A\rightarrow B$ is a $C_0(X)$-linear homomorphism, then the map
	$$f^*\Phi:f^*A\rightarrow f^*B$$
	given by $(f^*\Phi)(\psi)(y)=\Phi_{f(y)}(\psi(y))$ is a $C_0(Y)$-linear homomorphism.
	Moreover, the pullback construction is functorial meaning if $\Psi:B\rightarrow C$ is another $C_0(X)$-linear $*$-homo\-morphism into a $C_0(X)$-algebra $C$ then $f^*\Psi\circ f^*\Phi=f^*(\Psi\circ \Phi)$.
\end{lemma}

\paragraph{Push forward.}
Let $A$ be a $C_0(X)$-algebra and $f:X\rightarrow Y$ a continuous map. Then we can turn $A$ into a $C_0(Y)$-algebra as follows:
Since the action $\Phi:C_0(X)\rightarrow Z(M(A))$ is non-degenerate there exists a unique extension $$\tilde{\Phi}:C_b(X)\cong M(C_0(X))\rightarrow M(A)$$ to the bounded continuous functions on $X$. We need the following 
\begin{lemma}
The image of $\tilde{\Phi}$ is contained in the centre $Z(M(A))$ of $M(A)$.
\end{lemma}
\begin{proof}
Recall from \cite[Lemma~8.3]{Williams}, that it suffices to show, that $\tilde{\Phi}(f)ab=a\tilde{\Phi}(f)b$ for all $a,b\in A$ and $f\in C_b(X)$. Furthermore, since $\Phi$ is non-degenerate, it suffices to check this for elements of the form $\tilde{a}=\Phi(g)a\in A$ with $g\in C_0(X)$ and $a\in A$.
So let $g\in C_0(X)$ and $a,b\in A$ be given. Then we have
$$
\tilde{\Phi}(f)\tilde{a}b
 =\Phi(fg)ab
 =a\Phi(fg)b
 =a\tilde{\Phi}(f)\Phi(g)b
 =a\Phi(g)\tilde{\Phi}(f)b
=\Phi(g)a\tilde{\Phi}(f)b
 =\tilde{a}\tilde{\Phi}(f)b,
$$
and the proof is complete.
\end{proof}
If we now consider the induced homomorphism $f^*:C_0(Y)\rightarrow C_b(X)$ we can just compose it with $\tilde{\Phi}$ to obtain a homomorphism $C_0(Y)\rightarrow Z(M(A))$.
In other words: For all $\varphi\in C_0(Y)$ and $a\in A$ we can define $\varphi\cdot a:=\tilde{\Phi}(\varphi\circ f)a$. In order to see that this indeed turns $A$ into a $C_0(Y)$-algebra we just need to check the non-degeneracy condition, which is the content of the following Proposition.
\begin{prop}\label{Prop:Pushforward} Let $A$ be a $C_0(X)$-algebra and $f:X\rightarrow Y$ be a continuous map. Then $A$ is a $C_0(Y)$-algebra with respect to the homomorphism $\tilde{\Phi}\circ f^*:C_0(Y)\rightarrow Z(M(A))$.
\end{prop}
\begin{proof}
We only need to check, that $\tilde{\Phi}\circ f^*$ is non-degenerate.
First observe, that $f^*$ is non-degenerate in the sense that $f^*(C_0(Y))C_0(X)$ is dense in $C_0(X)$. This follows easily from the Stone-Weierstrass theorem since if $x\neq y\in X$ then we can choose a function $\varphi\in C_0(X)$ such that $\varphi(x)=1$ and $\varphi(y)=0$. Furthermore let $\psi\in C_0(Y)$ be a function such that $\psi(f(x))=1$. Then $(f^*(\psi)\varphi)(x)=\psi(f(x))\varphi(x)=1\neq 0=\psi(f(y))\varphi(y)=(f^*(\psi)\varphi)(y)$.

If $a\in A$ and $\varepsilon>0$ are given, there exist $\varphi\in C_0(X)$ and $b\in A$ such that $\norm{\tilde{\Phi}(\varphi)b-a}<\frac{\varepsilon}{2}$ since $\Phi$ is non-degenerate. Since $f^*(C_0(Y))C_0(X)$ is dense in $C_0(X)$ we can find functions $g\in C_0(Y)$ and $h\in C_0(X)$ such that $\norm{f^*(g)h-\varphi}<\frac{\varepsilon}{2\norm{b}}$.
Consequently, we get that \begin{align*}
\norm{\tilde{\Phi}(f^*(g))\Phi(h)b-a} & =\norm{\tilde{\Phi}(f^*(g)h)b-a}\\
& \leq \norm{\tilde{\Phi}(f^*(g)h)b-\tilde{\Phi}(\varphi)b}+\norm{\tilde{\Phi}(\varphi)b-a}\\
& \leq \norm{f^*(g)h-\varphi}\norm{b}+\norm{\tilde{\Phi}(\varphi)b-a}<\varepsilon
\end{align*}
\end{proof}

It is important to note, that this construction (in contrast to the pullback) does not change the $\mathrm{C}^*$-algebra itself, but just the associated bundle structure.
We will sometimes write $f_*A$ for the pushforward of $A$ along $f$.
One has the following general description of the fibres:
\begin{prop}\label{Prop:Fibres of Pushforward}
Let $A$ be a $C_0(X)$-algebra and $f:X\rightarrow Y$ be a continuous map between locally compact Hausdorff spaces. For $y\in Y$ let $X_y:=f^{-1}(\lbrace y\rbrace)$. Then, viewing $A$ as a $C_0(Y)$-algebra via pushing forward along $f$, there is an isomorphism \[A_y\rightarrow \Gamma_0(X_y,\mathcal{A}_{\mid X_y}).\]
\end{prop}
\begin{proof}
Identify $A$ with the section algebra $\Gamma_0(X,\mathcal{A})$ and consider the restriction homomorphism
\[res:\Gamma_0(X,\mathcal{A})\rightarrow \Gamma_0(X_y,\mathcal{A}_{\mid X_y}).\]
We will show, that this homomorphism factors through the desired isomorphism.
First of all $ker(res)$ can be identified with the ideal $I_y$:
For all $x\in X_y$, $\varphi\in C_0(Y\setminus\lbrace y\rbrace)$ and $a\in A$ we clearly have $(\varphi\cdot a)(x)=\varphi(f(x))a(x)=\varphi(y)a(x)=0$ and thus $I_y\subseteq ker(res)$. If conversely $a\in ker(res)$ and $\varepsilon>0$ is given then $K:=\lbrace x\in X\mid \norm{a(x)}\geq \varepsilon\rbrace$ is compact. By continuity $f(K)$ is also compact. Since clearly $y\notin f(K)$ there is a function $\varphi\in C_0(Y)$ with $0\leq \varphi\leq 1$ such that $\varphi=1$ on $f(K)$ and $\varphi(y)=0$.
Then $\varphi\cdot a\in I_y$. For $x\in K$ we have $\norm{a(x)-(\varphi\cdot a)(x)}=\norm{a(x)-\varphi(f(x))a(x)}=0$ and for $x\notin K$ we have $\norm{a(x)-\varphi(f(x))a(x)}=\betrag{1-\varphi(f(x))}\norm{a(x)}<\varepsilon$ by construction. Thus, we can conclude $\norm{a-\varphi\cdot a}=\sup\limits_{x\in X}\norm{a(x)-\varphi(f(x))a}<\varepsilon$ and hence $a\in I_y$.
Surjectivity follows from another easy application of Proposition \ref{Prop:DensityCriterionC(X)-algebras}.
\end{proof}
The following describes the interplay of the pushforward and the pullback construction:
\begin{prop}\label{Prop:PullbackAndPushforward}
	Let $f:Y\rightarrow X$ and $g:Z\rightarrow X$ be continuous maps. Consider also the pullback space $Y\times_X Z=\lbrace (y,z)\in Y\times Z\mid f(y)=g(z)\rbrace$ with the canonical projection maps $\pi_Y:Y\times_X Z\rightarrow Y$ and $\pi_Z:Y\times_X Z\rightarrow Z$. Suppose $A$ is a $C_0(Z)$-algebra. Then $f^*(g_*A)$ is canonically isomorphic to $(\pi_Y)_*(\pi_Z^*A)$ as $C_0(Y)$-algebras.
\end{prop}
\begin{proof}
	We will define a map $$\Phi: f^*(g_*A)\rightarrow (\pi_Y)_*(\pi_Z^*A).$$
	Note first, that for $y\in Y$ the fibres of each of these $C_0(Y)$-algebras are given by $$f^*(g_*A)_y=(g_*A)_{f(y)}=\Gamma_0(Z_{f(y)},\mathcal{A}_{\mid Z_{f(y)}}),\text{ and }$$
	$$(\pi_Y)_*(\pi_Z^*A)_y=\Gamma_0((Y\times_X Z)_y, \pi_Z^*\mathcal{A}_{\mid (Y\times_X Z)_y}).$$
	For $\varphi\in f^*(g_*A)=\Gamma_0(Y,f^*(g_*\mathcal{A}))$ define $\left(\Phi(\varphi)(y)\right)(y,z)=(\varphi(y))(z)$.
	It is straightforward to check, that $\Phi$ is an isometric, $C_0(Y)$-linear $*$-homo\-morphism. Surjectivity however is obvious for the homomorphism $\Phi_y$ at the level of each fibre, hence an application of Lemma \ref{Lem:IsomorphismCriteriumForC(X)-linearHomomorphisms} finishes the proof. 
\end{proof}

\paragraph{Tensor products.}
Let $\otimes_{max}$ denote the maximal tensor product of $C^*$-Algebras. If $A$ and $B$ are $C^*$-algebras then the canonical embeddings $i_A:A\rightarrow M(A\otimes_{max}B)$ and $i_B:B\rightarrow M(A\otimes_{max}B)$ extend to commuting embeddings $M(A)\rightarrow M(A\otimes_{max}B)$ and $M(B)\rightarrow M(A\otimes_{max}B)$. One easily checks, that these embeddings take central multipliers to central multipliers.
By the universal property of the maximal tensor product, there is a homomorphism $$ZM(A)\otimes_{max}ZM(B)\rightarrow ZM(A\otimes_{max}B),$$
characterized by the formula $(m\otimes n)(a\otimes b)=ma\otimes nb$.
\begin{prop}\cite[Corollaire~3.16]{MR1395009} Let $A$ be a $C_0(X)$-algebra and $B$ a $C_0(Y)$-algebra with structure homomorphisms $\Phi:A\rightarrow ZM(A)$ and $\Psi:C_0(Y)\rightarrow ZM(B)$. Then the composition $$C_0(X)\otimes C_0(Y)\stackrel{\Phi\otimes\Psi}{\rightarrow} ZM(A)\otimes_{max} ZM(B)\rightarrow ZM(A\otimes_{max} B)$$ is a non-degenerate $*$-homomorphism. Hence $A\otimes_{max} B$ is a $C_0(X\times Y)$-algebra.
Furthermore, there are canonical isomorphisms $$(A\otimes_{max}B)_{(x,y)}\cong A_x\otimes_{max} B_y.$$
\end{prop}

If $A$ and $B$ are two $C_0(X)$-algebras, we would like to consider a notion of tensor product, which is again a $C_0(X)$-algebra. To this end consider the diagonal map $\Delta:X\rightarrow X\times X$ and define the \textit{(maximal) balanced tensor product} of $A$ and $B$ over $X$ to be the pullback $A\otimes_X^{max}B:=\Delta^*(A\otimes_{max}B)$.
Note that there is a canonical isomorphism $A\otimes_X^{max}B\cong A\otimes_{max}B/I_\Delta$, where $I_\Delta=\overline{C_0((X\times X)\setminus im(\Delta))A\otimes_{max}B}$ is the ideal in $A\otimes_{max} B$ corresponding to the closed subset $im(\Delta)\subseteq X\times X$.

Next we want to remind the reader of the definition of actions of groupoids on $\mathrm{C}^*$-algebras.
For a more detailed exposition see \cite{MR2547343}.

A \textit{groupoid dynamical system} $(A,G,\alpha)$ consists of a locally compact Hausdorff groupoid $G$, a $C_0(G^{(0)})$-algebra $A$ and a family $(\alpha_g)_{g\in G}$ of $*$-isomorphisms $\alpha_g:A_{d(g)}\rightarrow A_{r(g)}$ such that $\alpha_{gh}=\alpha_g\circ \alpha_h$ for all $(g,h)\in G^{(2)}$ and such that $g\cdot a:=\alpha_g(a)$ defines a continuous action of $G$ on the upper-semicontinuous bundle $\mathcal{A}$ associated to $A$.

It follows easily from the definition that for all $u\in G^{(0)}$ we have $\alpha_u=id_{A_u}$ and for all $g\in G$ we have $\alpha_{g^{-1}}=\alpha_g^{-1}$.

We will often omit the action $\alpha$ in our notation and just say that $A$ is a $G$-algebra.
The following lemma gives us a better handle on the continuity assumption for the action:
\begin{lemma}\label{Lem:EasierCharacterizationOfContinuiutyOfAction}
Let $A$ be a $C_0(G^{(0)})$-algebra and $\alpha=(\alpha_g)_{g\in G}$ be a family of $*$-isomorphisms $\alpha_g:A_{d(g)}\rightarrow A_{r(g)}$, such that $\alpha_{gh}=\alpha_g\circ \alpha_h$ for all $(g,h)\in G^{(2)}$. Then $(A,G,\alpha)$ is a groupoid dynamical system, if and only if for every $a\in A$ the map $g\mapsto \alpha_g(a(d(g)))$ is a continuous section $G\rightarrow r^*\mathcal{A}$.
\end{lemma}
\begin{proof}
If $(A,G,\alpha)$ is a groupoid dynamcial system, it is clear that the mapping $g\mapsto \alpha_g(a(d(g)))$ is continuous.

For the converse we need to show, that if $(g_\lambda,b_\lambda)_\lambda$ is a net in $G\ast \mathcal{A}$ converging to some element $(g,b)$, then $\alpha_{g_\lambda}(b_\lambda)\rightarrow \alpha_g(b)$ in $r^*\mathcal{A}$.
We want to apply \cite[Proposition~C.20]{Williams} again.
Choose $a\in A$ with $a(d(g))=b$.
If we put $u_\lambda:=\alpha_{g_\lambda}(a(d(g_\lambda)))$ and $u:=\alpha_g(a(d(g)))=\alpha_g(b)$, then property (a) holds by our assumption and (b) and (c) are automatically satisfied. It remains to check (d), i.e. that for all $\varepsilon>0$ we eventually have $\norm{\alpha_{g_\lambda}(b_\lambda)-u_\lambda}<\varepsilon$.
But $\norm{\alpha_{g_\lambda}(b_\lambda)-u_\lambda}=\norm{b_\lambda-a(d(g_\lambda))}$ and since $b_\lambda\rightarrow b$ we have that $b_\lambda$ will eventually be contained in the basic open neighbourhood $W(a,G^{(0)},\varepsilon)$ of $b$, which finishes the proof of (d).
\end{proof}

We will now study several constructions of groupoid dynamical systems.
\paragraph{Pullbacks.} Suppose that $\Phi:H\rightarrow G$ is a groupoid homomorphism. Let $\Phi_0:H^{(0)}\rightarrow G^{(0)}$ be the corresponding map between the unit spaces.
If $(A, G,\alpha)$ is a groupoid dynamical system, we obtain an isomorphism of $C_0(G)$-algebras:
$$\Phi^*\alpha:\Phi^*(d_G^*A)\rightarrow \Phi^*(r_G^*A)$$
by Lemma \ref{Lem:PullbackOfHomomorphisms}.
Now using the identifications
$$d_H^*(\Phi_0^*A)=(\Phi_0\circ d_H)^*A=(d_G\circ \Phi)^*A=\Phi^*(d_G^*A)$$ and similarly $$r_H^*(\Phi_0^*A)=\Phi^*(r_G^*A),$$ we obtain a $C_0(H)$-linear $\ast$-isomorphism
$$d_H^*(\Phi_0^*A)\rightarrow r_H^*(\Phi_0^*A),$$
which defines an action of $H$ on $\Phi_0^*A$ by \cite[Lemma~4.3]{MR2547343}.

A particular instance of this is given by the inclusion of a closed subgroupoid. Let $H$ be a closed subgroupoid of $G$ and $\iota:H\hookrightarrow G$ the inclusion map. If $A$ is a $G$-algebra we write $A_{\mid H}:=\iota_0^*A$ and the action of $H$ on $A_{\mid H}$ is just the restriction of the action of $G$ on $A$.

\paragraph{Pushforward.} Suppose $X$ is a (left) $G$-space with anchor map $p:X\rightarrow G^{(0)}$ and $(A,G\ltimes X,\alpha)$ is a groupoid dynamical system. Then pushing forward along $p$ we can also view $A$ as a $C_0(G^{(0)})$-algebra. Recall that $A_u$ is canonically identified with $\Gamma_0(p^{-1}(u),\mathcal{A})$. We can define a family $(\beta_g)_g$ of $*$-homo\-morphisms $\beta_g:A_{d(g)}\rightarrow A_{r(g)}$ by
$$\beta_g(f)(x)=\alpha_{(g,x)}(f(g^{-1}x)).$$
\begin{prop}\label{Prop:PushforwardAction}
	The tripel $(A,G,\beta)$ is a groupoid dynamical system.
\end{prop}
\begin{proof}
	First of all $\beta_g: A_{d(g)}\rightarrow A_{r(g)}$ is an isomorphism, as one easily computes that $\beta_g^{-1}=\beta_{g^{-1}}$ is an inverse. A similar computation yields that $\beta_{gh}=\beta_g\circ \beta_h$ for all $(g,h)\in G^{(2)}$. It remains to check, that the action of $G$ on the bundle $p_*\mathcal{A}$ is continuous. Recall that the action of $G\ltimes X$ is implemented by an isomorphism $\alpha:D^*A\rightarrow R^*A$, where $D,R:G\ltimes X\rightarrow X$ denote the domain and range maps respectively.
	Using the pushforward construction along the projection $\pi:G\ltimes X\rightarrow G$ onto the first factor, we obtain a $*$-isomorphism $$\pi_*\alpha:\pi_*(D^*A)\rightarrow \pi_*(R^*A).$$
	Now an application of Proposition \ref{Prop:PullbackAndPushforward} provides the identifications $\pi_*(D^*A)\cong d^*(p_*A)$ and $\pi_*(R^*A)\cong r^*(p_*A)$. A quick computation reveals that under these identifications we have $(\pi_*\alpha)_g=\beta_g$.
\end{proof}

\paragraph{Tensor products.}
Given groupoid dynamical systems $(A,G,\alpha)$ and $(B,G,\beta)$ we want to define the \textit{diagonal action} of $G$ on the balanced tensor product $A\otimes_{G^{(0)}}^{max} B$, following \cite{LeGall99}. Using the canonical identifications of $C_0(G)$-algebras $d^*(A\otimes_{G^{(0)}}^{max} B)=d^*A\otimes_{G}^{max} d^*B$ and $r^*(A\otimes_{G^{(0)}}^{max} B)=r^*A\otimes_{G}^{max} r^*B$ the desired action is defined by the isomorphism
$$\alpha\otimes \beta: d^*A\otimes_{G}^{max} d^*B\rightarrow r^*A\otimes_{G}^{max} r^*B.$$
For $g\in G$ we have $(\alpha\otimes \beta)_g=\alpha_g\otimes\beta_g$.

\paragraph{Crossed products.}
In this short paragraph we remind the reader of the definition of reduced crossed products of $\mathrm{C}^*$-algebras by étale groupoids roughly following \cite{MR1900993}.
Let $G$ be an étale groupoid and $(A,G,\alpha)$ a groupoid dynamical system. Consider the complex vector space $\Gamma_c(G,r^*\mathcal{A})$. It carries a canonical $*$-algebra structure with respect to the following operations:
$$(f_1\ast f_2)(g)=\sum\limits_{h\in G^{r(g)}} f_1(h)\alpha_h(f_2(h^{-1}g))$$
and
$$f^*(g)=\alpha_g(f(g^{-1})^*).$$
See for example \cite[Proposition~4.4]{MR2547343} for a proof of this fact.
For $u\in G^{(0)}$ consider the Hilbert $A_u$-module $\ell^2(G^u,A_u)$. It is the completion of the space of finitely supported $A_u$-valued functions on $G^u$, with respect to the inner product 
$$\lk \xi,\eta\rk=\sum\limits_{h\in G^u}\xi(h)^*\eta(h).$$
We can then define a $*$-representation $\pi_u:\Gamma_c(G,r^*\mathcal{A})\rightarrow \mathcal{L}(\ell^2(G^u,A_u))$ by
$$\pi_u(f)\xi(g)=\sum\limits_{h\in G^u}\alpha_g(f(g^{-1}h))\xi(h).$$
Using this family of representations, we can define a $\mathrm{C}^*$-norm on the convolution algebra $\Gamma_c(G,r^*\mathcal{A})$ by
$$\norm{f}_r:=\sup\limits_{u\in G^{(0)}}\norm{\pi_u(f)}.$$
The reduced crossed product $A\rtimes_r G$ is defined to be the completion of $\Gamma_c(G,r^*\mathcal{A})$ with respect to $\norm{\cdot}_r$.

We will now define and study a noncommutative analogue of the construction of the induced space, that we studied at the end of section 2. The definition is well-known in the group case and has appeared in the literature before also in the groupoid setting (see for example \cite{MR2928324}), but since we could not find a study of the basic properties, we chose to give a detailed exposition here. Most of our treatment follows ideas quite similar to the group case, which are presented nicely in \cite{Morita}.

Let $(A,G,\alpha)$ be a groupoid dynamical system and $X$ a right $G$-space with anchor map $p:X\rightarrow G^{(0)}$.
Consider the upper-semi\-continuous $\mathrm{C}^*$-bundle $\mathcal{A}$ over $G^{(0)}$ associated to $A$. Form the pull-back $p^*(\mathcal{A})$ to obtain an upper-semicontinuous $\mathrm{C}^*$-bundle over $X$. Then define $Ind_G^X(A)$ to be the set of all bounded continuous sections $f\in \Gamma_b(X,p^*(\mathcal{A}))$, such that
\begin{enumerate}
\item for all $x\in X$ and $g\in G_{p(x)}$ we have $\alpha_g(f(x))=f(xg^{-1})$, and
\item the map $[xG\mapsto\norm{f(x)}]$ vanishes at infinity.
\end{enumerate}
As $Ind_G^X(A)$ is a closed $*$-subalgebra of $\Gamma_b(X,p^*(\mathcal{A}))$, it is a $\mathrm{C}^*$-algebra. If the action of $G$ on $X$ is proper, $Ind_G^X(A)$ carries more structure:
\begin{prop}\label{Prop:Induced Algebra is a field}
Let $(A,G,\alpha)$ be a groupoid dynamical system and $X$ a proper right $G$-space. Then $Ind_G^X (A)$ is a $C_0(X/G)$-algebra with respect to the action
$$(\varphi\cdot f)(x)=\varphi(xG)f(x),$$
for $\varphi\in C_0(X/G)$ and $f\in Ind_G^X (A)$.
\end{prop}
\begin{proof}
First recall that the orbit space for a proper action is a locally compact Hausdorff space, so that our at least claim makes sense. Secondly, using \cite[Lemma~8.3]{Williams}, we can easily check, that the formula above defines an action of $C_0(X/G)$ as central multipliers: For $f,g\in Ind_G^X(A)$ and $\varphi\in C_0(X/G)$ we have
$$\varphi(ff')(x)=\varphi(xG)f(x)f'(x)=f(x)\varphi(xG)f'(x)=f(\varphi f')(x).$$
It remains to check the non-degeneracy of the action. So let $f\in Ind_X^GA$ and $\varepsilon>0$ arbitrary. By definition of the induced algebra there exists a compact subset $K\subseteq X/G$ such that $\norm{f(x)}<\varepsilon$ for all $xG\not\in K$. Choose a function $\varphi\in C_0(X/G)$ with $0\leq \varphi\leq 1$ such that $\varphi(xG)=1$ for all $xG\in K$. Then we have $\norm{\varphi f-f}<\varepsilon$.
\end{proof}

In what follows we want to identify the fibres of $Ind_G^X (A)$ with respect to this $C_0(X/G)$-algebra structure. 
\begin{lemma} \label{r!}
Let $G$ be locally compact Hausdorff groupoid with Haar system $(\lambda^u)_{u\in G^{(0)}}$ and let $A$ be a $C_0(G^{(0)})$-algebra. Given an element $f\in \Gamma(G,r^*\mathcal{A})$ such that $supp(f)\cap r^{-1}(K)$ is compact for all compact $K\subseteq G^{(0)}$ let
\begin{equation}
\lambda(f)(u):=\int\limits_{G^u} f(g)d\lambda^u(g).
\end{equation}
This defines an element $\lambda(f)\in \Gamma(G^{(0)},\mathcal{A})$.
\end{lemma}
\begin{proof}
	If $f$ is compactly supported this is well-known (see \cite[Proposition~3.53]{Goehle} for a detailed proof). In the general setting we can proceed as in the scalar case presented in Lemma \ref{Lem:proper support}.
\end{proof}

	The next lemma is a groupoid analogue of \cite[Lemma~6.17]{Morita}, which tells us that there are lots of non-trivial elements in $Ind_G^X(A)$.
	\begin{lemma}
		Let $G$ be a locally compact Hausdorff groupoid with Haar system $(\lambda^u)_{u\in G^{(0)}}$. If $(A,G,\alpha)$ is a groupoid dynamical system and $X$ a proper, right $G$-space with anchor-map $p:X\rightarrow G^{(0)}$, then for every $\varphi\in C_c(X)$ and $a\in A$ the formula
		\[\varphi \diamond a(x):= \int\limits_{G^{p(x)}} \varphi(xg)\alpha_g(a(d(g)))d\lambda^{p(x)}(g)\]
		gives a well-defined element $\varphi\diamond a\in Ind_G^X(A)$.
	\end{lemma}
	\begin{proof}
		Since the action of $G$ on $X$ is proper, the set $\lbrace g\in G^{p(x)}\mid x\cdot g\in supp(\varphi)\rbrace$ is compact for each fixed $x\in X$. Thus, the integrand is an element in $C_c(G^{p(x)},A_{p(x)})$ and we can form the integral.
		For each $t\in G_{p(x)}$ we have
		\begin{align*}
		\varphi\diamond a(xt^{-1}) & \ \ =\int\limits_{G^{p(xt^{-1})}} \varphi(xt^{-1}g)\alpha_g(a(d(g)))d\lambda^{p(xt^{-1})}(g)\\
		& \stackrel{g\mapsto tg}{=} \int\limits_{G^{p(x)}} \varphi(xg)\alpha_{tg}(a(d(g)))d\lambda^{p(x)}(g)\\
		& \ \ =\alpha_t \left(  \int\limits_{G^{p(x)}} \varphi(xg)\alpha_{g}(a(d(g)))d\lambda^{p(x)}(g)\right)\\ 
		& \ \ = \alpha_t(\varphi\diamond a(x))
		\end{align*}
		Furthermore $\varphi\diamond a$ is bounded. To see this note that the set $S:=\lbrace g\in G\mid supp(\varphi)\cdot g\cap supp(\varphi)\neq \emptyset\rbrace$ is compact. From Lemma \ref{measure of compact set is bounded} we know that there is a $C>0$ such that $\lambda^{p(x)}(S)<C$ for all $x\in X$. Then we have $\norm{\varphi\diamond a(x)}\leq \int\limits_{G^{p(x)}} \betrag{\varphi(xg)}d\lambda^{p(x)}(g) \norm{a}\leq \lambda^{p(x)}(S)\norm{\varphi}\norm{a}\leq C\norm{\varphi}\norm{a}$. 
		We want to see that $\varphi\diamond a$ is continuous.
		Note that $(y,g)\mapsto \varphi(y)\alpha_g(a(d(g)))$ is an element in $\Gamma(X\rtimes G, r_{X\rtimes G}^*(p^*\mathcal{A}))$ with proper support and thus by Lemma \ref{r!} the map
		\begin{align*} x\mapsto& \int\limits_{(X\rtimes G)^x}\varphi(y)\alpha_g(a(d(g)))d(\delta_x\otimes\lambda^{p(x)})(y,g)\\
		&=\int\limits_{G^{p(x)}}\varphi(xg)\alpha_g(a(d(g)))d\lambda^{p(x)}(g)
		\end{align*} is continuous.
		Finally, note that the support of the map $xG\mapsto \norm{\varphi\diamond a(x)}$ is contained in the image of the compact set $supp(\varphi)\subseteq X$ under the quotient map, and hence certainly vanishes at infinity.	\end{proof}
	
	We are now ready to identify the fibres. To simplify the notation (and because we are mainly interested in this particular situation) we will now also assume that the action of $G$ on $X$ is \textit{free} in the sense that $xg=x$ implies that $g$ is a unit.
	
	\begin{prop}\label{Prop:FibresOfInducedAlgebra}
		Let $G$ be a locally compact Hausdorff groupoid with Haar system $(\lambda^u)_{u}$. If $(A,G,\alpha)$ is a groupoid dynamical system and $X$ a free and proper, right $G$-space with anchor map $p:X\rightarrow G^{(0)}$, then $Ind_G^X (A)$ is a $C_0(X/G)$-algebra, such that the fibre $(Ind_G^X(A))_{xG}$ over $xG\in X/G$ is canonically isomorphic to $A_{p(x)}$.
	\end{prop}
	\begin{proof} The first part of the assertion has already been dealt with in Proposition~\ref{Prop:Induced Algebra is a field}. It remains to identify the fibres.
		For $x\in X$ consider the evaluation map $$ev_x:Ind_G^X(A)\rightarrow A_{p(x)}.$$ We will show, that the kernel of $ev_x$ coincides with the ideal
		$$I_{xG}=\overline{C_0(X/G\setminus\lbrace xG\rbrace)Ind_G^X(A)}$$ and that $ev_x$ is surjective.
		Let us start with the kernel. If $\varphi\in C_0(X/G\setminus\lbrace xG\rbrace)$ and $f\in Ind_G^X(A)$ we have $ev_x(\varphi\cdot f)=\varphi(xG)f(x)=0$. Thus $I_{xG}\subseteq ker(ev_x)$.
		If conversely $f\in ker(ev_x)$ we have $f(xg)=\alpha_{g^{-1}}(f(x))=0$ for all $g\in G$. Hence $f$ is zero on the whole orbit of $x$. Given $\varepsilon>0$ the set $K:=\lbrace yG\mid \norm{f(y)}\geq \varepsilon\rbrace$ is compact by definition of the induced algebra.
		Since $X/G$ is Hausdorff there exists a $\varphi\in C_c(X/G)$, $0\leq\varphi\leq 1$ such that $\varphi(xG)=0$ and $\varphi=1$ on $K$. One easily checks that $\varphi\cdot f\in I_{xG}$ and $\norm{f-\varphi\cdot f}<\varepsilon$.
		
		To prove surjectivity it suffices to show that $ev_x$ has dense range. So let $a(p(x))\in A_{p(x)}$ and $\varepsilon>0$ be given. Choose a neighbourhood $U$ of $p(x)$ in $G$ such that $\norm{\alpha_g(a(d(g)))-a(p(x))}<\varepsilon$ for all $g\in G^{p(x)}\cap U$. Since the action is free and proper one checks (using Proposition~\ref{Prop:CharacterizationsProperAction}(3)) that $xU$ is open as a subset of $xG$ and hence we can choose $V\subseteq X$ open such that $V\cap xG=xU$. If $\phi\in C_c(X)$ is positive and has support contained in $V$ define $$\varphi(x):=\left(\int\limits_{G^{p(x)}} \phi(xg)d\lambda^{p(x)}(g)\right)^{-1}\phi(x).$$ Then $$\int\limits_{G^{p(x)}} \varphi(xg)d\lambda^{p(x)}(g)=1$$ and we have
		\begin{align*}
		\norm{\varphi\diamond a(x)-a(p(x))} & =\norm{\int\limits_{G^{p(x)}} \varphi(xg)\alpha_g(a(d(g)))dg-
		%	\left(\int\limits_{G^{p(x)}} \varphi(xg)dg\right)
			a(p(x))}\\
		& \leq \int\limits_{G^{p(x)}} \varphi(xg)\norm{\alpha_g(a(d(g)))-a(p(x))}d\lambda^{p(x)}(g)\\
		& <\varepsilon.
		\end{align*}
	\end{proof}\begin{bem}
	Note that it follows from the proof above and Proposition \ref{Prop:DensityCriterionC(X)-algebras} that
	$$ \text{span}\lbrace \varphi\diamond a\mid \varphi\in C_c(X), a\in A\rbrace$$
	is dense in $Ind_G^X A$.
\end{bem}

We will now turn to the situation which is of most interest for our purposes. Let $G$ be a groupoid and $H\subseteq G$ a closed subgroupoid. Set $X:=d^{-1}(H^{(0)})\subseteq G$.
Then $H$ acts from the right on $X$, where the anchor map is the restriction of the domain map to $X$ and the product is just given by multiplication. This action is obviously free and proper since $X\rtimes H$ is a closed subgroupoid of the proper groupoid $G\rtimes G$.
As the restriction of the range map to $X$ is invariant under the $H$-action, it factors through a continuous map $\tilde{r}:X/H\rightarrow G^{(0)}$. This map serves as the anchor map for the canonical action of $G$ on $X/H$ given by multiplication (note that $gx\in X$ for all $g\in G$ and $x\in X$ with $d(g)=r(x)$).

Note that for each $(g,xH)\in G\ltimes X/H$ Proposition \ref{Prop:FibresOfInducedAlgebra} gives us isomorphisms $\widetilde{ev_x}:(Ind_H^X A)_x\rightarrow A_{d(x)}$ and $\widetilde{ev_{g^{-1}x}}:(Ind_H^X A)_{g^{-1}x}\rightarrow A_{d(x)}$. Hence we get an isomorphism
$$\alpha_{(g,xH)}:=\widetilde{ev_x}^{-1}\circ\widetilde{ev_{g^{-1}x}}:(Ind_H^X A)_{g^{-1}xH}\rightarrow (Ind_H^X A)_{xH}$$
Let $\alpha=(\alpha_{(g,xH)})_{(g,xH)\in G\ltimes X/H}$ be the family of all these ismorphisms. We want to see that $(Ind_H^X A,G\ltimes X/H,\alpha)$ is a groupoid dynamical system.
To check continuity of the action we need the following observation:
\begin{lemma}\label{Lemma:Nets Converging to the same point are eventually close}
Let $q:\mathcal{A}\rightarrow X$ be an upper-semicontinuous $C^*$-bundle. Suppose $(a_\lambda)_\lambda$ and $(b_\lambda)_\lambda$ are nets in $\mathcal{A}$ such that $q(a_\lambda)=q(b_\lambda)$ and $\lim_\lambda a_\lambda =a=\lim_\lambda b_\lambda$. Then 
$$\lim_\lambda \norm{a_\lambda-b_\lambda}=0.$$
\end{lemma}
\begin{proof}
Let $\varepsilon>0$ be given. Choose $f\in \Gamma_0(X,\mathcal{A})$ such that $f(q(a))=a$. Then $a$ is contained in the basic open set 
$$W(f,\frac{\varepsilon}{2})=\lbrace b\in \mathcal{A}\mid \norm{b-f(q(b))}<\frac{\varepsilon}{2}\rbrace.$$
By assumption, for large $\lambda$ we have $a_\lambda, b_\lambda\in W(f,\frac{\varepsilon}{2})$. Consequently, we eventually have
$$\norm{a_\lambda -b_\lambda}\leq \norm{a_\lambda-f(q(a_\lambda))}+\norm{f(q(b_\lambda))-b_\lambda}<\varepsilon.$$
\end{proof}
\begin{prop}
The triple $(Ind_H^X A,G\ltimes X/H,\alpha)$ is a groupoid dynamical system.
\end{prop}
\begin{proof}
Let us first check that $\alpha$ is compatible with the groupoid structure. We compute
\begin{align*}
\alpha_{(g_1,xH)}\circ\alpha_{(g_2,g_1^{-1}xH)}&=\widetilde{ev_x}^{-1}\circ\widetilde{ev_{g_1^{-1}x}}\circ\widetilde{ev_{g_1^{-1}x}}^{-1}\circ\widetilde{ev_{g_2^{-1}g_1^{-1}x}}\\
& =\widetilde{ev_x}^{-1}\circ\widetilde{ev_{(g_1g_2)^{-1}x}}\\
&=\alpha_{(g_1g_2,xH)}
\end{align*}
Next, we have to check continuity. By Lemma \ref{Lem:EasierCharacterizationOfContinuiutyOfAction}, it is enough to check, that for any net $(g_\lambda,x_\lambda H)_\lambda$ in $G\ltimes X/H$ with $(g_\lambda,x_\lambda H)\rightarrow (g,xH)\in G\ltimes X/H$ and every $f\in Ind_H^X A$ we have
$$\alpha_{(g_\lambda,x_\lambda H)}(f+I_{g_\lambda^{-1}x_\lambda H})\rightarrow \alpha_{(g,xH)}(f+I_{g^{-1}xH})$$
By definition, we have $\alpha_{(g,xH)}(f+I_{g^{-1}xH})=\widetilde{ev_x}^{-1}(f(g^{-1}x))$.
To achieve a contradiction, suppose that the net $\widetilde{ev_{x_\lambda}}^{-1}(f(g_\lambda^{-1}x_\lambda))$ does not converge to $\widetilde{ev_x}^{-1}(f(g^{-1}x))$. Then, by definition of the topology on the bundle associated to the $C_0(X/H)$-algebra $Ind_H^X A$, there exists $f'\in Ind_H^X A$ such that $f'(x)=f(g^{-1}x)$ and $\varepsilon>0$, such that after passing to a suitable subnet and relabelling, we can assume for all $\lambda$:
$$\norm{f(g_\lambda^{-1}x_\lambda)-f'(x_\lambda)}=\norm{\widetilde{ev_{x_\lambda}}^{-1}(f(g_\lambda^{-1}x_\lambda))-f'+I_{x_\lambda H}}\geq \varepsilon$$
After passing to another subnet (and relabelling), we may also assume that $x_\lambda\rightarrow x$ by \cite[Proposition~1.15]{Williams}.
But then, by continuity of $f$ and $f'$ we have $f(g_\lambda^{-1}x_\lambda)\rightarrow f(g^{-1}x)=f'(x)\leftarrow f'(x_\lambda)$. Hence Lemma \ref{Lemma:Nets Converging to the same point are eventually close} implies, that
$$\norm{f(g_\lambda^{-1}x_\lambda)-f'(x_\lambda)}\rightarrow 0,$$
a contradiction.
\end{proof}

\begin{bem}\label{Remark:Equivalence and Induction}
	The dynamical system $(Ind_H^X A,G\ltimes X/H,\alpha)$ can also be obtained using the construction of a pullback along an equivalence of groupoids in the sense of \cite{LeGall99}. Given a closed subgroupoid $H\subseteq G$ the space $X:=d^{-1}(H^{(0)})\subseteq G$ as defined above implements a $G\ltimes X/H - H$-equivalence. One can show that $Ind_H^X A$ and the pullback $X^*(A)$ are isomorphic as $G\ltimes X/H$-algebras.
\end{bem}

 If $A$ is an $H$-algebra we can use the pushforward construction along $\tilde{r}$ to turn $Ind_H^X A$ into a $C_0(G^{(0)})$-algebra. Concretely, for $\varphi\in C_0(G^{(0)})$ and $f\in Ind_H^X A$ this action is given by $$(\varphi\cdot f)(x)=\varphi(r(x))f(x).$$
Let us also identify the fibres of $Ind_H^X A$ with respect to this $C_0(G^{(0)})$-action.
\begin{lemma}
	In the above situation the fibre $(Ind_H^X A)_u$ of $Ind_H^X A$ over $u\in G^{(0)}$ is canonically isomorphic to the algebra $Ind_H^{X^u} A$.
\end{lemma}
\begin{proof}
	Consider the restriction homomorphism
	\[\textsf{res}: Ind_H^X A\rightarrow Ind_H^{X^u} A.\]
	The kernel of $\textsf{res}$ can be identified with $I_u=\overline{C_0(G^{(0)}\setminus \lbrace u\rbrace)Ind_H^X A}$ as follows:
	Let $\varphi\in C_0(G^{(0)}\setminus \lbrace u\rbrace)$ and $f\in Ind_H^X A$. Then for all $x\in X^u$ we clearly have $(\varphi\cdot f)(x)=\varphi(r(x))f(x)=\varphi(u)f(x)=0$. And thus $I_u\subseteq ker(\textsf{res})$.
	For the converse inclusion let $f\in Ind_H^X A$ such that $\textsf{res}(f)=0$. From the definition of $Ind_H^X A$ we know that for any $\varepsilon>0$ the set $K=\lbrace xH\in X/H\mid \norm{f(x)}\geq \varepsilon\rbrace$ is compact. Since $\tilde{r}$ is continuous $\tilde{r}(K)$ is also compact. Since $u\notin \tilde{r}(K)$ we can find a function $\varphi\in C_c(G^{(0)})$ such that $0\leq\varphi\leq 1$, $\varphi\equiv 1$ on $\tilde{r}(K)$ and $\varphi(u)=0$.
	Then clearly $\varphi\cdot f\in I_u$ and we have $\norm{f-\varphi\cdot f}< \varepsilon$ since if $xH\in K$, then $r(x)=\tilde{r}(xH)\in \tilde{r}(K)$ and $\norm{f(x)-\varphi(r(x))f(x)}=\norm{f(x)-f(x)}=0$ and if $xH\notin K$ then $\norm{f(x)-\varphi(r(x))f(x)}=\betrag{1-\varphi(r(x))}\norm{f(x)}<\varepsilon$.
	Thus, we have $f\in I_u$.
	
	To finish the proof we need to show that $\textsf{res}$ is surjective. To this end it is enough to show that $im(\textsf{res})$ is dense in $Ind_H^{X^u} A$. It is clear that $im(\textsf{res})$ is a linear subspace in $Ind_H^{X^u} A$. Moreover, it is closed under the $C_0(X^u/H)$-action since if $\varphi\in C_0(X^u/H)$ and $f\in im(\textsf{res})$ then we can identify $X^u/H$ with the closed subspace $\tilde{r}^{-1}(\lbrace u\rbrace)\subseteq X/H$ and thus find an element $\tilde{\varphi}$ such that $\tilde{\varphi}_{\mid X^u/H}=\varphi$. If $\tilde{f}$ with $res(\tilde{f})=f$ then clearly $\varphi\cdot f=\textsf{res}(\tilde{\varphi}\cdot\tilde{f})\in im(\textsf{res})$.
	Furthermore, for all $xH\in X^u/H$ we know that $\lbrace \textsf{res}(f)(x)\mid f\in Ind_H^X A\rbrace=ev_x(Ind_H^X A)$ is dense in $A_{d(x)}$ from Proposition \ref{Prop:FibresOfInducedAlgebra}. Since $A_{d(x)}=(Ind_H^{X^u} A)_{xH}$ we can apply Proposition \ref{Prop:DensityCriterionC(X)-algebras} to conclude that $im(\textsf{res})$ is dense in $Ind_H^{X^u} A$ as desired.
\end{proof}

\begin{prop} \label{Prop:Action on Induced Algebra}
	Consider the family of isomorphisms $(\beta_g)_{g\in G}$, where
	$$\beta_g:Ind_H^{X^{d(g)}}\rightarrow Ind_H^{X^{r(g)}},\ \ \beta_g(f)(x)=f(g^{-1}x).$$
	Then $(Ind_H^X A,G,\beta)$ is a groupoid dynamical system.
\end{prop}
\begin{proof}
	Apply Proposition \ref{Prop:PushforwardAction} to $(Ind_H^X A,G\ltimes X/H,\alpha)$.
\end{proof}
	
For later purposes we want to examine what happens, if we restrict our $G$-action on $Ind_H^X A$ to the subgroupoid $H$ again. We have the following result:
\begin{lemma}
The restriction $(Ind_H^X A)_{\mid H}$ of the $G$-algebra $Ind_H^X A$ to the subgroupoid $H$ is isomorphic to the induced algebra $Ind_H^{G'} A$, where $G'=G_{H^{(0)}}^{H^{(0)}}\subseteq X$.
\end{lemma}
\begin{proof}
Recall that $(Ind_H^X A)_{\mid H}$ is defined as the algebra of continuous sections of the bundle $\coprod_{u\in H^{(0)}} Ind_H^{X^u} A$ vanishing at infinity.
Thus, we can define a map $\Phi:(Ind_H^X A)_{\mid H}\rightarrow Ind_H^{G'} A$ by letting 
$\Phi(f)(x)=f(r(x))(x)$. One easily checks that this is a $C_0(H^{(0)})$-linear $*$-homo\-morphism. It is not hard to see that the composition of $\Phi$ followed by the restriction map $Ind_H^{G'} A\rightarrow Ind_H^{X^u} A$ coincides with the evaluation homomorphism $ev_u:(Ind_H^X A)_{\mid H}\rightarrow Ind_H^{X^u} A$. Hence $\Phi$ induces the identity on each fibre, which is an isomorphism. By Lemma \ref{Lem:IsomorphismCriteriumForC(X)-linearHomomorphisms} it follows that $\Phi$ must be an isomorphism itself. Following the construction of the restricted action it is easy to see that $\Phi$ is compatible with the $H$-actions on both sides.
\end{proof}
Earlier we claimed that the process of induction should generalize the construction of induced spaces presented in section 2. The following proposition finally justifies this:
\begin{prop}\label{Prop:InducedAlgebra:CommutativeCase} Let $G$ be a locally compact Hausdorff groupoid and $H\subseteq G$ a closed subgroupoid. If $Y$ is a left $H$-space with anchor map $p:Y\rightarrow H^{(0)}$, then $C_0(Y)$ turns into an $H$-algebra. Consider the right $H$-space $X:=d^{-1}(H^{(0)})$. Then $Ind_H^X(C_0(Y))$ is canonically isomorphic to $C_0(G\times_H Y)$, where $G\times_H Y$ is the classical induced $G$-space.
\end{prop}
\begin{proof}
We want to define a map from $Ind_H^X(C_0(Y))$ to $C_0(G\times_H Y)$. For this let $\mathcal{B}$ denote the upper-semicontinuous $C^*$-bundle associated to the $C_0(H^{(0)})$-algebra $C_0(Y)$. Now let $f\in Ind_H^X(C_0(Y))$ be given. Then for each $x\in X$ we have that $f(x)\in (d_{\mid X}^*(\mathcal{B}))_{x}=\mathcal{B}_{d(x)}=C_0(Y)_{d(x)}=C_0(Y_{d(x)})$ where $Y_{d(x)}=p^{-1}(\lbrace d(x)\rbrace)\subseteq Y$.
Define $\Phi:Ind_H^X(C_0(Y))\rightarrow \ell^\infty(G\times_H Y)$ by $$\Phi(f)([x,y]):=(f(x))(y).$$
We need to see, that this is well-defined. Recall that the left action of $H$ on $G\times_{G^{(0)}} Y$ is given by $h\cdot (x,y):= (xh^{-1},hy)$. Then we have $\Phi(f)([xh^{-1},hy])=(f(xh^{-1}))(hy)=(lt_{h}(f(x)))(hy)=(f(x))(y)$, where $lt:d^*(C_0(Y))\rightarrow r^*(C_0(Y))$ denotes the action of $H$ on $C_0(Y)$ induced from the $H$ action on $Y$.
Let us show that $\Phi$ has image in $C_0(G\times_H Y)$. First consider functions of the form $\varphi\diamond g$ for $\varphi\in C_c(X)$ and $g\in C_c(Y)$.
Let $k:G\times_{G^{(0)}} Y\rightarrow\CC$ be the function $k(x,y)=\varphi(x)g(y)$. Clearly $k$ has compact support. Combining this with the fact that $H$ acts properly on $G\times_{G^{(0)}} Y$ we obtain that the map
$H\ltimes (G\times_{G^{(0)}}Y)\rightarrow\CC$ given by
$(h,x,y)\mapsto k(h^{-1}(x,y))$ is continuous and properly supported.
Thus, the map
$$(x,y)\mapsto \int\limits_{H\ltimes (G\times_{G^{(0)}}Y)^{(x,y)}}k(h^{-1}(x',y'))d\lambda^{d(x)}\otimes \delta_{(x,y)}(h,x',y')$$
is continuous by Lemma \ref{r!}. But the latter integral equals 
$$\int\limits_{H^{d(x)}}\varphi(xh)g(h^{-1}y)d\lambda^{d(x)}(h)=\Phi(\varphi\diamond g)([x,y]).$$
Thus $\Phi(\varphi\diamond g)$ is continuous and compactly supported.
Since the linear span of elements of the form $\varphi\diamond g$ is dense in $Ind_H^X C_0(Y)$ and $\Phi$ is clearly a $*$-homomorphism and isometric, its image is contained in $C_0(G\times_H Y)$.
A quick application of the Stone–Weierstrass theorem gives that $im(\Phi)=C_0(G\times_H Y)$.
%We need to show that $\Phi(f)$ vanishes at infinity. Let $\varepsilon>0$ be given. Then there is a compact subset $K\subseteq X/H$ such that $\norm{f(x)}<\varepsilon$ for all $x\in X$ such that $xH\notin K$. \textbf{Weiter ????}\\
%One easily verifies injectivity of $\Phi$.
%On the other hand if $f\in C_0(G\times_H Y)$ is given then $y\mapsto f([x,y])$ defines an element $\tilde{f}_x$ in $C_0(Y_{d(x)})$ for each $x\in X$. Thus $\tilde{f}(x)=\tilde{f}_x$ defines a well-defined section $\tilde{f}\in \Gamma_b(X,p^*(\mathcal{B}))$. \textbf{(Why continuous ???????)} Furthermore we have $\tilde{f}(xh^{-1})(y)=f([xh^{-1},y])=f([x,h^{-1}y])=\tilde{f}(x)(h^{-1}y)=(lt_h(f(x)))(y)$ for all $y\in Y_{d(xh^{-1})}=Y_{r(h)}$.
\end{proof}

We also have, that the process of induction is compatible with the maximal tensor product in the following sense:

\begin{lemma}\label{Lemma:Induction and Tensor product}
	Let $G$ be a locally compact Hausdorff groupoid and $H\subseteq G$ a proper subgroupoid. If $A$ is an $H$-algebra and $B$ a $G$-algebra we have a canonical isomorphism of $G$-algebras
	$$\Phi: (Ind_H^X A)\otimes_{G^{(0)}}^{max} B\rightarrow Ind_H^X(A\otimes_{H^{(0)}}^{max} B_{\mid H})$$
	satisfying
	$$\Phi(f\otimes b)(g)=f(g)\otimes \beta_{g^{-1}}(b(r(g)))$$
	for all $f\in Ind_H^X A$ and $b\in B$.
\end{lemma}
\begin{proof}
	It is easy to check that $\Phi(f\otimes b)\in Ind_H^X(A\otimes_{H^{(0)}}^{max} B_{\mid H})$. Recall, that we can identify the fibre over $u\in G^{(0)}$ as $((Ind_H^X A)\otimes B)_u\cong Ind_H^{X^u} A\otimes B_u$ and $(Ind_H^X(A\otimes B_{\mid H}))_u\cong Ind _H^{X^u}(A\otimes B_{\mid H})$. Using this identification we get that the image of $\Phi(f\otimes b)$ in the fibre $(Ind_H^X(A\otimes B_{\mid H}))_u$ can be identified with the function $g\mapsto f(g)\otimes \beta_{g^{-1}}(b(u))$.
	Hence we can compute
	\begin{align*}
	\norm{\Phi(f\otimes b)}& = \sup\limits_{u\in G^{(0)}} \norm{\Phi(f\otimes b)(u)}\\
	& = \sup\limits_{u\in G^{(0)}} \sup\limits_{g\in X^u} \norm{f(g)\otimes \beta_{g^{-1}}(b(u))}\\
	& = \sup\limits_{u\in G^{(0)}} \sup\limits_{g\in X^u} \norm{f(g)}\norm{b(u)}\\
	& = \sup\limits_{u\in G^{(0)}} \norm{f_{\mid X^u}}\norm{b(u)}\\
	& = \sup\limits_{u\in G^{(0)}} \norm{f_{\mid X^u}\otimes b(u)}\\
	& =\norm{f\otimes b}\\
	\end{align*}
	Hence $\Phi$ extends to a $C_0(G^{(0)})$-linear $*$-homomorphism. To check it is an isomorphism, it is enough to check that $\Phi$ induces an isomorphism on each fibre.
	Viewing $Ind _H^{X^u}(A\otimes B_{\mid H})$ as a $C_0(X^u/H)$-algebra it is also not hard to show that $im(\Phi_u)$ is a $C_0(X^u/H)$-linear subspace such that for each fixed $g\in X^u$ the set
	$$\lbrace \Phi_u(\xi)(g)\mid \xi\in Ind_H^{X^u} A\otimes B_u\rbrace$$
	is dense in $(Ind _H^{X^u}(A\otimes B_{\mid H}))_{gH}=A_{d(g)}\otimes B_{d(g)}$. Thus, $im(\Phi_u)$ is dense in $Ind _H^{X^u}(A\otimes B_{\mid H})$ by Proposition \ref{Prop:DensityCriterionC(X)-algebras}.
	Conversely, these arguments show that functions of the form $f\boxtimes b\in Ind_H^{X^u}(A\otimes B_H)$ for $f\in Ind_H^{X^u} A$ and $b\in B_u$ defined by $(f\boxtimes b)(g):=f(g)\otimes \beta_{g^{-1}}(b)$ generate a dense subspace of $Ind_H^{X^u}(A\otimes B_H)$. The same computation as above then shows that $f\boxtimes b\mapsto f\otimes b$ defines a bounded linear homomorphism $Ind_H^{X^u}(A\otimes B_H)\rightarrow Ind_H^{X^u}A \otimes B_u$, which is clearly inverse to $\Phi_u$. Consequently, $\Phi_u$ is an isomorphism for all $u\in G^{(0)}$ and hence $\Phi$ must be an isomorphism by Lemma \ref{Lem:IsomorphismCriteriumForC(X)-linearHomomorphisms}.
\end{proof}

\section{Equivariant KK-Theory}
In this section we first review the basic constructions of groupoid equivariant $\KK$-Theory and lift some well-known results from the group case to the realm of groupoids. Our exposition is based on the  work of Le Gall (cf. \cite{LeGall,LeGall99}).
Let us start by reviewing some facts on Hilbert modules over $C_0(X)$-algebras:

Let $A$ be a $C_0(X)$-algebra and $E$ a right Hilbert $A$-module. For $\varphi\in C_0(X)$ we can define an action of $C_0(X)$ on $EA=E$ by adjointable operators by
$$\varphi\cdot (xa):=x(a\varphi)$$
It is straightforward to check that this action actually takes values in the center $Z(L(E))$ of the adjointable operators on $E$. Using the canonical isomorphism $M(K(E))\cong L(E)$ we actually get a $*$-homomorphism $\Phi:C_0(X)\rightarrow Z(M(K(E)))$. For rank-one operators this action is given by $\varphi\cdot\Theta_{x,y}=\Theta_{\varphi x,y}$ (here for $x,y\in E$, $\Theta_{x,y}$ denotes the adjointable operator given by $\Theta_{x,y}(z):=x\lk y,z\rk_A$). It is straightforward to show that $\Phi$ is non-degenerate and hence, that $K(E)$ is a $C_0(X)$-algebra.

Similar to $C_0(X)$-algebras we can also view $E$ as a fibred object in the following way:
For $x\in X$ let $E_x$ be the quotient (as a vector space) of $E$ by the closed subspace $\overline{C_0(X\setminus\lbrace x\rbrace)E}$. Denote the image of an element $e\in E$ under the quotient map on $E_x$ by $e(x)$.
Then we can define an $A_x$-valued inner product on $E_x$ by
\[\lk e(x),e'(x)\rk_{A_x}:=\lk e,e'\rk_A(x).\]
One can show that $E_x$ is complete with respect to the norm induced by this inner product.

\begin{bem}\label{Remark:Fibre of Hilbert module can be defined as tensor product}
	Note that one could also define the fibre $E_x$ as the tensor product $E\otimes_A A_x$ (compare \cite[§4.1]{LeGall99}). The canonical morphism $$E\otimes_A A_x\rightarrow E_x,$$
	sending an elementary tensor $e\otimes a(x)$ to the product $(ea)(x)$, is an isomorphism.
\end{bem}

If $E,F$ are two Hilbert $A$-modules, then every operator $T\in L(E,F)$ is automatically compatible with the $C_0(X)$-structures on $E$ and $F$. Hence $T$ factors through a well-defined operator $T_x\in L(E_x,F_x)$ for every $x\in X$. Using \cite[Lemma~C.11]{Williams} one can show that $\norm{T}=\sup_{x\in X}\norm{T_x}.$
If $T\in K(E)$ is a compact operator, then so is $T_x$ for every $x\in X$. For a rank one operator $\Theta_{e,f}\in K(E)$ this is obvious since $(\Theta_{e,f})_x=\Theta_{e(x),f(x)}$. The general case follows by approximating $T\in K(E)$ by finite linear combinations of rank one operators.
This gives rise to a convenient description of the compact operators on of $E_x$. Indeed, the canonical map $T\mapsto T_x$ factors through an isomorphism $$K(E)_x\cong K(E_x),$$
where $K(E)_x$ denotes the fibre of $K(E)$ over $x$ with respect to the $C_0(X)$-structure described above (see \cite[Proposition~4.2]{LeGall99}).

%We have the following
%\begin{lemma}
%	Let $A$ be a $C_0(X)$-algebra and $E$ a right Hilbert $A$-module.
%	\begin{enumerate}
%		\item For each $e\in E$ the map $x\mapsto \norm{e(x)}$ is upper semicontinuous and vanishes at infinity.
%		\item For each $e\in E$ we have $\norm{e}=\sup\limits_{x\in X} \norm{e(x)}$.
%		\item For $e\in E$ and $f\in C_0(X)$ we have $(f e)(x)=f(x)e(x)$.
%	\end{enumerate}
%\end{lemma}
%\begin{proof}
%	Since $\lk e,e\rk_A\in A$ we have that $x\mapsto \norm{\lk e,e\rk_A(x)}$ is upper semicontinous and vanishes at infinity. Consequently, the mapping $x\mapsto \sqrt{\norm{\lk e,e\rk_A(x)}}=\norm{e(x)}$ is upper semicontinuous and vanishes at infinity as well.
	
	%Now we can compute
	
	%\begin{align*}\norm{e}^2=\norm{\lk e,e\rk_A}&=\sup\limits_{x\in X} \norm{\lk e,e\rk_A(x)}\\
	% &=\sup\limits_{x\in X} \norm{\lk e(x),e(x)\rk_{A_x}}=\sup\limits_{x\in X}\norm{e(x)}^2,
	%\end{align*}
	%establishing $(2)$.
	
	%For the last part note that for $e=e'\psi\in EC_0(X)$ we have $ef(x)-ef=e'\psi f(x)-e'\psi f=e'(\psi f(x)-\psi f)$. But $\psi f(x)-\psi f$ vanishes in $x$ and thus in the quotient $(ef)(x)=e(x)f(x)$.
%\end{proof}
Let $\mathcal{E}=\coprod_{x\in X} E_x$ be the disjoint union of the fibres. We want to see, that in analogy to $C_0(X)$-algebras, there is a topology on $\mathcal{E}$ such that $E$ is isomorphic (as a Hilbert-$A$-module) to $\Gamma_0(X,\mathcal{E})$, where the inner product and $A$-action on the latter are defined pointwise (using the identification $\Gamma_0(X,\mathcal{A})\cong A$).

We need some preparations for this: Consider the compact operators $K(E\oplus A)$. Then we have an embedding $i_E:E\rightarrow K(E\oplus A)$ given by $$i_E(e)=\begin{pmatrix}
0 & e\\
0 & 0
\end{pmatrix}.$$ Analogously, we get embeddings of each fibre $i_{E_x}:E_x\rightarrow K(E_x\oplus A_x)\cong K(E\oplus A)_x$. Since $K(E\oplus A)$ is a $C_0(X)$-algebra, there is a topology on $\mathcal{K}(E\oplus A):=\coprod_{x\in X} K(E\oplus A)_x$ such that $K(E\oplus A)\cong \Gamma_0(X,\mathcal{K}(E\oplus A))$.
The inclusions $i_{E_x}$ induce an inclusion $i:\mathcal{E}\rightarrow \mathcal{K}(E\oplus A)$ and we equip $\mathcal{E}$ with the induced topology. Write $\Gamma_0(X,\mathcal{E})$ for the continuous sections of the bundle $\mathcal{E}\rightarrow X$ vanishing at infinity. Then we get a commutative diagram, where the homomorphism at the top is given by $e\mapsto [x\mapsto e(x)]$ and the right vertical map is given by sending $f\in \Gamma_0(X,\mathcal{E})$ to the map $x\mapsto i_{E_x}(f(x))$:
\begin{center}
	\begin{tikzpicture}[description/.style={fill=white,inner sep=2pt}]
	\matrix (m) [matrix of math nodes, row sep=3em,
	column sep=2.5em, text height=1.5ex, text depth=0.25ex]
	{ E &  \Gamma_0(X,\mathcal{E})\\
		K(E\oplus A) &  \Gamma_0(X,\mathcal{K}(E\oplus A))\\ };
	\path[->,font=\scriptsize]
	(m-1-1) edge node[auto] {$  $} (m-1-2)
	(m-2-1) edge node[auto] {$ \cong $} (m-2-2)
	(m-1-1) edge node[auto] { $ i_E $ } (m-2-1)
	(m-1-2) edge node[auto] {$  $} (m-2-2)
	;
	\end{tikzpicture}
\end{center}
Thus, the isomorphism in the bottom row restricts to an isomorphism $E\rightarrow \Gamma_0(X,\mathcal{E})$ as desired.

In the next step, we want to define pullbacks of Hilbert modules with respect to the $C_0(X)$-action. If $f:Y\rightarrow X$ is a continuous map and $A$ is a $C_0(X)$-algebra we can form the pullback $f^*A$ of $A$ under $f$. We equip it with the canonical right Hilbert $f^*A$-module structure. Define a left $A$-action $\Phi:A\rightarrow L(f^*A)$ by $(\Phi(a)f)(y)=a(f(y))f(y)$. One easily checks that this is a well-defined $*$-homomorphism.

\begin{defi}\cite[Définition~4.3]{LeGall99}
	Suppose $A$ is a $C_0(X)$-algebra and $E$ a right Hilbert $A$-module. If $f:Y\rightarrow X$ is a continuous map we define the pullback $f^*E$ of $E$ as the internal tensor product $f^*E:=E\otimes_{\Phi} f^*A$.
\end{defi}
For $y\in Y$ we then have $(f^*E)_y= (E\otimes_{\Phi} f^*A)_y\cong E\otimes_{\Phi} f^*A\otimes_{f^*A} (f^*A)_y\cong E\otimes_A A_{f(y)}=E_{f(y)}$. Here we used that for each $C_0(X)$-algebra $A$ there is a canonical isomorphism $A\otimes_A A_x\rightarrow A_x$ given by $a\otimes b(x)\mapsto ab(x)$.
The following proposition is concerned with the behaviour of the interior tensor product under pullbacks.
\begin{prop}\cite[Proposition~2.3.3]{LeGall} \label{Prop:Bundles and Pullbacks}
	Let $A, B$ be two $C_0(X)$-algebras. If $E$ is a Hilbert $A$-module, $F$ is a Hilbert $B$-module, and  $\Phi:A\rightarrow L(F)$ is a $*$-homo\-morphism, then for every continuous map $f:Y\rightarrow X$ there is a canonical isomorphism of Hilbert $f^*B$-modules
	$$f^*E\otimes_{f^*A}f^*F\rightarrow f^*(E\otimes_A F).$$	
	In particular for each $x\in X$, there is a canonical isomorphism
	$$(E\otimes_A F)_x\cong E_x\otimes_{A_x}F_x.$$
\end{prop}
We can now define what we mean by a groupoid action on a Hilbert module. For this let $(A,G,\alpha)$ be a groupoid dynamical system and $E$ be a right Hilbert $A$-module. From the discussion above we know that $E$ is equipped with a $C_0(G^{(0)})$-action arising from the corresponding action on $A$. Now, if $d,r:G\rightarrow G^{(0)}$ denote the domain and range maps respectively, we can form the pullback modules $d^*E$ and $r^*E$. By construction $r^*E$ is a right Hilbert $r^*A$-module, but we can also equip it with the structure of a right Hilbert $d^*A$-module by letting $x\cdot a:=x\cdot \alpha(a)$ and $\lk x,y\rk_{d^*A}:=\alpha^{-1}(\lk x,y\rk_{r^*A})$.

Thus, we can consider elements $T\in L_{d^*A}(d^*E,r^*E)$. For $g\in G$ consider the operator $T_g\in L_{A_{d(g)}}(E_{d(g)},E_{r(g)})$ induced by $T$ on each fibre. Using Remark \ref{Remark:Fibre of Hilbert module can be defined as tensor product} this operator can also be described as \[T_g=T\otimes\alpha_g:E_{d(g)}=d^*E\otimes_{d^*A} A_{d(g)}\rightarrow r^*E\otimes_{d^*A} A_{r(g)}=E_{r(g)}.\]

\begin{defi}
Let $A$ be a $G$-algebra and $E$ a right Hilbert $A$-module. An action of $G$ on $E$ is a unitary $V\in L_{d^*A}(d^*E,r^*E)$ such that $V_gV_{g'}=V_{gg'}$ for all $(g,g')\in G^{(2)}$.
\end{defi}

For every locally compact Hausdorff groupoid $G$ with Haar-system $\lambda$ there is a canonical $G$-equivariant Hilbert $C_0(G^{(0)})$-module denoted $L^2(G)$ given as the completion of the complex vector space $C_c(G)$ with respect to the $C_0(G^{(0)})$-valued inner product
$$\lk f_1,f_2\rk(x)=\int\limits_{G^x}\overline{f_1(g)}f_2(g)d\lambda^x(g),$$
and right $C_0(G^{(0)})$-action
$$(f\cdot\varphi)(g)=f(g)\varphi(r(g)).$$
Note that $L^2(G)$ is a full Hilbert $C_0(G^{(0)})$-module in the sense that the ideal $\lk L^2(G),L^2(G)\rk$ is dense in $C_0(G^{(0)})$ by an application of the Stone-Weierstraß-Theorem.

Now we define a $G$-action on $L^2(G)$: From \cite[Lemma~4.37]{Goehle} we know that there are isomorphisms $d^*(C_0(G^{(0)}))\cong C_0(G\times_{d,r}G)$ and $r^*(C_0(G^{(0)}))\cong C_0(G\times_{r,r}G)$.
Thus we have $$d^*(L^2(G))=L^2(G)\otimes_{C_0(G^{(0)})}d^*(C_0(G^{(0)}))\cong L^2(G)\otimes_{C_0(G^{(0)})}C_0(G\times_{d,r}G)$$
and similarly $r^*(L^2(G))\cong L^2(G)\otimes_{C_0(G^{(0)})}C_0(G\times_{r,r}G).$
Now we define $V:d^*(L^2(G))\rightarrow r^*(L^2(G))$ as $id_{L^2(G)}\otimes lt$,
where $lt:C_0(G\times_{d,r}G)\rightarrow C_0(G\times_{r,r}G)$ is given by
$$lt(f)(g,h)=f(g,g^{-1}h).$$
Then $V$ is a unitary with $V_{gg'}=V_gV_{g'}$ for all $(g,g')\in G^{(2)}$.
%\begin{bem}

%To see this we apply the Stone-Weierstraß-Theorem: If $x\in G^{(0)}$ pick a compact neighbourhood $V\subseteq G$ of $x$. Now let $f\in C_c(G)$ be any function such that $f= 1$ on $V$. Then we have
%$$\lk f,f\rk(x)=\int\limits_{G^x}\betrag{f(g)}^2d\lambda^x(g)\geq \int\limits_V\betrag{f(g)}^2d\lambda^x(g)=\lambda^x(V\cap G^x)>0$$
%since $supp(\lambda^x)=G^x$.
%If $x,y\in G^{(0)}$ such that $x\neq y$ the set $G\setminus G^y$ is an open neighbourhood of $x$. Since $G$ is locally compact we can find a compact neighbourhood $V$ of $x$ such that $V\subseteq G\setminus G^y$. Then pick a function $f\in C_c(G)$ such that $f=1$ on $V$ and $f=0$ off of $G\setminus G^y$. It follows that
%$\lk f,f\rk(x)\neq 0$ by the same computation as above but $\lk f,f\rk(y)=\int\limits_{G^y}\betrag{f(g)}^2d\lambda^y(g)=0$.
%\end{bem}
More generally, if $A$ is any $G$-algebra we can view it as a $C_0(G^{(0)})-A$ bimodule and form the $G$-equivariant right Hilbert $A$-module
$$L^2(G,A):=L^2(G)\otimes_{C_0(G^{(0)})} A.$$
Note that we could also concretely construct $L^2(G,A)$ as the completion of the pre-Hilbert $A$-module $\Gamma_c(G,d^*\mathcal{A})$ with respect to the inner product
$$\lk f_1,f_2\rk_A(x)=\int\limits_{G^x} \alpha_g(f_1(g)^*f_2(g))d\lambda^x(g)$$
and the right $A$-action
$$(f\cdot a)(g)=f(g)\alpha_{g^{-1}}(a(r(g))).$$
A canonical isomorphism
$$\Phi: L^2(G)\otimes_{C_0(G^{(0)})} A\rightarrow \overline{\Gamma_c(G,d^*\mathcal{A})}$$
is given on elementary tensors by
$$\Phi(f\otimes a)(g)=f(g)\alpha_{g^{-1}}(a(r(g)))$$
for $f\in C_c(G)$ and $a\in A$.
%One easily checks that $\Phi$ is isometric. Thus, it suffices to show that $im(\Phi)$ is dense in $\Gamma_c(G,d^*\mathcal{A})$. But it follows from Corollary \ref{Cor:DenseInInductiveLimit} that elements of the form $\Phi(f\otimes a)$ span a dense subset of $\Gamma_c(G,d^*\mathcal{A})$ with respect to the inductive limit topology.
%So for $\varphi\in \Gamma_c(G,d^*\mathcal{A})$ we can find a net $(\varphi_i)_i$ in $span\lbrace\Phi(f\otimes a)\mid f\in C_c(G),a\in A\rbrace$ and a compact set $K\subseteq G$ such that $supp(\varphi_i)\subseteq K$ eventually and $\norm{\varphi-\varphi_i}_\infty\rightarrow 0$.
%But then we eventually have:
%\begin{align*}
%\norm{\varphi-\varphi_i}^2& =\sup\limits_{x\in G^{(0)}} \norm{\int\limits_{G^x}\alpha_g((\varphi-\varphi_i)(g)^*(\varphi-\varphi_i)(g))d\lambda^x(g)}\\
%& \leq \sup\limits_{x\in G^{(0)}}\int\limits_{G^x}\norm{(\varphi-\varphi_i)(g)}^2d\lambda^x(g)\\
%& \leq \norm{\varphi_i-\varphi}^2_\infty \lambda^x(supp(\varphi-\varphi_i))\\
%& \leq  \norm{\varphi_i-\varphi}^2_\infty \lambda^x(K)\\
%& \leq  \norm{\varphi_i-\varphi}^2_\infty C\\
%& \rightarrow 0
%\end{align*}
%Thus $\varphi_i\rightarrow \varphi$ in the norm induced by the inner product. Consequently, we have $\varphi\in \overline{im(\Phi)}=im(\Phi)$.
The following result is a special case of \cite[Proposition~2.3.2]{LeGall}:
\begin{prop}
There is a $G$-equivariant $*$-isomorphism
$$\Psi:K(L^2(G))\otimes_{G^{(0)}}^{max} A\rightarrow K(L^2(G,A))$$
given by $\Psi(T\otimes a)(\xi\otimes b)=T\xi\otimes ab$.
Consequently, $L^2(G,A)$ implements a $G$-equivariant Morita-equivalence
$$(K(L^2(G))\otimes_{G^{(0)}}^{max} A, Ad\ V\otimes\alpha)\sim_M (A,\alpha).$$
\end{prop}
%\begin{proof}
%First identify $A$ with $K_A(A)$ via $a\mapsto L_a$, where $L_a(b)=ab$ as in \cite[Example~2.26]{Morita}. Then define a map 
%$$ \Phi:K(L^2(G))\otimes K(A)\rightarrow K(L^2(G,A))$$
%by $\Phi(S\otimes T)(f\otimes a)=Sf\otimes Ta$. An easy computation then shows that 
%$\Phi(\theta_{f_1,f_2}\otimes \theta_{a,b})=\theta_{f_1\otimes a,f_2\otimes b}$. Thus $\Phi$ is well-defined and we can also immediately define an inverse on rank-one operators by this formula.
%\end{proof}

Even more generally, let $E$ be a $G$-equivariant Hilbert $A$-module. As seen above there is a natural $*$-homomorphism $\Phi:C_0(G^{(0)})\rightarrow L(E)$ induced by the $C_0(G^{(0)})$-structure of $A$. Thus we can form the tensor product
$$L^2(G,E):=L^2(G)\otimes_{\Phi} E$$
Again we could also explicitly construct $L^2(G,E)$ as the completion of the pre-Hilbert $A$-module $\Gamma_c(G,d^*\mathcal{E})$ with respect to the inner product
$$\lk f_1,f_2\rk_A(x)=\int\limits_{G^x} \alpha_g(\lk f_1(g),f_2(g)\rk_{A_{d(g)}})d\lambda^x(g)$$
equipped with a right $A$-action given by
$$(f\cdot a)(g)=f(g)\alpha_{g^{-1}}(a(r(g))).$$
Again, an isomorphism
$$\Phi: L^2(G)\otimes_{\Phi} E\rightarrow \overline{\Gamma_c(G,d^*\mathcal{E})}$$
is given on elementary tensors by
$$\Phi(f\otimes e)(g)=f(g)V_{g^{-1}}(e(r(g)))$$
for $f\in C_c(G)$ and $e\in E$.

Finally, we recall the definitions of groupoid equivariant $\KK$-theory, as introduced by Le Gall in \cite{LeGall,LeGall99}. Throughout we will assume, that $G$ is a locally compact, second countable Hausdorff groupoid.
Let $A$ and $B$ be two $G$-algebras.
A \textit{$G$-equi\-variant Kasparov Triple} for $(A,B)$ is a triple $(E,\Phi,T)$, where $E$ is a $G$-equivariant $\ZZ/2\ZZ$-graded right Hilbert $B$-module, $\Phi:A\rightarrow L(E)$ is a graded $G$-equivariant $*$-homomorphism and $T\in L(E)$ is an adjointable operator of degree $1$, such that
$\Phi(a)(T-T^*),\ \Phi(a)(T^2-1),\ [\Phi(a),T]\in K(E)$ for every $a\in A$, and for every element $f\in r^*A\cong \Gamma_0(G,r^*\mathcal{A})$ the mapping
$$g\mapsto \Phi_{r(g)}(f(g))(T_{r(g)}-V_gT_{d(g)}V_g^*)$$
defines and element in $\Gamma_0(G,r^*\mathcal{K}(E))=r^*(K(E))$.

Two Kasparov triples $(E_i,\Phi_i,T_i)$, $i=1,2$ for $(A,B)$ are called \textit{unitarily equivalent} if there exists a $G$-equivariant unitary $U\in L(E_1,E_2)$ of degree $0$, which intertwines the representations $\Phi_1$ and $\Phi_2$ as well as the operators $T_1$ and $T_2$. We denote the set of all unitary equivalence classes of such triples by $\mathbb{E}^G(A,B)$. A Kasparov triple $(E,\Phi,T)$ is called \textit{essential} if $\overline{\Phi(A)E}=E$.

A \textit{homotopy} in $\mathbb{E}^G(A,B)$ is an element in $\mathbb{E}^G(A,C([0,1],B))$ and the triples in $\mathbb{E}^G(A,B)$ obtained by evaluating at $0$ and $1$ respectively are called \textit{homotopic}. Homotopy is an equivalence relation on $\mathbb{E}^G(A,B)$ and the set of homotopy classes of $\mathbb{E}^G(A,B)$ is denoted by $\KK^G(A,B)$.

It is not hard to see, that homotopy respects the operation of taking direct sums of Kasparov triples. Using this one can show that $\KK^G(A,B)$ is an abelian group with respect to taking direct sums of the representing Kasparov triples. The same proof as in the non-equivariant setting (see \cite[Proposition~17.3.3]{MR1656031}) works.

The higher $\KK$-groups are defined as follows:
For $n\in \NN$ and two $G$-algebras $A$ and $B$, define
$$\KK^G_n(A,B)=\KK^G(A\otimes C_0(\RR^n),B)$$

It is well-known that $\KK^G$ is functorial, contravariant in the first, and covariant in the second variable.
%\begin{prop}
%Let $A_1,A_2$ and $B$ be $G$-algebras, and $\varphi:A_1\rightarrow A_2$ a $G$-equivariant $*$-homo\-morphism. If $(E,\Phi,T)\in\mathbb{E}^G(A_2,B)$, then the triple $(E,\Phi\circ \varphi,T)\in\mathbb{E}^G(A_1,B)$ and the mapping $(E,\Phi,T)\mapsto (E,\Phi\circ \varphi,T)$ defines a group homomorphism
%$$\varphi^*:\KK^G(A_2,B)\rightarrow \KK^G(A_1,B).$$
%If $(E,\Phi,T)\in \mathbb{E}^G(B,A_1)$, then $(E\otimes_\varphi A_2,\Phi\otimes 1,T\otimes 1)\in \mathbb{E}^G(B,A_2)$ and the mapping $(E,\Phi,T)\mapsto (E\otimes_\varphi A_2,\Phi\otimes 1,T\otimes 1)$ defines a group homomorphism
%$$\varphi_*:\KK^G(B,A_1)\rightarrow \KK^G(B,A_2).$$
%\end{prop}
As in the non-equivariant case $\KK^G$-theory comes with a version of the Kasparov product, i.e. for separable $G$-algebras $A,B$ and $C$ there exists a bilinear map
$$\otimes_C:\KK^G(A,C)\times \KK^G(C,B)\rightarrow \KK^G(A,B),$$
which is associative in the appropriate sense (see \cite[Theorème~6.3]{LeGall99} for details).
We shall also use the fact, that the equivariant $\KK$-theory is functorial with respect to groupoid homomorphisms (see \cite[Propositions~7.1 and 7.2]{LeGall99})

An important special case of this is given by the inclusion of a subgroupoid $H\hookrightarrow G$. In this case we will also denote the resulting map $\mathrm{KK}^G(A,B)\rightarrow \mathrm{KK}^H(A_{\mid H},B_{\mid H})$ by $res_H^G$ and call it the \textit{restriction homomorphism}.

The following proposition extends the pushforward construction for $\mathrm{C}^*$-algebras as in Proposition \ref{Prop:PushforwardAction} to Hilbert modules and hence provides a homomorphism on the level of $\mathrm{KK}^G$-theory.
\begin{prop}\label{Prop:Pushforward in KK}
	Let $G$ be a locally compact Hausdorff groupoid and $X$ a $G$-space with anchor map $p:X\rightarrow G^{(0)}$. For every pair of $G\ltimes X$-algebras $A$ and $B$ the map $p$ gives rise to a homomorphism
	$$p_*:KK^{G\ltimes X}(A,B)\rightarrow KK^G(A,B),$$
	compatible with the Kasparov product in the following sense: If $A,B$ and $C$ are separable $G\ltimes X$-algebras and $x\in KK^{G\ltimes X}(A,C)$ and $y\in KK^{G\ltimes X}(C,B)$, then
	$$p_*(x\otimes_C y)=p_*(x)\otimes_C p_*(y).$$
\end{prop}

\begin{proof}
	On the level of Kasparov triples $(E,\Phi,T)\in \mathbb{E}^{G\ltimes X}(A,B)$ the desired map is basically given by the identity. Viewing $A$ and $B$ as $G$-algebras via the pushforward construction (see Proposition \ref{Prop:PushforwardAction}) also $E$ inherits a canonical fibration over $G^{(0)}$ and using the same formulas as in the $\mathrm{C}^*$-algebraic construction we can push the action of $G\ltimes X$ forward to obtain an action of $G$ on $E$. Since neither the operator $T$ nor the left action $\Phi$ of $A$ on $E$ changed, it follows from the isomorphism $\pi_*(R^*(K(E)))\cong r^*(p_*(K(E)))$, where $R:G\ltimes X\rightarrow X$ is the range map and $\pi:G\ltimes X\rightarrow G$ is the projection on the first factor (confer Proposition \ref{Prop:PullbackAndPushforward}), that $(E,\Phi,T)$ equipped with this $G$-action represents an element in $\mathbb{E}^G(A,B)$. Applying the same arguments to a homotopy gives the desired homomorphism. Using again, that only the action on $E$ changes under $p_*$ it is easy to see, that $p_*$ respects the Kasparov product.
\end{proof}
\begin{prop}\label{Prop:Induction homomorphism}
	Let $G$ be a locally compact Hausdorff groupoid admitting a Haar system and $H\subseteq G$ a closed subgroupoid. Suppose, that $A$ and $B$ are separable $H$-algebras. Then there is an \textit{induction homomorphism} 
	$$\mathsf{Ind}_H^G:\mathrm{KK}^H(A,B)\rightarrow \mathrm{KK}^G(Ind_H^X A,Ind_H^X B),$$
	where $X:=d^{-1}(H^{(0)})$. The homomorphism $\mathsf{Ind}_H^G$ is
	compatible with the Kasparov product in the following sense:
	If $A,B$ and $C$ are separable $H$-algebras and $x\in \mathrm{KK}^H(A,C)$ and $y\in \mathrm{KK}^H(C,B)$, then
	$$\mathsf{Ind}_H^G(x\otimes_C y)=\mathsf{Ind}_H^G(x)\otimes_{Ind_H^G C} \mathsf{Ind}_H^G(y).$$
\end{prop}
\begin{proof}
	The space $X=d^{-1}(H^{(0)})\subseteq G$ with the induced topology implements an equivalence between the groupoids $G\ltimes X/H$ and $H$. Hence by \cite[Definition~7.1, Theorem~7.2]{LeGall99} there is a canonical homomorphism $X^*:\mathrm{KK}^H(A,B)\rightarrow \mathrm{KK}^{G\ltimes X/H}(Ind_H^G A,Ind_H^G B)$ compatible with the Kasparov product (compare Remark \ref{Remark:Equivalence and Induction}). If we now compose this homomorphism with the homomorphism obtained by pushing forward alsong $G\ltimes X/H\rightarrow G$ as in Proposition \ref{Prop:Pushforward in KK} we obtain the desired map and compatibility with the product follows since both maps in this composition have this property.
	Alternatively, one could define this map explicitly along the lines of \cite[§5]{MR1388299} as follows: If $x\in \mathrm{KK}^G(A,B)$ is represented by the Kasparov triple $(E,\Phi,T)$, then we can form the induced Hilbert $Ind_H^X B$-module $Ind_H^X E$ as the set of all $\xi\in \Gamma_b(X,d^*\mathcal{E})$ such that $V_h(\xi(x))=\xi(xh^{-1})$ for all $x\in X$ and  $h\in H$ and $[xH\mapsto \norm{\xi(x)}]\in C_0(X/H),$ equipped with the pointwise actions and inner products. Pointwise action on the left gives a representation $\mathsf{Ind}_H^G\Phi:Ind_H^X A\rightarrow L(Ind_H^X E)$. Using a cutoff function $c:X\rightarrow \RR^+$ for the groupoid $X\rtimes H$ as in Definition \ref{Defi:cutoff} we can define an operator $\widetilde{T}\in L(Ind_H^X E)$ by
	$$(\widetilde{T}\xi)(x)=\int\limits_{H^{d(x)}}c(xh)V_h(T(\xi(xh)))d\lambda^{d(x)}(h).$$
	Then $(Ind_H^X E,Ind_H^X \Phi,\widetilde{T})$ can be shown to be a Kasparov tripel representing the element $\mathsf{Ind}_H^G(x)\in \mathrm{KK}^G(Ind_H^X A,Ind_H^X B)$.
	
\end{proof}
Finally, Le Gall showed in \cite[Propositions~7.2.1~and~7.2.2]{LeGall} that for a locally compact $\sigma$-compact groupoid $G$ equipped with a Haar system and two $G$-algebras $A$ and $B$ there exits a canonical descent homomorphism
$$j_{G}:\KK^G(A,B)\rightarrow \KK(A\rtimes_r G, B\rtimes_r G),$$ 
which is is compatible with the Kasparov product.

For later reference let us outline the construction of the map $j_G$ in the étale setting:
Given a Kasparov triple $(E,\Phi,T)\in\mathbb{E}^G(A,B)$
we can define a right $\Gamma_c(G,r^*\mathcal{B})$-module structure and a $\Gamma_c(G,r^*\mathcal{B})$-valued inner product on $\Gamma_c(G,r^*\mathcal{E})$ by
$$\lk \xi_1,\xi_2\rk (g)=\sum\limits_{h\in G_{r(g)}} \beta_{h^{-1}}(\lk \xi_1(h),\xi_2(hg)\rk)$$
and
$$(\xi f)(g)=\sum\limits_{h\in G^{r(g)}} \xi(h)\beta_h(f(h^{-1}g)).$$

The Hilbert $B\rtimes_r G$-module obtained by completion is denoted by $E\rtimes_r G$. A representation $\widetilde{\Phi}:A\rtimes_r G\rightarrow L(E\rtimes_r G)$ is determined by the formula
$$(\widetilde{\Phi}(f)\xi)(g)=\sum\limits_{h\in G^{r(g)}}\Phi_{r(h)}(f(h))V_h(\xi(h^{-1}g)),$$
where $f\in \Gamma_c(G,r^*\mathcal{A})$ and $\xi\in \Gamma_c(G,r^*\mathcal{E})$.
Finally, one defines an operator $\widetilde{T}\in L(E\rtimes_r G)$ by $$(\widetilde{T}\xi)(g):=T_{r(g)}(\xi(g)).$$
Then one can show that $(E\rtimes_r G,\widetilde{\Phi},\widetilde{T})\in \mathbb{E}(A\rtimes_r G,B\rtimes_r G)$ and the map $j_G$ is given by $j_G([E,\Phi,T])=[E\rtimes_r G,\widetilde{\Phi},\widetilde{T}]$.
\begin{bem}
Equivalently, one can use the canonical representation $B\rightarrow M(B\rtimes_r G)$ to define $E\rtimes_r G$ as the tensor product $E\otimes_B (B\rtimes_r G)$ (see \cite[Définition~7.2.1]{LeGall}).
\end{bem}
\section{Automatic Equivariance}
In this section we shall elaborate, when the operator in a Kasparov triple can be chosen in an equivariant way. The main ideas are based on the paper \cite{Meyer00equivariantkasparov}, which deals with the case of locally compact groups.

Let $A$ and $B$ be (trivially graded) $G$-algebras and let $(E,\Phi,T)$ be an equivariant Kasparov triple for $(A,B)$. We call $T'\in L(E)$ a \textit{compact perturbation} of $T$ if the operators $\Phi(a)(T'-T)$ and $(T'-T)\Phi(a)$ are compact for all $a\in A$. In this case the triples $(E,\Phi,T)$ and $(E,\Phi,T')$ are \textit{operator homotopic} via the trivial path $T_s:=(1-s)T+sT'$ and hence represent the same element in $\KK^G(A,B)$ (see for example \cite[Corollary~17.2.6]{MR1656031}).
To illustrate the usefulness of the above notion, we want to show (the well-known result) that if $G$ is a proper groupoid, then every element in $\KK^G(A,B)$ can be represented by a Kasparov triple with a $G$-equivariant operator. For the proof we need the following notion:

\begin{defi}\cite[Definition~6.7]{Tu99}\label{Defi:cutoff}
	Let $G$ be a locally compact Hausdorff groupoid equipped with a Haar system $(\lambda^u)_{u\in G^{(0)}}$. A \textit{cutoff} function for $G$ is a continuous map $c:G^{(0)}\rightarrow \RR^+$ such that
	\begin{enumerate}
		\item for every $u\in G^{(0)}$ we have $\int_{G^u} c(d(g))d\lambda^u(g)=1$, and
		\item the map $r:supp(c\circ d)\rightarrow G^{(0)}$ is proper.
	\end{enumerate}
\end{defi}
Tu showed in \cite[Propositions~6.10 and~6.11]{Tu99}) that a locally compact Hausdorff groupoid equipped with a Haar system admits a cutoff function if and only if it is proper. If moreover the orbit space $G\setminus G^{(0)}$ is compact, then $G$ admits a cutoff function with compact support.

We are now ready for the proof of the promised example using compact perturbations.
\begin{prop}\label{proper groupoids: Operators can be chosen equivariant}
Let $G$ be a proper groupoid with Haar system $(\lambda^u)_{u\in G^{(0)}}$ and $(E,\Phi,T)\in \EE^G(A,B)$ a $G$-equivariant Kasparov-triple. Then there is a $G$-equivariant operator $T^G\in L(E)$ which is a compact pertubation of $T$.
\end{prop}
\begin{proof}
Let $(E,\Phi,T)\in \EE^G(A,B)$ be given. Choose a cutoff function $c$ for $G$.
Then for $u\in G^{(0)}$ define
$$(T^G)_u=\int\limits_{G^u} c(d(g))V_gT_{d(g)}V_{g^{-1}}d\lambda^u(g).$$
This clearly defines an operator $T^G:E\rightarrow E$. It is adjointable since we can apply the whole construction to $T^*$ and an easy computation reveals that $(T^*)^G$ is the adjoint for $T^G$. Another elementary computation using inner products shows that $T^G$ is indeed $G$-equivariant.

It remains to show that $T^G$ is a compact pertubation of $T$, i.e. we need to see that $\Phi(a)(T^G-T)\in K(E)$ for all $a\in A$. By density we can assume that $a$ viewed as a section $G^{(0)}\rightarrow\mathcal{A}$ has compact support. We have
\begin{align*}
(\Phi(a)(T^G-T))_u & =\Phi(a)_u\left(\int\limits_{G^u}c(d(g))V_gT_{d(g)}V_{g^{-1}}d\lambda^u(g)-T_u\right)\\
%& = \Phi(a)_u\left(\int\limits_{G^u}c(d(g))V_gT_{d(g)}V_{g^{-1}}d\lambda^u(g)-\int\limits_{G^u}c(d(g))T_u d\lambda^u(g)\right)\\
& = \Phi(a)_u\left(\int\limits_{G^u}c(d(g))\left(V_gT_{d(g)}V_{g^{-1}}-T_u\right)d\lambda^u(g)\right)\\
& = \int\limits_{G^u}c(d(g))\Phi(a)_u\left(V_gT_{d(g)}V_{g^{-1}}-T_{u}\right)d\lambda^u(g) \\
& = \int\limits_{G^u}\Phi_{r(g)}(c(d(g))a(r(g)))\left(V_gT_{d(g)}V_{g^{-1}}-T_{r(g)}\right)dg.\\
\end{align*}
Note that $g\mapsto c(d(g))a(r(g))$ defines an element $b$ in $\Gamma_c(G,r^*\mathcal{A})$ (continuity is obvious and $supp(b)\subseteq supp(c\circ d)\cap r^{-1}(supp(a))$ implies that $b$ has compact support). Since $(E,\Phi,T)$ is a $G$-equivariant Kasparov triple the family $$(\Phi_{r(g)}(c(d(g))a(r(g)))\left(V_gT_{d(g)}V_{g^{-1}}-T_{r(g)}\right))_{g\in G}$$ defines an element in $r^*K(E)$. Then, by Lemma \ref{r!}, integration against the Haar system yields an element in $K(E)$. Consequently, the above computation shows $\Phi(a)(T^G-T)\in K(E)$ as desired.
\end{proof}

\begin{defi}
	Let $E_1$ be a graded $G$-equivariant Hilbert $A$-module and $E_2$ be a graded $G$-equivariant Hilbert $A-B$-bimodule and $E:=E_1\hat{\otimes}_A E_2$. For $x\in E_1$ define an operator $T_x\in L(E_2,E)$ by 
	$$T_x(y)=x\otimes y.$$
	Let $F_2\in L(E_1)$. An operator $F\in L(E)$ is called an \textit{$F_2$-connection} if
	$T_xF_2-(-1)^{\partial x\partial F_2}FT_x\in K(E_2,E)$ and $F_2T_x^*-(-1)^{\partial x\partial F_2}T_x^*F\in K(E,E_2)$
	for all $x\in E_1$.
\end{defi}

%\begin{bem}
%	Suppose $(E,\Phi,F)\in \mathbb{E}^G(A,B)$ is an essential triple. Then we have a canonical identification $E\cong A\otimes_{\Phi} E$.
%	\begin{enumerate}
%		\item Under the above identification the operator $T_a$ is just %given by $\Phi(a)$ and since
%		$[\Phi(a),F]\in K(E)$ for all $a\in A$ we have that $F$ is an $F$-connection.
%		\item The operator $F\in L(E)$ is a $0$-connection if and only if both $F\Phi(a)$ and $\Phi(a)F$ are in $K(E)$. Consequently, an operator $F'\in L(E)$ is a compact perturbation of $F$ if and only if $F-F'$ is a $0$-connection.
%	\end{enumerate}
%\end{bem}

%The following Lemma is a groupoid equivariant version of \cite[Proposition~18.3.4]{MR1656031} and proved in the same way.
%\begin{lemma} Let $E_1$ be a $G$-equivariant Hilbert-$A$-module, $E_2$ be a $G$-equivariant Hilbert $A-B$-bimodule and $E_3$ a $G$-equivariant Hilbert $B-C$ bimodule.
%	\begin{enumerate}
		
%		\item If $F$ is an $F_2$-connection and $F'$ is an $F'_2$-connection, then $F+F'$ is an $(F_2+F'_2)$-connection.
		
%		\item Let $F_3\in L(E_3)$ such that $[F_3,B]\in K(E_3)$. Now if $F_{23}$ is an $F_3$-connection on $E_2\otimes_B E_3$, and $F$ is an $F_{23}$-connection on $E_1\otimes_A(E_2\otimes_B E_3)$, then $F$ is an $F_3$-connection on $(E_1\otimes_A E_2)\otimes_B E_3$.
%	\end{enumerate}
%\end{lemma}

Now we prove a generalization of \cite[Lemma~3.1]{Meyer00equivariantkasparov}.
\begin{lemma} Let $G$ be a $\sigma$-compact locally compact groupoid with Haar system, and let $A$ and $B$ be $\sigma$-unital $G$-algebras and $(E,\Phi,T) \in \mathbb{E}^G(A,B)$ an essential Kasparov triple. Then there is a $G$-equivariant $T$-connection $T'$ on $L^2(G,E)\cong L^2(G,A)\otimes_\Phi E$. If $T$ is a self-adjoint contraction, then so is $T'$.
\end{lemma}
\begin{proof}
	Consider the space $\Gamma_c(G,d^*\mathcal{E})$ of continuous sections of $d^*\mathcal{E}$ with compact support. The inner product 
	$$\lk f_1,f_2\rk_B(u)=\int\limits_{G^u}\beta_g(\lk f_1(g),f_2(g)\rk_{B_{d(g)}})d\lambda^u(g)$$
	together with the right $B$-action
	$$(f\cdot b)(g)=f(g)\beta_{g^{-1}}(b(r(g)))$$
	turns $\Gamma_c(G,d^*\mathcal{E})$ into a pre-Hilbert $B$-module. Its completion is canonically identified with $L^2(G,E)$ via the isomorphism which sends an elementary tensor $f\otimes e\in L^2(G)\otimes_{C_0(G^{(0)})} E$ to the function $g\mapsto f(g)V_{g^{-1}}e(r(g))$. Here, the $V_g$ denote the unitaries implementing the $G$-action on $E$. Since $\Phi$ is essential, we have $$L^2(G,E)\cong L^2(G,A)\otimes_\Phi E.$$ Now define
	$T':\Gamma_c(G,d^*\mathcal{E})\rightarrow \Gamma_c(G,d^*\mathcal{E})$
	by $$(T'f)(g)=T_{d(g)}(f(g)).$$
	We have
	\begin{align*}
	\norm{T'f}^2& = \norm{\lk T'f,T'f\rk_B}\\
	& = \sup\limits_{u\in G^{(0)}}\norm{\lk T'f,T'f\rk_B(u)}\\
	& = \sup\limits_{u\in G^{(0)}}\norm{\int\limits_{G^u}\beta_g(\lk (T'f)(g),(T'f)(g)\rk_{B_{d(g)}})d\lambda^u(g)}\\
	& = \sup\limits_{u\in G^{(0)}}\norm{\int\limits_{G^u}\beta_g(\underbrace{\lk T_{d(g)}(f(g)),T_{d(g)}(f(g))\rk_{B_{d(g)}}}_{\leq \norm{T}^2\lk f(g),f(g)\rk})d\lambda^u(g)}\\
	& \leq \norm{T}^2\norm{f}^2
	\end{align*}
	Thus, $T'$ is bounded with $\norm{T'}\leq \norm{T}$.
	Let us check that $T'$ is indeed $G$-equivariant. If $V'$ denotes the unitary implementing the canonical $G$-action on $L^2(G,E)$ given by $(V'_gf)(s)=f(g^{-1}s)$, then we have
	\begin{align*}
(T_{r(g)}'V_g'f)(s)&=T_{d(s)}(V_g'f(s))\\
&=T_{d(s)}(f(g^{-1}s))\\
&=(T_{d(g)}'f)(g^{-1}s)\\
&=(V_g'T_{d(g)}'f)(s).
	\end{align*}
	An easy computation reveals that self-adjointness of $T$ implies self-adjointness of $T'$.
	
	We claim that $T'$ is a $T$-connection. To show this we have to check that $K:=T_\xi T-T'T_{\xi}\in K(E,L^2(G,E))$ for all $\xi\in L^2(G,A)$.
	Let us first take a closer look at the rank one operators in $ K(E,L^2(G,E))$. For $x,y\in E$ and an element in $L^2(G,E)$ of the form $f\otimes e(g):=f(g)V_{g^{-1}}e(r(g))$ for $f\in C_c(G)$ and $e\in E$ we have
	\begin{align*}
	\theta_{f\otimes e,x}(y)(g)& = ((f\otimes e)\cdot\lk x,y\rk_A)(g)\\
	& = (f\otimes e)(g)\cdot\alpha_{g^{-1}}(\lk x,y\rk_A(r(g)))\\
	& = f(g)V_{g^{-1}}(e\cdot \lk x,y\rk_A(r(g)))\\
	& = (f\otimes \theta_{e,x}(y))(g),
	\end{align*}
	where the last equation again uses the identification $L^2(G)\otimes_{C_0(G^{(0)})}E\cong L^2(G,E)$ described above.
	
	Back to the operator $K$:
	Since elements of the form $f\otimes a$, where $$(f\otimes a)(g)=f(g)\alpha_{g^{-1}}(a(r(g))),$$ form a dense subset of $L^2(G,A)$, in the definition of $T_\xi$ we can restrict to $\xi$ of this form. So in the following computations, let $\xi:=f\otimes a$.
	Recall that the canonical isomorphism $L^2(G,A)\otimes_\Phi E\cong L^2(G,E)$ identifies $\xi\otimes e$ with the function $g\mapsto \Phi_{d(g)}(\xi(g))V_{g^{-1}}(e(r(g)))$.
	Thus, for all $e\in E$ and $g\in G$ we can compute
	\begin{align*}
	(Ke)(g)& =(T_\xi Te)(g)-(T'T_{\xi}e)(g)\\
	& = (\xi\otimes Te)(g) - T_{d(g)}(T_{\xi}e(g))\\
	& = \Phi_{d(g)}(\xi(g))V_{g^{-1}}T_{r(g)}(e(r(g)))-T_{d(g)}\Phi_{d(g)}(\xi(g))V_{g^{-1}}e(r(g))\\
	& = f(g)\Phi_{d(g)}(\alpha_{g^{-1}}(a(r(g))))V_{g^{-1}}T_{r(g)}(e(r(g)))\\
	& \quad -f(g)T_{d(g)}\Phi_{d(g)}(\alpha_{g^{-1}}(a(r(g))))V_{g^{-1}}e(r(g))\\
	& = f(g)V_{g^{-1}}\Phi_{r(g)}(a(r(g)))T_{r(g)}(e(r(g)))\\
	& \quad -f(g)T_{d(g)}V_{g^{-1}}\Phi_{r(g)}(a(r(g)))e(r(g))\\
	& = (\ast)
	\end{align*}
	By adding and subtracting the term 
	$f(g)V_{g^{-1}}T_{r(g)}\Phi_{r(g)}(a(r(g)))e(r(g))$ in the last line we get
	$$(\ast)= (f\otimes [\Phi(a),T]e)(g)+f(g)(V_{g^{-1}}T_{r(g)}-T_{d(g)}V_{g^{-1}})\Phi(a(r(g)))e(r(g)).$$
	Now approximating $[\Phi(a),T]$ by sums of rank one operators and using our description of these it is not hard to see that $e\mapsto f\otimes [\Phi(a),T]e\in K(E,L^2(G,E))$.
	The second summand in $(\ast)$ can be rewritten as
	$$V_{g^{-1}}(T_{r(g)}-V_gT_{d(g)}V_{g^{-1}})\Phi(f(g)a(r(g)))\cdot e(r(g)).$$
	Since $(E,\Phi,T)$ is a $G$-equivariant Kasparov triple, the family $$(T_{r(g)}-V_gT_{d(g)}V_{g^{-1}})\Phi(f(g)a(r(g))))_{g\in G}$$ defines an element in $r^*(K(E))$ and since $f$ has compact support it can be approximated by finite sums of elements of the form $\psi\otimes F$ for $\psi\in C_c(G)$ and $F\in K(E)$ where $(\psi\otimes F)(g)=\psi(g)F_{r(g)}$. Passing to such elements we are left with the term
	$$\psi(g)V_{g^{-1}}F_{r(g)}e(r(g))=\psi(g)V_{g^{-1}}(Fe(r(g)))=(\psi\otimes Fe)(g)$$
	But $e\mapsto \psi\otimes Fe$ can be approximated by rank-one operators as above and thus we have shown that $K\in K(E,L^2(G,E))$.
\end{proof}
Now we can use the exact same arguments as in \cite[Proposition~3.2]{Meyer00equivariantkasparov} to show:
\begin{prop} \label{AutomaticEquivariance}
	Suppose $A$ and $B$ are $\sigma$-unital $G$-algebras and $(E,\Phi,T)$ is an essential Kasparov triple in $\mathbb{E}^G(K(L^2(G))\otimes_{G^{(0)}}^{max} A,B)$. Then there exists a $G$-equivariant compact perturbation of $T$.
\end{prop}
\section{The Compression Isomorphism}

Before we can construct the compression isomorphism we need the following preliminary observation:
\begin{lemma}\label{Lem:InclusionIntoInducedAlgebra}
	Let $G$ be an étale, locally compact Hausdorff groupoid and $H\subseteq G$ a clopen subgroupoid, such that $H^{(0)}=G^{(0)}$. If $A$ is an $H$-algebra, then there is an $H$-equivariant embedding
	$$ i_A:A\rightarrow Ind_H^{G}A$$
	given by the formula
	\[ i_A(a)(g)=\left\{\begin{array}{ll} \alpha_{g^{-1}}(a(r(g))) & ,g\in H \\
	0_{d(g)} & ,else \end{array}\right\}\]
\end{lemma}
\begin{proof}
	%Let us first show that $i_A(a)$ is continuous. So let $(g_\lambda)_\lambda$ be a net in $H$ such that $g_\lambda\rightarrow g$ for some $g\in G'\setminus H$. We need to show that $\alpha_{g_\lambda^{-1}}(a(r(g_\lambda)))\rightarrow 0_{d(g)}$. From axiom (A5) of the definition of upper semicontinuous $C^*$-bundles we know that it is enough to show, that given $\varepsilon>0$ we eventually have
	%$$\norm{\alpha_{g_\lambda^{-1}}(a(r(g_\lambda)))}=\norm{a(r(g_\lambda))}<\varepsilon.$$
	%Since $a\in A\cong \Gamma_0(H^{(0)},\mathcal{A})$ we can find a compact set $K\subseteq H^{(0)}$ such that $\norm{a(x)}<\varepsilon$ for all $x\in H^{(0)}\setminus K$. Now pick compact neighbourhood $U$ of $g\in G'$ such that $d_{\mid U}:U\rightarrow d(U)$ is a homeomorphism (if $U$ is just an open neighbourhood we can pass to compact neighbourhood $V$ of $g$ such that $g\in \mathring{V}\subseteq V\subseteq U$ by local compactness of $G'$).
	%Now we apply the properness of $H$ to conclude that $H_{d(U)}^K$ is compact. Hence $G'\setminus H_{d(U)}^K$ is open and since $g\notin  H_{d(U)}^K$ we eventually have $g_\lambda\in  G'\setminus H_{d(U)}^K$. But since eventually we have $g_\lambda\in U$ we must have $d(g_\lambda) \in H_{d(U)}$. Thus we must have $r(g_\lambda)\in H^{(0)}\setminus K$ eventually and thus $\norm{a(r(g_\lambda))}<\varepsilon$ as claimed.\\
	First, we check that $i_A(a)$ is indeed an element in $Ind_H^G A$.
	The continuity of $i_A(a)$ is clear, as $H$ is clopen in $G$.
	Now let $h\in H$ and $g\in G$ such that $d(g)=d(h)$. Then we clearly have $g\in H \Leftrightarrow gh^{-1}\in H$ and thus in this case we can compute
	$$i_A(a)(gh^{-1})=\alpha_{hg^{-1}}(a(r(gh^{-1})))=\alpha_h(\alpha_{g^{-1}}(a(r(g)))=\alpha_h(i_A(a)(g)).$$
	If $g\notin H$ we have $i_A(a)(gh^{-1})=0_{A_{r(h)}}=\alpha_h(i_A(a)(g))$.
	It remains to verify that $gH\mapsto \norm{i_A(a)(g)}$ vanishes at infinity. Given $\varepsilon>0$ there exists a compact subset $K\subseteq H^{(0)}$ such that $\norm{a(u)}<\varepsilon$ for all $u\notin K$. Let $C$ be the image of $K$ in the quotient space $G/H$. Now if $gH\notin C$, then either $g\in G\setminus H$, in which case $\norm{i_A(a)(g)}=0$, or $g\in H$, in which case $r(g)H=gH\notin C$. But then $r(g)\notin K$, which implies $\norm{i_A(a)(g)}=\norm{a(r(g))}<\varepsilon$.
	It is straightforward to see that $i_A$ is an $H$-equivariant isometric $\ast$-homo\-morphism.
	%For the $H$-equiva\-riance we compute for all $h\in H$ and $g\in H^{r(g)}$:
%	$\beta_h((i_A)_{d(h)}(a(d(h))))(g)=(i_A)_{d(h)}(a(d(h)))(h^{-1}g)=\alpha_{g^{-1}h}(a(d(h)))=(i_A)_{r(h)}(\alpha_h(a(d(h))))(g).$
\end{proof}

Let us proceed with the construction of the compression homomorphism:
Consider an étale, locally compact Hausdorff groupoid $G$ with an étale subgroupoid $H\subseteq G$. Let $X:=G_{H^{(0)}}$ and $G':=G_{H^{(0)}}^{H^{(0)}}$. Suppose, that $H$ is clopen in $G'$. Now if $A$ is an $H$-algebra and $B$ is a $G$-algebra we define the compression homomorphism $$\textsf{comp}_H^G:\KK^G(Ind_H^XA,B)\rightarrow \KK^H(A,B_{\mid H})$$
as the composition
\begin{center}
	\begin{tikzpicture} %[description/.style={fill=white,inner sep=2pt}]
	\matrix (m) [matrix of math nodes, row sep=3em,
	column sep=2em, text height=1em, text depth=0.25em]
	{ \KK^G(Ind_H^XA,B) & \KK^H(Ind_H^{G'}A,B_{\mid H})\\
		& \KK^H(A,B_{\mid H})
		\\ };
	\path[->,font=\scriptsize]
	%horizontal arrows
	(m-1-1) edge node[auto] {$ \textsf{res}^G_H  $} (m-1-2)
	(m-1-2) edge node[auto] {$ i_A^*  $} (m-2-2)
	(m-1-1) edge [dashed,->] node [below] { $ \textsf{comp}_H^G $} (m-2-2)
	;
	\end{tikzpicture}
\end{center}
Here $\textsf{res}^G_H$ is the homomorphism induced by the inclusion map $H\hookrightarrow G$ (cf. \cite[Proposition~7.1]{LeGall99}), and $i_A$ is the inclusion map from Lemma \ref{Lem:InclusionIntoInducedAlgebra}.
We are now proceeding to prove the main theorem of this section:
\begin{satz}\label{CompressionIsomorphism}
	Let $G$ be an étale locally compact Hausdorff group\-oid with a clopen, proper subgroupoid $H\subseteq G$. Let $X:=G_{H^{(0)}}$. If $A$ is an $H$-algebra and $B$ is a $G$-algebra, then
	$$\textsf{comp}_H^G:\KK^G(Ind_H^XA,B)\rightarrow \KK^H(A,B_{\mid H})$$
	is an isomorphism.
\end{satz}
In order to prove the above theorem we will construct an inverse.
Let $(E,\Phi,T)$ be a Kasparov triple representing an element in the group
$\KK^{H}(A,B_{\mid H})$ and let $V$ denote the unitary operator implementing the action of $H$ on $E$. Since $H$ is proper, we can assume that $T$ is $H$-equivariant by Proposition \ref{proper groupoids: Operators can be chosen equivariant}.
Consider the complex vector space $\widetilde{E}_c$ consisting of bounded continuous sections $\xi:X\rightarrow d_{\mid X}^*(\mathcal{E})$ such that
\begin{itemize}
	\item $\xi(gh^{-1})=V_{h}(\xi(g))$ for all $g\in X$ and $h\in H$ with $d(g)=d(h)$, and
	\item  the map $gH\mapsto \norm{\xi(g)}$ has compact support in $X/H$.
\end{itemize}
Then $\widetilde{E}_c$ becomes a $G$-equivariant pre-Hilbert $B$-module as follows. Using the identification $B\cong \Gamma_0(G^{(0)},\mathcal{B})$ we define a $B$-valued inner product by letting
\[\lk\xi,\eta\rk_B(u):=\sum\limits_{gH\in X^u/H} \beta_g(\lk \xi(g),\eta(g)\rk_{B_{d(g)}}).\]
The second condition on the elements of $\widetilde{E}_c$ guarantees that the sum in the formula above is finite (since $X^u/H$ is discrete).
Let us check that $\lk\xi,\eta\rk_B$ defines an element in $\Gamma_c(G^{(0)},\mathcal{B})$:
Consider the map $$gH\mapsto \beta_g(\lk \xi(g),\eta(g)\rk_{B_{d(g)}}.$$ This map is clearly continuous and hence an element in $\Gamma(X/H,\tilde{r}^*(\mathcal{B}))$, where $\tilde{r}:X/H\rightarrow G^{(0)}$ is the map induced by the restriction of the range map of $G$ to $X$. Moreover, its support is easily checked to be contained in the intersection of the compact supports of the maps $gH\mapsto \norm{\xi(g)}$ and $gH\mapsto\norm{\eta(g)}$, and hence compact as well. Thus, our claim follows from the following Lemma:

\begin{lemma}
	Let $G,H,X$ be as above and $f\in C_c(X/H)$. Then the map 
	$$u\mapsto \sum\limits_{gH\in X^u/H} f(gH)$$
	is continuous.
\end{lemma}
\begin{proof}
	For this we only need to note, that the map $\tilde{r}:X/H\rightarrow G^{(0)}$ is a local homeomorphism. Then the same proof, that shows continuity for the system of counting measures on an étale groupoid (see \cite{Paterson1999}), gives the desired result. But if $U$ is an open $r$-section of $G$, then $\tilde{r}$ will be a homeomorphism onto an open set, when restricted to the image of $U\cap X$ in $X/H$.
\end{proof}
The right $B$-action on $\widetilde{E}_c$ is defined by the formula
\[(\xi\cdot b)(g):=\xi(g)\beta_{g^{-1}}(b(r(g))).\]
A straightforward computation shows, that $\xi\cdot b$ is indeed an element of $\widetilde{E}_c$ again.
The support of the map $gH\mapsto\norm{(b\cdot \xi)(g)}$ is clearly compact since the support of $\xi$ is.
Let us check that with the above defined inner product and $B$-action $\widetilde{E}_c$ is indeed a pre-Hilbert $B$-module:
It is straightforward to check that the inner product is linear in the second and conjugate linear in the first variable.
Also, we clearly have $\lk \xi,\xi\rk_B\geq 0$ for all $\xi\in\widetilde{E}_c$. Now if $\lk \xi,\xi\rk_B(u)=0$ for all $u\in G^{(0)}$ then $\lk \xi(g),\xi(g)\rk_{B_{d(g)}}=0$ for all $g\in X$ and thus $\xi=0$.
Compatibility of the $B$-action with the inner product follows from another straightforward computation.

Let $\widetilde{E}$ be the completion of $\widetilde{E}_c$ with respect to the norm induced by the inner product.

To define the $G$-action on $\widetilde{E}$, let us identify the fibres.
For $u\in G^{(0)}$ consider the complex vector space of bounded continuous sections $\xi:X^u\rightarrow d^*\mathcal{E}$ such that
\begin{itemize}
	\item $\xi(gh^{-1})=V_{h}(\xi(g))$ for all $g\in X^u$ and $h\in H$ such that $d(g)=d(h)$, and
	\item  the map $gH\mapsto \norm{\xi(g)}$ has compact support in $X^u/H$.
\end{itemize}
We can turn this into a pre-Hilbert $B_u$-module by defining
\[ \lk \xi,\eta\rk_{B_u}:=\sum\limits_{gH\in X^u/H} \beta_g(\lk \xi(g),\eta(g)\rk_{B_{d(g)}})\]
and \[(\xi\cdot b(u))(g):=\xi(g)\cdot \beta_{g^{-1}}(b(u)).\]
Let $F_u$ denote the completion with respect to this inner product. Similar to the case of induced algebras one verifies, that 
for $u\in G^{(0)}$ the restriction map $\textsf{res}:\widetilde{E}_c\rightarrow F_u$, $\xi\mapsto \xi_{\mid X^u}$ factors through an isomorphism between the Hilbert $B_u$-modules $\widetilde{E}_u$ and $F_u$.

Let us now define the $G$-action on $\widetilde{E}$: For $g\in G$ define an operator $\widetilde{V}_g\in L(\widetilde{E}_{d(g)},\widetilde{E}_{r(g)})$ by
\[ (\widetilde{V}_g\xi)(s):=\xi(g^{-1}s)\ \forall s\in X^{r(g)}.\]
With this action $\widetilde{E}$ is a $G$-equivariant Hilbert $B$-module.

Define a $*$-homo\-morphism $\widetilde{\Phi}:Ind_H^X A\rightarrow L(\widetilde{E})$ by the formula
\[(\widetilde{\Phi}(f)\cdot \xi)(g):=\Phi_{d(g)}(f(g))\cdot\xi(g).\]
Last but not least define an operator $\widetilde{T}\in L(\widetilde{E})$ by the formula
\[(\widetilde{T}\xi)(g)=T_{d(g)}(\xi(g)).\]

To see that $\widetilde{T}$ is bounded and hence extends to an operator on $\widetilde{E}$ recall the following two general facts:
\begin{enumerate}
	\item If $a,b\in A$ are positive elements with $a\leq b$, then  $\norm{a}\leq\norm{b}$.
	\item If $E$ is a right Hilbert $A$-module, then $$\lk Tx,Tx\rk_A\leq \norm{T}^2\lk x,x\rk_A$$ for all $x\in E$ and $T\in L(E)$ (see \cite[Corollary~2.22]{Morita}).
\end{enumerate}
Because of the above facts and using that the positive elements form a cone we have that 
\[\norm{\sum\limits_{gH}\beta_g(\lk T_{d(g)}(\xi(g)),T_{d(g)}(\xi(g))\rk_{B_{d(g)}})}\leq \norm{\sum\limits_{gH}\beta_g(\norm{T_{d(g)}}^2\lk \xi(g),\xi(g)\rk)},\]
where the sum is over $X^u/H$.
Thus, we can compute:
\begin{align*}
\norm{\widetilde{T}\xi}^2 & =\norm{\lk \widetilde{T}\xi,\widetilde{T}\xi\rk_B}\\
& =\sup\limits_{u\in G^{(0)}}\norm{\lk \widetilde{T}\xi,\widetilde{T}\xi\rk_B(u)}\\
& = \sup\limits_{u\in G^{(0)}}\norm{\sum\limits_{gH\in X^u/H}\beta_g(\lk T_{d(g)}(\xi(g)),T_{d(g)}(\xi(g))\rk_{B_{d(g)}})}\\
& \leq \sup\limits_{u\in G^{(0)}}\norm{\sum\limits_{gH\in X^u/H}\beta_g(\norm{T_{d(g)}}^2\lk \xi(g),\xi(g)\rk)}\\
& \leq \norm{T}^2 \sup\limits_{u\in G^{(0)}}\norm{\sum\limits_{gH\in X^u/H}\beta_g(\lk \xi(g),\xi(g)\rk)}\\
& =\norm{T}^2 \sup\limits_{u\in G^{(0)}} \norm{\lk \xi,\xi\rk_B(u)}\\
& =\norm{T}^2\norm{\xi}^2
\end{align*}
Hence $\widetilde{T}$ extends to a bounded operator on $\widetilde{E}$. It is clearly adjointable with $(\widetilde{T})^*=\widetilde{T^*}$.
We want to show that $(\widetilde{E},\widetilde{\Phi},\widetilde{T})$ is a $G$-equivariant Kasparov-triple for $Ind_H^X A$ and $B$.
To this end we will need some helpful Lemmas. Note that for every $u\in G^{(0)}$ we also have a homomorphism
$$ i_A^u:A\rightarrow Ind_H^{X^u}A$$
from $A$ into each fibre of $Ind_H^X A$, given by the same formulas as $i_A$. Here, continuity of $i_A^u(a)$ is not a problem as $X^u$ carries the discrete topology.
\begin{lemma}
	Let $u\in G^{(0)}$. Consider the set $$A_0=\lbrace\sum\limits_{i=1}^{n}\tilde{\alpha}_{g_i}(i_A^{d(g_i)}(a_i))\mid n\in\NN, g_i\in X^u, a_i\in A\rbrace,$$
	where $\tilde{\alpha}$ is the action of $G$ on $Ind_H^X A$ defined in Proposition \ref{Prop:Action on Induced Algebra}.
	Then $A_0$ is dense in $Ind_H^{X^u}A$.
\end{lemma}
\begin{proof}
	We want to apply Proposition \ref{Prop:DensityCriterionC(X)-algebras} to $A_0$.
	To this end let us first note that $A_0$ is a linear subspace of $Ind_H^{X^u}A$ and moreover it is $C_0(X^u/H)$-invariant. To see this let $a\in A$, $g\in X^u$ and $\varphi\in C_0(X^u/H)$. Then for every $s\in X^u$ such that $g^{-1}s\in H$ we have $gH=sH$ and can compute:
	\begin{align*}
	(\varphi\cdot (\tilde{\alpha}_g(i_A^{d(g)}(a))))(s)&=\varphi(sH)\tilde{\alpha}_g(i_A^{d(g)}(a))(s)\\
	& = \varphi(gH)i_A^{d(g)}(a)(g^{-1}s)\\
	& = i_A^{d(g)}(\varphi(gH)a)(g^{-1}s)\\
	& = \tilde{\alpha}_g(i_A^{d(g)}(\varphi(gH)a))(s)
	\end{align*}
	Since $i_A(a)$ vanishes if $g^{-1}s$ is not in $H$, we can conclude:
	$$ \varphi\cdot (\tilde{\alpha}_g(i_A^{d(g)}(a)))=\tilde{\alpha}_g(i_A^{d(g)}(\varphi(gH)a))\in A_0.$$
	So to see that $A_0$ is dense we just need to show that for any fixed $g\in X^u$ we have $\lbrace f(g)\mid f\in A_0\rbrace=A_{d(g)}(\cong (Ind_H^{X^u}A)_{gH})$.
	But since for any $a\in A$ we have $\tilde{\alpha}_g(i_A^{d(g)}(a))(g)=i_A(a)(g^{-1}g)=i_A(a)(d(g))=\alpha_{d(g)}(a(d(g)))=a(d(g))$ this is obvious.
\end{proof}

Next, we use this result to identify a nice dense subset of the whole algebra $Ind_H^{X}A$. For this write $Ind_H^X\mathcal{A}$ for the upper semi-continuous $C^*$-bundle associated to the $C_0(G^{(0)})$-algebra $Ind_H^X A$.
Let us recall some notation: For $\varphi\in C_c(G)$ and $a\in A$ we can define $\varphi\otimes i_A(a)\in \Gamma_c(G,d^* (Ind_H^X \mathcal{A}))$ by
$$(\varphi\otimes i_A(a))(g)=\varphi(g)i_A(a)_{\mid X^{d(g)}}=\varphi(g)i_A^{d(g)}(a).$$
Furthermore, let 
$$\lambda:\Gamma_c(G,r^*(Ind_H^X \mathcal{A}))\rightarrow Ind_H^{X} A$$ be the continuous map from Lemma \ref{r!} given by the formula
$$\lambda(f)(u)=\sum\limits_{g\in G^u}f(g),\ \forall  u\in G^{(0)}.$$ 
\begin{lemma}
	The set 
	$$\Gamma=\lbrace \lambda(\tilde{\alpha}(\varphi\otimes i_A(a)))\mid a\in A, \varphi\in C_c(G)\rbrace$$
	is dense in $Ind_H^XA$.
\end{lemma}
\begin{proof}
	First we note that $\Gamma$ is a $C_0(G^{(0)})$-invariant linear subspace of $Ind_H^XA$, since for $\psi\in C_0(G^{(0)})$ we have
	\begin{align*}
	\psi\cdot \lambda(\tilde{\alpha}(\varphi\otimes i_A(a)))(u)& =\sum\limits_{g\in G^u}\psi(u)\varphi(g)\tilde{\alpha}_g(i_A^{d(g)}(a))\\
	& = \sum\limits_{g\in G^u}(\psi\otimes \varphi)(g)\tilde{\alpha}_g(i_A^{d(g)}(a))\\
	& =\lambda(\tilde{\alpha}(\psi\otimes \varphi)\otimes i_A(a))(u),
	\end{align*}
	where $\psi\otimes\varphi\in C_c(G)$ is given by $(\psi\otimes\varphi)(g)=\psi(r(g))\varphi(g)$.
	Then note that for fixed $u\in G^{(0)}$ we have
	$A_0\subseteq \lbrace \lambda(\tilde{\alpha}(\varphi\otimes i_A(a)))(u)\mid \varphi\in C_c(G), a\in A\rbrace\subseteq Ind_H^{X^u}A$.
	By the previous lemma $A_0$ is dense in $Ind_H^{X^u}A$ and thus, so is the middle set. 
	Consequently, $\Gamma$ is dense in $Ind_H^XA$ by yet another application of Proposition \ref{Prop:DensityCriterionC(X)-algebras}.
\end{proof}
We are now prepared for:
\begin{lemma}
	$(\widetilde{E},\widetilde{\Phi},\widetilde{T})\in \EE^G(Ind_{H}^XA,B)$.
\end{lemma}
\begin{proof}
	As a first step we check that $\widetilde{T}$ is $G$-equivariant. For this note that for $u\in G^{(0)}$ the operator $\widetilde{T}_u:\widetilde{E}_u\rightarrow\widetilde{E}_u$ is given by the same formula as $\widetilde{T}$ itself. Thus for all $g\in G$, $\xi\in \widetilde{E}_{d(g)}$ and $s\in X^{r(g)}$ we can compute:
	\begin{align*}
	(\widetilde{T}_{r(g)}\widetilde{V}_g \xi)(s) & = T_{d(s)}((\widetilde{V}_g\cdot \xi)(s))\\
	& = T_{d(s)}(\xi(g^{-1}s))\\
	& = (\widetilde{T}\xi)(g^{-1}s)\\
	& = (\widetilde{V}_g \widetilde{T}_{d(g)}\xi)(s)
	\end{align*}
	Consequently, it is enough to check that $[\widetilde{T},\widetilde{\Phi}(f)]$, $(\widetilde{T}^2-1)\widetilde{\Phi}(f)$ and $(\widetilde{T}^*-\widetilde{T})\widetilde{\Phi}(f)$ are compact operators on $\widetilde{E}$ for all $f\in Ind_{H}^X A$. We will do this in two steps:\\
	\underline{Step 1: $f=i_A(a)$}:\\
	For this we note that there is an embedding $i_E:E\hookrightarrow \widetilde{E}$ of $E$ as a direct summand of $\widetilde{E}$ given by the formula
	\[  i_E(e)(g)=\left\{\begin{array}{ll} V_{g^{-1}}(e(r(g))) & ,g\in H \\
	0_{d(g)} & ,else \end{array}\right\}.\]
	This embedding induces a corresponding embedding $i_{K(E)}:K(E)\rightarrow K(\widetilde{E})$. 
	By checking on rank-one operators and going through the formulas we can see that for $F\in K(E)$ we have the following equation:
	\[(i_{K(E)}(F) \xi)(g)=\left\{\begin{array}{ll} (V_{g^{-1}}F_{r(g)}V_g) \cdot \xi(g) & ,g\in H \\
	0_{d(g)} & ,else \end{array}\right\}\]
	
	Note also that for $a\in A$ we have $(\widetilde{\Phi}(i_A(a))\xi)(g)=0$ if $g\notin H$.
	For $g\in H$ we can use the $H$-equivariance of $T$ to compute:
	\begin{align*}
	(i_{K(E)}([T,\Phi(a)])\xi)(g) & = (V_{g^{-1}}[T_{r(g)},\Phi_{r(g)}(a(r(g)))]V_g) (\xi(g))\\
	& = [T_{d(g)},\Phi_{d(g)}(\alpha_{g^{-1}}(a(r(g))))] \xi(g)\\
	& = [T_{d(g)},\Phi_{d(g)}(i_A(a)(g))] \xi(g)\\
	& = ([\widetilde{T},\widetilde{\Phi}(i_A(a))]\xi)(g).
	\end{align*}
	Consequently, we have $i_{K(E)}([T,\Phi(a)])=[\widetilde{T},\widetilde{\Phi}(i_A(a))]$ for all $a\in A$.
	Similar computations show that $i_{K(E)}((T^2-1)\Phi(a))=(\widetilde{T}^2-1)\widetilde{\Phi}(i_A(a))$ and $i_{K(E)}((T-T^*)\Phi(a))=(\widetilde{T}-\widetilde{T}^*)\widetilde{\Phi}(i_A(a))$ for all $a\in A$.\\
	\underline{Step 2: $f=\lambda(\tilde{\alpha}(\varphi\otimes i_A(a)))$}\\
	Since $(\widetilde{T}-\widetilde{T}^*)\widetilde{\Phi}(i_A(a))\in K(\widetilde{E})$ by the first step, we have
	$$ \widetilde{V}(\varphi\otimes(\widetilde{T}-\widetilde{T}^*)\widetilde{\Phi}(i_A(a)))\widetilde{V}^*\in \Gamma_c(G,r^*\mathcal{K}(\widetilde{E}))$$
	for all $\varphi\in C_c(G)$ and hence
	$$\lambda (\widetilde{V}(\varphi\otimes(\widetilde{T}-\widetilde{T}^*)\widetilde{\Phi}(i_A(a)))\widetilde{V}^*)\in K(\widetilde{E})$$ by Lemma \ref{r!}.
	Let us show that 
	$$(\widetilde{T}-\widetilde{T}^*)\widetilde{\Phi}(\lambda(\tilde{\alpha}(\varphi\otimes i_A(a))))=\lambda (\widetilde{V}(\varphi\otimes(\widetilde{T}-\widetilde{T}^*)\widetilde{\Phi}(i_A(a)))\widetilde{V}^*))$$
	For $f=\lambda(\tilde{\alpha}(\varphi\otimes i_A(a)))$ we compute:
	\begin{align*}
	((\widetilde{T}-\widetilde{T}^*)&\widetilde{\Phi}(f)\cdot \xi)(s)\\
	 & = (T_{d(s)}-T_{d(s)}^*)\Phi_{d(s)}(\lambda(\tilde{\alpha}(\varphi\otimes i_A(a)))(s))\cdot \xi(s)\\
	& = \sum\limits_{g\in G^{r(s)}} \varphi(g) (T_{d(s)}-T_{d(s)}^*)\Phi_{d(s)}(\tilde{\alpha}_g(i_A^{d(g)}(a)(s)))\xi(s)\\
	& = \sum\limits_{g\in G^{r(s)}} \varphi(g) ((\widetilde{T}_{r(g)}-\widetilde{T}_{r(g)}^*)\widetilde{\Phi}_{r(g)}(\tilde{\alpha}_g(i_A^{d(g)}(a))\xi))(s)\\
	& = \sum\limits_{g\in G^{r(s)}}  \widetilde{V}_g\varphi(g)(\widetilde{T}_{d(g)}-\widetilde{T}_{d(g)}^*)\widetilde{\Phi}_{d(g)}(i_A^{d(g)}(a))\widetilde{V}^*_g\xi)(s)\\
	& =\sum\limits_{g\in G^{r(s)}} ((  (\widetilde{V}(\varphi\otimes(\widetilde{T}-\widetilde{T}^*)\widetilde{\Phi}(i_A(a)))\widetilde{V}^*))(g)\xi)(s)\\
	& = \left(\lambda (\widetilde{V}(\varphi\otimes (\widetilde{T}-\widetilde{T}^*)\widetilde{\Phi}(i_A(a)))\widetilde{V}^*)(r(s))\cdot \xi_{\mid r(s)}\right)(s)\\
	& = (\lambda (\widetilde{V}(\varphi\otimes (\widetilde{T}-\widetilde{T}^*)\widetilde{\Phi}(i_A(a)))\widetilde{V}^*)\cdot \xi)(s)
	\end{align*}
	Similarly, we compute
	$$[\widetilde{T},\widetilde{\Phi}(\lambda(\tilde{\alpha}(\varphi\otimes i_A(a))))]=\lambda(\widetilde{V}(\varphi\otimes [\widetilde{T},\widetilde{\Phi}(i_A(a))])\widetilde{V}^*)$$
	and 
	$$(\widetilde{T}^2-1)\widetilde{\Phi}(\lambda(\tilde{\alpha}(\varphi\otimes i_A(a))))=\lambda (\widetilde{V}(\varphi\otimes(\widetilde{T}^2-1)\widetilde{\Phi}(i_A(a)))\widetilde{V}^*)).$$
	From the previous lemma we know that elements of the form $\lambda(\tilde{\alpha}(\varphi\otimes i_A(a)))$ form a dense subset of $Ind_H^X A$ and thus the result follows by continuity.
\end{proof}
Applying the same constructions to a homotopy we conclude that the mapping $(E,\Phi,T)\mapsto (\widetilde{E},\widetilde{\Phi},\widetilde{T})$ induces a map
in equivariant $\KK$-theory, which we call the inflation map:
$$\textsf{inf}_H^G:\KK^H(A,B_{\mid H})\rightarrow \KK^G(Ind_H^X A,B)$$

\begin{proof}[Proof of Theorem \ref*{CompressionIsomorphism}]
	As a first step we claim that the result is invariant under passing to a Morita-equivalent algebra in the first variable. Indeed if $A'$ is Morita-equivalent to $A$ and if we let $x\in \KK^H(A',A)$ be the corresponding invertible $\KK^H$-element, the claim will follow from the commutativity of the following diagram:
	\begin{center}
		\begin{tikzpicture} %[description/.style={fill=white,inner sep=2pt}]
		\matrix (m) [matrix of math nodes, row sep=3em,
		column sep=1em, text height=1em, text depth=0.25em]
		{ \KK_*^G(Ind_H^X \ A,B) & \KK_*^H(A,B_{\mid H^{(0)}})\\
			\KK_*^G(Ind_H^X\ A',B) & \KK_*^H(A',B)       
			\\ };
		\path[->,font=\scriptsize]
		%horizontal arrows
		(m-1-1) edge node[auto] {$ \textsf{comp}_H^G $} (m-1-2)
		(m-2-1) edge node[auto] {$  \textsf{comp}_H^G $} (m-2-2)
		
		% vertical arrows
		(m-1-1) edge node[left] {$ \textsf{Ind}_H^G(x)\otimes \cdot $} (m-2-1)    
		(m-1-2) edge node[auto] {$ x\otimes \cdot $} (m-2-2)
		;
		\end{tikzpicture}
	\end{center}
	Here $\mathsf{Ind}_H^G(x)$ denotes the image of $x$ under the induction homomorphism
	$$\mathsf{Ind}_H^G: \KK^H(A',A)\rightarrow \KK^G(Ind_H^X A',Ind_H^X A),$$
	from Proposition \ref{Prop:Induction homomorphism}.
	Commutativity of the above diagram follows from the equation
	$$[i_{A'}]\otimes \textsf{res}_H^G(\textsf{Ind}_H^G(x))=x\otimes [i_A],$$
	since then for any $y\in \KK^G(Ind_H^X\ A, B)$ we have
	\begin{align*}
	x\otimes \textsf{comp}_H^G(y)& = x\otimes i_A^*(\textsf{res}_H^G(y))\\
	& = x\otimes [i_A]\otimes \textsf{res}_H^G(y)\\
	& = [i_{A'}]\otimes \textsf{res}_H^G(\textsf{Ind}_H^G(x))\otimes \textsf{res}_H^G(y)\\
	& = [i_{A'}]\otimes \textsf{res}_H^G(\textsf{Ind}_H^G(x)\otimes y)\\
	& = \textsf{comp}_H^G(\textsf{Ind}_H^G(x)\otimes y).
	\end{align*}
	
	We will now show that the inflation map constructed above is inverse to the compression homomorphism. We will begin with the easier direction:
	Let $(E,\Phi,T)$ represent an element in $\KK^{H}(A,B_{\mid H})$. We need to see that $\textsf{comp}_H^G([\widetilde{E},\widetilde{\Phi},\widetilde{T}]) = [E,\Phi,T]$.
	By definition the element $\textsf{comp}_H^G([\widetilde{E},\widetilde{\Phi},\widetilde{T}])$ can be represented by the triple $(\widetilde{E}_{\mid H},\widetilde{\Phi}_{\mid A_{\mid H}}\circ i_A,\widetilde{T}_{\mid \widetilde{E}_{\mid H}})$.
	It is not too hard to see that $\widetilde{E}_{\mid H}$ can be obtained by the same definitions as $\widetilde{E}$ if we just consider bounded continuous functions  $\xi:G'\rightarrow d_{\mid G'}^*\mathcal{E}$, where $G'=G_{H^{(0)}}^{H^{(0)}}$.
	Consider the split-exact sequence coming from the restriction map $\textsf{res}:\widetilde{E}\rightarrow \Gamma_0(H^{(0)},\mathcal{E})\cong E$; $\xi\mapsto \xi_{\mid H^{(0)}}$. The split is then given by the map $i_E$ and thus $\widetilde{E}=i_E(E)\oplus ker(\textsf{res})$. Now for $a\in A$ and $\xi\in ker(\textsf{res})\subseteq \widetilde{E}$ we have
	$$\widetilde{\Phi}(i_A(a))(\xi)(g) = \Phi_{d(g)}(i_A(a)(g))(\xi(g))=0,$$
	since for $g\in H$ we have $\xi(g)=V_{g^{-1}}(\xi(r(g)))=0$ and for $g\in G\setminus H$ we have that $i_A(a)(g)=0$.
	
	On the other hand given $e\in E$, $a\in A$ and $g\in H$ we compute
	\begin{align*}
	(\widetilde{\Phi}(i_A(a))i_E(e))(g)& =\Phi_{d(g)}(i_A(a)(g))(i_E(e)(g))\\
	&=\Phi_{d(g)}(\alpha_{g^{-1}}(a(r(g))))V_{g^{-1}}e(r(g))\\
	& =V_{g^{-1}}\Phi_{r(g)}(a(r(g)))e(r(g))\\
	& =V_{g^{-1}}(\Phi(a)e)(r(g))\\
	&=i_E(\Phi(a)e)(g).
	\end{align*}
	Since both sides are clearly zero for $g\notin H$, we have
	$$\widetilde{\Phi}(i_A(a))i_E(e)=i_E(\Phi(a)e).$$
	Combining these results we get that under the identification $E\cong i_E(E)$ and for all $a\in A$ we have $$\widetilde{\Phi}(i_A(a))(e+\xi)=\Phi(a)(e),$$ and thus $\widetilde{\Phi}\circ i_A$ decomposes as $\Phi\oplus 0$ under the decomposition $\widetilde{E}=i_E(E)\oplus ker(\textsf{res})$. Similar (but even easier) computations yield that $\widetilde{T}=T\oplus \widetilde{T}_{\mid ker(\textsf{res})}$.
	We conclude that
	\begin{align*}
\textsf{comp}_H^G([\widetilde{E},\widetilde{\Phi},\widetilde{T}])&=[(\widetilde{E}_{\mid H},\widetilde{\Phi}_{\mid A_{\mid H}}\circ i_A,\widetilde{T}_{\mid \widetilde{E}_{\mid H}})]\\
&=[(E,\Phi,T)]+\underbrace{[(ker(\textsf{res}),0,\widetilde{T}_{\mid ker(\textsf{res})})]}_{=0}.
	\end{align*}
	This completes the proof of
	$$\textsf{comp}_H^G\circ \textsf{inf}_H^G=id_{\KK^H(A,B_{\mid H})}.$$
	
	For the converse we make use of the first paragraph of this proof and pass to the stabilization $A\otimes_{H^{(0)}} K(L^2(G^{H^{(0)}}))$ of $A$ (if necessary) which is Morita-equivalent to $A$ via the imprimitivity bimodule $L^2(G^{H^{(0)}},A)=L^2(G^{H^{(0)}})\otimes_{C_0(H^{(0)})} A$.
	Using the identification $K(L^2(G))_{\mid H}\cong K(L^2(G^{H^{(0)}}))$, we have a canonical isomorphism
	$$Ind_H^X (A\otimes_{H^{(0)}} K(L^2(G^{H^{(0)}})))\cong(Ind_H^X\ A)\otimes_{G^{(0)}} K(L^2(G))$$
	by Lemma \ref{Lemma:Induction and Tensor product}. Thus, given a representative $(F,\Psi,S)$ of an element in $\KK^G(Ind_H^X\ A,B)$, we may assume that $\Psi$ is essential and $S$ is $G$-equivariant by Proposition \ref{AutomaticEquivariance}.
	
	Since $X^u/H$ is discrete for every $u\in G^{(0)}$ the characteristic function $\chi_{gH}$ is an element in $C_0(X^{r(g)}/H)$ . Using these functions we can define a family of pairwise orthogonal projections $\lbrace p_{gH}\mid gH\in X^u/H\rbrace$ on the Hilbert $Ind_H^{X^u} A$-$B_u$-module $F_u$
	by letting
	$$p_{gH}(\Psi_u(f)e(u))=\Psi_u(\chi_{gH}f)e(u).$$
	Let us check that this definition is actually continuous in $gH$ or in other words, that $gH\mapsto p_{gH}$ defines an element in $L(\tilde{r}^*(F))$:
	
	For this it is enough to show that for each $\varphi\in C_c(X/H)$, $f\in Ind_H^X A$ and $e\in F$ we have that
	$$gH\mapsto P(\varphi\otimes \Psi(f)e)(gH):=\varphi(gH)p_{gH}(\Psi(f)e)$$
	is continuous, since elements of the form $\varphi\otimes \Psi(f)e$ are dense in $\tilde{r}^*(F)$.
	
	By density, it is enough to consider $f\in Ind_H^X A$ such that $gH\mapsto \norm{f(g)}$ has compact support and using a partition of unity argument, we can assume that this support is actually contained in an open set $U\subseteq X/H$ on which $\tilde{r}$ is injective.
	But then for any $gH\in U$ we have $$\chi_{gH}f_{\mid X^{r(g)}}=f_{\mid X^{r(g)}}$$
	since $f_{\mid X^{r(g)}}(x)\neq 0$ implies $xH\in X^{r(g)}\cap U$. But of course we have $gH\in X^{r(g)}\cap U$ as well and since $\tilde{r}(xH)=\tilde{r}(gH)$ we must have $gH=xH$ by injectivity of $\tilde{r}_{\mid U}$.
	Thus, we have $$f_{\mid X^{r(g)}}(x)=\left\{\begin{array}{ll} f(x) & gH=xH \\
	0 & ,else \end{array}\right\}=\chi_{gH}f_{\mid X^{r(g)}}(x).$$
	It follows that $gH\mapsto \chi_{gH}f_{\mid X^{r(g)}}$ is a compactly supported continuous section of the bundle over $X/H$ associated to $Ind_H^X A$.
	Consequently, for each $\varphi\in C_c(X/H)$ and $e\in F$ we have that $$gH\mapsto \varphi(gH)p_{gH}(\Psi_{r(g)}(f)e(r(g)))=\varphi(gH)\Psi_{r(g)}(\chi_{gH}f_{\mid X^{r(g)}})e(r(g))$$ is a compactly supported continuous section of $\tilde{r}^*(\mathcal{F})$, as desired.
	
	It is not hard to check that the following equality holds
	\begin{equation}\label{Projection/Action commutation}
	V_g p_{sH}=p_{gsH} V_g\ \forall (g,s)\in G^{(2)}.
	\end{equation}
	Define an operator $S'$ on $F$ by
	$$S'_u:=\sum\limits_{gH\in X^u/H} p_{gH}S_u p_{gH}$$
	Since for all $e\in F$ and $f\in (Ind_H^X A)_c$ the map
	$$gH\mapsto p_{gH}S_{r(g)}p_{gH}\Psi_{r(g)}(f_{\mid X^{r(g)}})e(r(g)$$ is continuous and compactly supported, integrating against the counting measures on the fibres of $X/H$ yields a well-defined operator $S'\in L(F)$.
	Using equation (\ref{Projection/Action commutation}) from above one easily verifies that $S'$ is still $G$-equivariant but additionally satisfies the relation $S'_{r(g)}p_{gH}=p_{gH}S'_{r(g)}$ for all $g\in X$. We will show that $S'$ is a compact perturbation of $S$ which allows us two assume that any element in $\KK^G(Ind_H^X\ A,B)$
	can be represented by an essential Kasparov triple with an equivariant operator, which commutes with the families of projections defined above.
	
	One easily checks that $$((S-S')\Psi(f))_u=\sum\limits_{gH\in X^u/H} (S_u-p_{gH}S_u)\Psi_u(\chi_{gH}f_{\mid X^u}).$$
	Using compactness of $[S,\Psi(\chi_{gH}f_{\mid X^u})]$ we can see that each summand in the above sum is compact. Then we use our standard argument again that the map $
	gH\mapsto (S_{r(g)}-p_{gH}S_{r(g)})\Psi_{r(g)}(\chi_{gH}f_{\mid X^{r(g)}})$ defines a  continuous section $X/H\rightarrow \widetilde{r}^*(\mathcal{K}(F))$ with compact support and therefore integration with respect to the system of counting measures on $X/H$ yields a continuous section $G^{(0)}\rightarrow \mathcal{K}(F)$, i.e. an element in $K(F)$.
	Now let $\chi_H$ be the characteristic function of the $\pi(H^{(0)})\subseteq X/H$. The set $\pi(H^{(0)})$ is clopen since the pre-image under the quotient map is just $H$, which is clopen in $X$ by assumption. Thus $\chi_H\in C_b(X/H)$.
	Now define a projection $p_H\in L(F)$ on the dense subset $\Psi(Ind_H^X\ A)F\subseteq F$ by
	$$p_H(\Psi(f)e)=\Psi(\chi_H\cdot f)e.$$
	Then $(E,\Phi,T):=(p_H F,p_H \Psi p_H, p_H S p_H)$ is a representative of the element $\textsf{comp}_H^G([F,\Psi,S])$.
	
	Now for $\xi\in \widetilde{E}_c$ and $u\in G^{(0)}$ define an element $\Theta(\xi)$ in $F$ by
	$$\Theta(\xi)(u)=\sum\limits_{gH\in X^u/H} V_g(\xi(g)).$$
	We want to show that this definition extends to a bounded linear map $\Theta:\widetilde{E}\rightarrow F$.
	For this we need the following:
	Whenever $e\in p_H F$ and $g\in G\setminus H$ we can use equation \ref{Projection/Action commutation} to see that $$(p_H)_{r(g)}V_g(e(d(g))=0.$$
	
	If $\xi\in\widetilde{E}_c$ and $g,s\in G^x$ for some $x\in G^{(0)}$ such that $gH\neq sH$, i.e. $s^{-1}g\in G\setminus H$ we have by the above result:
	$$\lk V_g(\xi(g)),V_s(\xi(s))\rk=\lk \underbrace{V_{s^{-1}g}(\xi(g))}_{\in (p_HF)_{d(g)}^\bot},\underbrace{\xi(s)}_{\in (p_H F)_{d(g)}}\rk=0$$
	Now we are ready to prove that $\Theta$ extends to an isometry as follows:
	\begin{align*}
	\norm{\Theta(\xi)}^2 & = \sup_{x\in G^{(0)}} \norm{\lk\Theta(\xi)(x),\Theta(\xi)(x)\rk}\\
	& = \sup_{x\in G^{(0)}} \norm{\sum\limits_{gH}\sum\limits_{sH} \lk V_g(\xi(g)),V_s(\xi(s))\rk}\\
	& = \sup_{x\in G^{(0)}} \norm{\sum\limits_{gH} \lk V_g(\xi(g)),V_g(\xi(g))\rk}\\
	& = \sup_{x\in G^{(0)}} \norm{\sum\limits_{gH} \beta_g(\lk \xi(g),\xi(g)\rk)}\\
	& = \norm{\xi}^2
	\end{align*}
	Let us also check that $\Theta$ is $G$-equivariant:
	\begin{align*}
	V_s(\Theta_{d(s)}(\xi)(d(s))) & = \sum\limits_{gH\in G^{d(s)}_{H^{(0)}}/H} V_{sg}(\xi(g)) &\\
	& = \sum\limits_{gH\in G^{r(s)}_{H^{(0)}}/H} V_{g}(\xi(s^{-1}g)) & (gH\mapsto s^{-1}gH)\\
	& = \sum\limits_{gH\in G^{r(s)}_{H^{(0)}}/H} V_{g}(\widetilde{V}_s(\xi)(g)) & \\
	& = (\Theta_{r(s)}(\widetilde{V}_s(\xi))) (r(s)) &
	\end{align*}
	Similarly, we can show that $\Theta$ intertwines $\widetilde{\Phi}$ with $\Psi$ and $\widetilde{T}$ with $S$.
	Now if $e\in F$ is arbitrary we can define $\xi\in\widetilde{E}$ by letting
	$$\xi(g)=(p_H)_{d(g)} V_{g^{-1}}\cdot e(r(g)).$$
	Then we can compute $\Theta(\xi)(x)=\sum_{gH} V_g((p_H)_{d(g)} V_{g^{-1}}\cdot e(x))=e(x).$ This completes the proof that
	\begin{align*}\textsf{inf}_H^G(\textsf{comp}_H^G([F,\Psi,S]))&=\textsf{inf}_H^G([E,\Phi,T])\\
	&=[\widetilde{E},\widetilde{\Phi},\widetilde{T}]=[F,\Psi,S].
	\end{align*}
\end{proof}
In the next section, we shall also need the following compatibility property of the compression homomorphism with respect to taking right Kasparov products:
\begin{lemma}\label{Lemma:Compression and Kasparov Product}
	Let $G$ be a second countable étale groupoid, $H\subseteq G$ a proper open subgroupoid and let $X:=G_{H^{(0)}}$. Let $A$ be an $H$-algebra and let $B$ and $B'$ be two $G$ algebras. Then, for every $x\in \mathrm{KK}^G(B,B')$ we have a commutative diagram:
	\begin{center}
		\begin{tikzpicture}[description/.style={fill=white,inner sep=2pt}]
		\matrix (m) [matrix of math nodes, row sep=3em,
		column sep=2.5em, text height=1.5ex, text depth=0.25ex]
		{ \mathrm{KK}^G(Ind_H^X A,B) & \mathrm{KK}^G(Ind_H^X A,B') \\
			\mathrm{KK}^H(A,B_{\mid H}) & \mathrm{KK}^H(A,B'_{\mid H}) \\
		};
		\path[->,font=\scriptsize]
		(m-1-1) edge node[auto] {$ \cdot \otimes x $} (m-1-2)
		(m-1-2) edge node[auto] {$ \textsf{comp}_H^G $} (m-2-2)
		(m-2-1) edge node[auto] { $ \cdot\otimes res_H^G(x) $ } (m-2-2)
		(m-1-1) edge node[auto] { $ \textsf{comp}_H^G $ } (m-2-1)
		;
		\end{tikzpicture}
	\end{center}
\end{lemma}
\begin{proof}
Using the definition of the compression homomorphism, it is enough to prove, that the following diagram commutes:
\begin{center}
	\begin{tikzpicture}[description/.style={fill=white,inner sep=2pt}]
	\matrix (m) [matrix of math nodes, row sep=3em,
	column sep=2.5em, text height=1.5ex, text depth=0.25ex]
	{ \mathrm{KK}^G(Ind_H^X A,B) & \mathrm{KK}^H(Ind_H^{G'} A,B_{\mid H}) & \mathrm{KK}^H(A,B_{\mid H}) \\
		\mathrm{KK}^G(Ind_H^X A,B') & \mathrm{KK}^H(Ind_H^{G'} A,B'_{\mid H}) & \mathrm{KK}^H(A,B'_{\mid H}) \\
	};
	\path[->,font=\scriptsize]
	(m-1-1) edge node[auto] {$ res_H^G $} (m-1-2)
	(m-1-2) edge node[auto] {$ \cdot\otimes res_H^G(x) $} (m-2-2)
	(m-2-1) edge node[auto] { $ res_H^G $ } (m-2-2)
	(m-1-1) edge node[auto] { $ \cdot \otimes x $ } (m-2-1)
	(m-1-2) edge node[auto] {$ i_A^* $} (m-1-3)
	(m-2-2) edge node[auto] {$ i_A^* $} (m-2-3)
	(m-1-3) edge node[auto] {$ \cdot\otimes res_H^G(x) $} (m-2-3)
	;
	\end{tikzpicture}
\end{center}
Commutativity of the diagram on the right follows from the associativity of the Kasparov product. Using the fact that the map $res_H^G$ is given by pulling back along the inclusion map $\iota:H\hookrightarrow G$, commutativity of the left diagram follows from \cite[Proposition~6.1.3]{LeGall}.
\end{proof}
\section{The Going-Down Principle}
In this section we state and prove the Going-Down (or restriction) principle for ample groupoids. After reminding the reader about universal spaces for proper actions of groupoids and the formulation of the Baum-Connes conjecture following \cite{MR1798599}, we first prove a special case of the restriction principle (see Theorem \ref{MainTheorem}), that can be applied directly in many cases. We then extend the formalism of Going-Down functors as in \cite{CEO} to our setting and state the main results in full generality.

Recall that a proper $G$-space $Z$ is called a \textit{universal proper $G$-space}, if for every proper $G$-space $X$ there exists a continuous $G$-equivariant map $\varphi:X\rightarrow Z$ which is unique up to $G$-equivariant homotopy.
Note that a universal proper $G$-space $Z$ as in the definition above is unique up to $G$-equivariant homotopy equivalence.
A priori it is not clear that a universal proper $G$-space always exists. 
But by combining several results of Tu (\cite[Proposition~6.13, Lemma~6.14]{Tu99} and \cite[Proposition~11.4]{Tu98}) we obtain that every second countable étale groupoid $G$ admits a locally compact universal proper $G$-space.

Recall, that a $G$-space $X$ is called \textit{$G$-compact (or cocompact)} if there exists a compact subset $K\subseteq X$, such that $X=GK$. We need the following elementary fact, whose proof we omit.
%\begin{lemma}
%Let $G$ be a groupoid and $Y$ be a proper $G$-space. If $X\subseteq Y$ is a $G$-compact subset, then $X$ is closed in $Y$.
%\end{lemma}
%\begin{proof}
%Let $(x_\lambda)_\lambda$ be a net in $X$ with $x_\lambda\rightarrow y$. Since $X$ is $G$-compact, there exists a compact set $K\subseteq X$ such that $X=GK$. Hence we can write $x_\lambda=g_\lambda k_\lambda$ for $g_\lambda\in G$ and $k_\lambda\in K$. Using the compactness of $K$, we can pass to a subnet and relabel to assume that $k_\lambda\rightarrow k$ for some $k\in K$. But $Y$ is a proper $G$-space, so by Proposition \ref{Prop:CharacterizationsProperAction} we can pass to another subnet and relabel to assume that also $g_\lambda$ converges to some $g\in G$. Hence we have $y=\lim x_\lambda=\lim g_\lambda k_\lambda=gk\in X$ as desired.
%\end{proof}

\begin{lemma}
Let $G$ be a locally compact Hausdorff groupoid. Furthermore, let $X$ be a $G$-compact $G$-space and $Y$ be a proper $G$-space. Then every $G$-equivariant continuous map $\varphi:X\rightarrow Y$ is automatically proper.
\end{lemma}
%\begin{proof}
%Let $K\subseteq Y$ be a compact subset. Our goal is to show that $\varphi^{-1}(K)$ is compact. To this end let $(x_\lambda)_\lambda$ be a net in $\varphi^{-1}(K)$. We claim that $(x_\lambda)_\lambda$ has a convergent subnet. Since $K$ is compact, we can pass to a subnet to assume that $\varphi(x_\lambda)\rightarrow y$ for some $y\in K\subseteq Y$. Next, we use the $G$-compactness of $X$ to find a compact subset $C\subseteq X$ such that $X=GC$. Hence, we may write $x_\lambda=g_\lambda c_\lambda$ for some $g_\lambda\in G$ and $c_\lambda\in C$. Passing to a subnet again, we may assume that $c_\lambda$ converges to some element $c\in C$ (using compactness of $C$). Using the continuity of $\varphi$ we have $\varphi(c_\lambda)\rightarrow \varphi(c)$. Since $\varphi$ is $G$-equivariant we also get $g_\lambda \varphi(c_\lambda)=\varphi(x_\lambda)\rightarrow y$.
%Now we can use properness of $Y$ (see Proposition \ref{Prop:CharacterizationsProperAction} $(4)$) to pass to yet another subnet and relabel, allowing us to assume that $g_\lambda\rightarrow g$ for some $g\in G$. But then we have $x_\lambda=g_\lambda c_\lambda\rightarrow gc$ proving our claim.
%\end{proof}

Let $\mathcal{E}(G)$ denote a universal proper $G$-space. Then, applying the above lemma, for any two $G$-compact  subsets $X_1\subseteq X_2\subseteq \mathcal{E}(G)$ we have a canonical $*$-homo\-morphism $C_0(X_2)\rightarrow C_0(X_1)$ given by restriction.
This homomorphism in turn induces a map
$$\KK^G(C_0(X_1),A)\rightarrow \KK^G(C_0(X_2),A)$$
for every $G$-algebra $A$. Thus, the following definition makes sense:
\begin{defi}
Let $G$ be an étale, second countable Hausdorff groupoid and $A$ be a $G$-algebra.
The \textit{topological $K$-theory of $G$ with coefficients in $A$} is defined as
$$\K_*^{\mathrm{top}}(G;A):=\varinjlim \KK_*^G(C_0(X),A),$$
where the direct limit is taken over all $G$-compact, locally compact and second countable subsets $X\subseteq \mathcal{E}(G)$.
\end{defi}

Next, we want to define the Baum-Connes assembly map. We shall need the following well-known result.
\begin{lemma}
Let $G$ be a proper étale groupoid with compact orbit space $G\setminus G^{(0)}$ and let $c:G^{(0)}\rightarrow \RR^+$ be a compactly supported cutoff function for $G$. Then the function $p_c:G\rightarrow \CC$, $g\mapsto \sqrt{c(d(g))c(r(g))}$ defines a projection in $C_r^*(G)$. Moreover the class $[p_c]\in \K_0(C_r^*(G))=\KK(\CC,C_r^*(G))$ does not depend on the choice of the cutoff function $c$.
\end{lemma}

We are now in the position to define the Baum-Connes assembly map: Let $A$ be a $G$-algebra. For every $G$-compact subspace $X\subseteq \mathcal{E}(G)$ we can consider the composition
$$\mu_X:\KK^G_*(C_0(X),A)\stackrel{j_G}{\rightarrow} \KK_*(C_r^*(G\ltimes X),A\rtimes_r G)\stackrel{[p_c]\otimes\cdot}{\rightarrow}\KK_*(\CC,A\rtimes_r G)$$
where $j_G$ is the descent homomorphism. Note, that we also used the identification $C_0(X)\rtimes_r G\cong C_r^*(G\ltimes X)$.
One easily checks, that the maps $\mu_X$ give rise to a well-defined homomorphism
$$\mu_A:\K_*^{\mathrm{top}}(G;A)\rightarrow \KK_*(\CC,A\rtimes_r G)=\K_*(A\rtimes_r G).$$
This is the \textit{Baum-Connes assembly map} for $G$ with coefficients in $A$.

Let us now turn to the Going-Down principle:
Let $P(G)$ denote the subset of all probability measures in $M(G)$, the space of all finite positive Radon measures on $G$. Recall, that for a measure $\mu\in M(G)$ the support of $\mu$ is defined as
\[supp(\mu)=\lbrace g\in G\mid \mu(U)>0 \textit{ for each open neighbourhood U of g}\rbrace.\]
Since we are working with the weak-*-topology, a description in terms of continuous functions with compact support would be much more convenient. Such a description is given by the following lemma.
\begin{lemma}\label{CharacterizationOfBeingInTheSupportOfaMeasure}
For $\mu\in M(G)$ and $g\in G$ we have that $g\in supp(\mu)$ if and only if $I_\mu(\varphi)>0$ for each $\varphi\in C_c^+(G)$ such that $\varphi(g)>0$.
\end{lemma}
\begin{proof}
Let $g\in supp(\mu)$ and $\varphi\in C_c^+(G)$ such that $\varphi(g)>0$. Find a $\varphi(g)>\varepsilon>0$. Since $\varphi$ is continuous we can find a neighbourhood $U$ of $g$ such that $\varphi(h)>\varepsilon$ for all $h\in U$. If we define $c:=\frac{1}{2}\inf\lbrace\varphi(x)\mid x\in U\rbrace>0$ then $c\chi_U\leq \varphi$ and thus $0<c\mu(U)=\int_G c\chi_Ud\mu \leq I_\mu(\varphi)$.

For the converse let $U\subseteq G$ be an open neighbourhood of an element $g\in G$. Pick a function $\varphi\in C_c^+(G)$ with $0\leq \varphi\leq 1$, $\varphi(g)=1$ and $supp(\varphi)\subseteq U$. Then $\mu(U)=\int_G \chi_Ud\mu\geq I_\mu(\varphi)>0$.
\end{proof}
Let $P(G)$ denote the probability measures on $G$ and for each $K\subseteq G$ compact define 
\[P_K(G)=\lbrace \mu\in P(G)\mid \forall g,h\in supp(\mu):\ r(g)=r(h)\textit{ and } g^{-1}h\in K\rbrace.\]
Note that there is a canonical left action of $G$ on $P_K(G)$ with respect to the anchor map $P_K(G)\rightarrow G^{(0)}$, $\mu\mapsto r(g)$ for any $g\in supp(\mu)$, given by translation.
It was shown in \cite[Proposition~3.1]{Tu12} that $P_K(G)$ is a locally compact, $G$-compact, proper $G$-space. Furthermore, if $X$ is any $G$-compact proper $G$-space, there exists a compact subset $K\subseteq G$ and a $G$-equivariant map $X\rightarrow P_K(G)$ (see \cite[Proposition~3.2]{Tu12}).
If $G$ is ample we can always choose the set $K$ to be compact and open, since if $K_1\subseteq K_2$ then obviously $P_{K_1}(G)\subseteq P_{K_2}(G)$ and if $K$ is any compact set it is contained in a compact open set. In the following, we will show that in this case the spaces $P_K(G)$ are geometric realizations of $G$-simplicial complexes in the following sense (compare \cite[Definition~3.1]{Tu99}):
\begin{defi}\label{G-simplicial complex} Let $G$ be an ample group\-oid and $n\in\NN$.
A \textit{$G$-simplicial complex} of dimension at most $n$ is a pair $(X,\Delta)$ consisting of a locally compact $G$-space $X$ (the set of vertices) and a collection $\Delta$ of finite, non-empty subsets of $X$ (called simplices) with at most $n+1$ elements such that:
\begin{enumerate}
\item the anchor map $p:X\rightarrow G^{(0)}$ has the property, that for every $x\in X$ there exists a compact open neighbourhood $U\subseteq X$ such that $p\mid_U:U\rightarrow p(U)$ is a homeomorphism onto a compact open subset of $G^{(0)}$.
\item for each $\sigma\in\Delta$ we have $\sigma\subseteq p^{-1}(u)$ for some $u\in G^{(0)}$,
\item if $\sigma\in \Delta$, then every non-empty subset of $\sigma$ is also an element of $\Delta$, and
\item for each $\sigma\in \Delta$, say $\sigma=\lbrace x_1,\ldots, x_n\rbrace\subseteq X_u$, there exists a compact open neighbourhood $V$ of $u$ in $G^{(0)}$ and continuous sections $s_1,\ldots s_n:V\rightarrow X$ of $p$ such that $\lbrace s_1(v),\ldots s_n(v)\rbrace\in \Delta$ for all $v\in V$ and $\lbrace s_1(u),\ldots, s_n(u)\rbrace=\sigma$.
\end{enumerate}
The $G$-simplicial complex is \textit{typed} if there is a discrete set $\mathcal{T}$ and a $G$-invariant continuous map $X\rightarrow \mathcal{T}$ whose restriction to the support of a single simplex in $\Delta$ is injective.
\end{defi}
The \textit{geometric realization} of a $G$-simplicial complex $(X,\Delta)$ is the set
$$\betrag{\Delta}=\lbrace\mu\in P(X)\mid supp(\mu)\in\Delta\rbrace$$
equipped with the weak-$\ast$-topology. The geometric realization $\betrag{\Delta}$ will always be a locally compact space and the action of $G$ on $\betrag{\Delta}$, induced by the acion of $G$ on $X$, is proper if $X$ is a proper $G$-space.

\begin{bem}\label{ImagesOfSections}
If $\sigma\in \Delta$, say $\sigma=\lbrace x_1,\ldots, x_n\rbrace\subseteq X_u$ as in item $(4)$ above and for each $1\leq i\leq n$ $U_i$ is a compact open neighbourhood of $x_i$ such that the $U_i$ are pairwise disjoint and $p_{\mid U_i}$ is a homeomorphism onto its image, then we may always assume that the section $s_i$ only takes images in $U_i$. If not, pass from the domain $V$ of the $s_i$ to $$\widetilde{V}=V\cap \bigcap\limits_{0\leq i\leq n} s_i^{-1}(U_i).$$
\end{bem}

Note that the realization of a $0$-dimensional complex $(X,\Delta)$ can be canonically identified with a subset of $X$. Using the existence of local sections as in item $(4)$ we can show that under this identification, $\betrag{\Delta}$ is actually open in $X$:
Let $x\in \Delta$ be given and $U$ in $X$ be an open neighbourhood of $x$ such that $p_{\mid U}$ is a homeomorphism onto its image. Furthermore let $V$ be a neighbourhood of $p(x)$ and $s:V\rightarrow X$ be a section as in $(4)$. By the above remark we may assume $s(V)\subseteq U$. Then $p^{-1}(V)\cap U$ is an open neighbourhood of $x$ and since $p^{-1}(V)\cap U=s(V\cap p(U))$, it is contained in $\Delta$.

Thus, if we restrict $p$ to the subset $\betrag{\Delta}$, it still has the property, that every point $x\in\Delta$ has a compact open neighbourhood $U$ such that $p_{\mid U}:U\rightarrow p(U)$ is a homeomorphism onto a compact open subset of $G^{(0)}$.

\begin{lemma} Let $G$ be an ample groupoid and $K$ be a compact open subset of $G$.
If we define $$\Delta_K(G)=\lbrace \sigma\subseteq G\mid \forall g,h\in \sigma:\ r(g)=r(h)\textit{ and }g^{-1}h\in K\rbrace$$
then $(G,\Delta_K(G))$ is a $G$-simplicial complex in the sense of Definition \ref{G-simplicial complex} and $P_K(G)$ is its geometric realization. We note that $\Delta_K(G)$ has finite dimension (as a $G$-simplicial complex).
\end{lemma}
\begin{proof}
We consider the action of $G$ on itself by left multiplication. Hence the anchor map is just the range map of $G$. Since $G$ is ample, condition (1) of Definition \ref{G-simplicial complex} clearly holds. As axioms (2) and (3) are built into the definition of $\Delta_K(G)$, it remains to prove (4): So let $\sigma=\lbrace g_1,\ldots, g_n\rbrace\in \Delta_K(G)$ be given and let $u:=r(g_1)=\ldots =r(g_n)$.
Let $\widetilde{U_i}$ be a compact open neighbourhood of $g_i$ such that $r_{\mid \widetilde{U}_i}:\widetilde{U}_i\rightarrow r(\widetilde{U}_i)$ is a homeomorphism. We would like to take the inverses of these maps on $\bigcap_{i=1}^n r(\widetilde{U}_i)$ as our sections but we need to make sure that images of a point form a simplex again. Thus, we use the continuity of the multiplication and the openness of $K$ to shrink the $\widetilde{U}_i$ appropriately. To be more precise:
Consider the continuous map
$$f:G\ltimes G\rightarrow G$$ given by $(g,h)\mapsto g^{-1}h$. As $K$ is open and $f$ is continuous, $f^{-1}(K)$ is open. Thus, for all $1\leq i,j\leq n$ we can find compact open neighbourhoods $U_{i,j}$ of $g_i$ and $V_{j,i}$ of $g_j$ such that $(U_{i,j}\times V_{j,i})\cap G\ltimes G\subseteq f^{-1}(K)$.
Let $$U_i:=\widetilde{U_i}\cap\bigcap\limits_{1\leq j\leq n} U_{i,j}\cap V_{i,j}.$$ Then each $U_i$ is a compact open neighbourhood of $g_i$. Let $V:=\bigcap r(U_i)$ and define $s_i:V\rightarrow U_i\subseteq G$ to be the inverse of the range map restricted to $U_i$. These are continuous sections by definition and for each $v\in V$ and $1\leq l,k\leq n$ we have $s_l(v)\in U_{l,k}$ and $s_k(v)\in V_{k,l}$ and thus
$s_i(v)^{-1}s_j(v)=f(s_i(v),s_j(v))\in K$ by construction. Consequently, we get $\lbrace s_1(v),\ldots s_n(v)\rbrace\in \Delta_K(G)$ for all $v\in V$.

Let us finally show that $\Delta_K(G)$ has finite dimension. It is not hard to see that $\Delta_K(G)=G\cdot \lbrace \sigma\in\Delta_K(G)\mid \sigma\subseteq K\rbrace$ and since translating a $\sigma\in\Delta_K(G)$ does not increase its cardinality it is enough to show that the cardinalities of elements of $\lbrace \sigma\in\Delta_K(G)\mid \sigma\subseteq K\rbrace$ are bounded. But for such a $\sigma\subseteq G^u$ we have $\betrag{\sigma}\leq \betrag{K\cap G^{u}}=\lambda^u(K)\leq \sup\lbrace \lambda^u(K)\mid u\in G^{(0)}\rbrace<\infty$ by Lemma \ref{measure of compact set is bounded}, where $\lambda$ denotes the Haar system given by the counting measure on each fibre.
\end{proof}
The arguments in \cite[Section~3.2]{Tu99} carry over to the $G$-equiva\-riant setting and show that the barycentric subdivision of a $G$-simplicial complex $(X,\Delta)$ is a typed $G$-simplicial complex whose geometric realization is $G$-equivariantly homeomorphic to the original one. However for the sake completeness let us at least recall the construction of the barycentric subdivision and show that it is a $G$-simplicial complex again.
\begin{defi}
Let $(X,\Delta)$ be $G$-simplicial complex. For an element $\mu\in\betrag{\Delta}$ with $supp(\mu)=\lbrace x_1,\ldots, x_n\rbrace$ let $$bc(\mu)=\frac{1}{n}\sum\limits_{i=1}^n \delta_{x_i}$$ denote the \textit{isobarycenter} of the simplex $supp(\mu)\in \Delta$.
Let $X'=\lbrace bc(\mu)\mid \mu\in \betrag{\Delta}\rbrace$ and define $\Delta'$ such that a set $\lbrace \nu_1,\ldots,\nu_l\rbrace$ is in $\Delta'$ if and only if $\bigcup\limits_{0\leq j\leq l}supp(\nu_j)\in\Delta$.
\end{defi}

\begin{prop}
The pair $(X',\Delta')$ is a $G$-simplicial complex.
\end{prop}
\begin{proof}
We will only show that $p':X'\rightarrow G^{(0)}$ satisfies property (1) from the definition. The other properties follow easily from the construction. Let $\mu\in X'$, say $\mu=\frac{1}{n}\sum\limits_{i=1}^n \delta_{x_i}$ for $x_1,\ldots,x_n\in X$ and let $U_i$ be a compact open neighbourhood of $x_i$ such that $p_{\mid U_i}$ is a homeomorphism onto its image. Since $G$ is Hausdorff we can assume that the $U_i$ are pairwise disjoint. Now from condition $(4)$ of the definition we get continuous sections $s_1,\ldots s_n:V\rightarrow X$, where $V$ is a compact open neighbourhood of $u:=p'(\mu)$. Following Remark \ref{ImagesOfSections} we can assume that $s_i(V)\subseteq U_i$. Consider the sets
$$W_i:=\lbrace \nu\in X'\mid supp(\nu)\cap U_i\neq \emptyset\rbrace.$$
Note that the intersection $supp(\nu)\cap U_i$ will contain at most one element, since $supp(\nu)$ is contained in one fibre and $U_i$ is the domain of a local homeomorphism.
It follows from Lemma \ref{CharacterizationOfBeingInTheSupportOfaMeasure} that $W_i$ is open. Now let $$W=p'^{-1}(V\cap\bigcap\limits_i p(U_i))\cap \bigcap\limits_i W_i.$$
It is now easy to see that $p'(W)=V\cap \bigcap_i p(U_i)$ and thus $p'(W)$ is compact and open. Furthermore, the map $p'(W)\rightarrow W$ sending an element $v$ to the measure $\frac{1}{n}\sum_{i=1}^n \delta_{s_i(v)}$ is a continuous inverse of $p'$. Hence also $W$ is compact and $p'$ satisfies property (1) from the definition of a $G$-simplicial complex.
\end{proof}
%Assume that $G$ is ample and the anchor map $p:X\rightarrow G^{(0)}$ satisfies the stronger condition, that every point has a compact open neighbourhood $U$ such that $p(U)$ is compact open and $p_{\mid U}:U\rightarrow p(U)$ is a homeomorphism. Then this property also passes to $p'$ since we can replace the merely open sets $U_i$ and $V$ in the above proof by compact open sets. Thus, $W$ and $p'(W)$ will also be compact open.

Let us now proceed to prove one of the main results of this paper:
\begin{satz}\label{MainTheorem}
Let $G$ be an ample, second countable, locally compact Hausdorff groupoid and let $A$ and $B$ be separable $G$-algebras. Suppose there is an element $x\in \KK^G(A,B)$ such that 
\[ \KK^H(C(H^{(0)}),A_{\mid H})\stackrel{\cdot\otimes res_H^G(x)}{\rightarrow}\KK^H(C(H^{(0)}), B_{\mid H})\]
is an isomorphism for all compact open subgroupoids $H\subseteq G$. Then the Kasparov-product with $x$ induces an isomorphism
$$ \cdot\otimes x:\K_*^{\mathrm{top}}(G;A)\rightarrow \K_*^{\mathrm{top}}(G;B).$$
\end{satz}
To show the above theorem we will show that for suitably general $G$-compact subsets $X\subseteq \mathcal{E}(G)$ the map 
$$\cdot\otimes x: \KK^G(C_0(X),A)\rightarrow \KK^G(C_0(X),B)$$
is an isomorphism.
Let us first consider the following special case:
\begin{prop}\label{ZeroDimensionalCase}
Under the assumptions of Theorem \ref{MainTheorem} the map
$$\cdot\otimes x: \KK^G(C_0(X),A)\rightarrow \KK^G(C_0(X),B)$$
is an isomorphism for every $G$-compact proper $G$-space $X$ whose anchor map $p:X\rightarrow G^{(0)}$ has the property, that for every $x\in X$ there exists a compact open neighbourhood $U$ of $x$ in $X$ such that $p_{\mid U}:U\rightarrow p(U)$ is a homeomorphism onto a compact open subset of $G^{(0)}$.
\end{prop}
\begin{proof}
Let us first consider the case that $X$ is the orbit of a single compact open subset $U$ such that $p(U)$ is compact and open in $G^{(0)}$ and $p_{\mid U}:U\rightarrow p(U)$ is a homeomorphism, i.e. $X=GU$.
Consider the set 
$$H=\lbrace g\in G\mid gU\cap U\neq \emptyset\rbrace.$$
Using the fact that $p_{\mid U}$ is a homeomorphism onto $p(U)$ it is not hard to see, that $H$ is a subgroupoid of $G$ and as such isomorphic to $(G\ltimes X)_U^U$ (the isomorphism $(G\ltimes X)_U^U\rightarrow H$ is given by the projection onto the first factor). Since $G\ltimes X$ is proper, the restriction $(G\ltimes X)_U^U$ to $U$ is compact. Clearly, the latter is also open in $G\ltimes X$. Since the anchor map $p:X\rightarrow G^{(0)}$ is open, we can deduce that the first projection $pr_1:G\ltimes X\rightarrow G$ is open. Thus, $H$ is a compact open subgroupoid of $G$. By our choice of $H$ we also have a canonical $G$-equivariant homeomorphism $G\times_H U\cong GU=X$ and thus an equivariant isomorphism $$Ind_H^G C(U)\cong C_0(G\times_H U)\cong C_0(X)$$
by Proposition \ref{Prop:InducedAlgebra:CommutativeCase}.
Using this we can consider the following diagram, which commutes by Lemma \ref{Lemma:Compression and Kasparov Product}.
\begin{center}
   \begin{tikzpicture}[description/.style={fill=white,inner sep=2pt}]
    \matrix (m) [matrix of math nodes, row sep=3em,
     column sep=2.5em, text height=1.5ex, text depth=0.25ex]
     { \KK^G(C_0(X),A) & \KK^G(C_0(X),B) \\
     \KK^H(C(U),A_{\mid H}) & \KK^H(C(U),B_{\mid H}) \\
     };
     \path[->,font=\scriptsize]
     (m-1-1) edge node[auto] {$ \cdot \otimes x $} (m-1-2)
     (m-1-2) edge node[auto] {$ \mathsf{comp}_H^G $} (m-2-2)
     (m-2-1) edge node[auto] { $ \cdot\otimes res_H^G(x) $ } (m-2-2)
     (m-1-1) edge node[auto] { $ \mathsf{comp}_H^G $ } (m-2-1)
    ;
    \end{tikzpicture}
  \end{center}
Since we have an isomorphism $C(U)\cong C(H^{(0)})$, the bottom line in this diagram is an isomorphism.
By Theorem \ref{MainTheorem} the homomorphism $\mathsf{comp}_H^G$ is an isomorphism as well and hence the result follows in this case.

Let us now consider the general case. As $X$ is $G$-compact it admits a finite cover of the form
$$X=\bigcup\limits_{i=1}^n GU_i,$$
where $U_i\subseteq X$ is compact open such that $p_{\mid U_i}$ is a homeomorphism onto its image.
Let us first consider the case $n=2$. By applying a standard argument as for example in \cite[Theorem~21.2.3]{MR1656031} there exist Mayer-Vietoris sequences and so we have a diagram with exact columns, where the horizontal maps are all given by taking Kasparov product with $x$ and we write $\KK^G_*(X,A)$ for $\KK^G_*(C_0(X),A)$ for brevity:
\begin{center}
  \begin{tikzpicture} %[description/.style={fill=white,inner sep=2pt}]
     \matrix (m) [matrix of math nodes, row sep=3em,
     column sep=1em, text height=1em, text depth=0.25em]
     { \vdots&\vdots \\ 
     \KK^G_{*+1}(X,A) & \KK^G_{*+1}(X,B)\\
     
       \KK^G_*(GU_1\cap GU_2,A) & \KK^G_*(GU_1\cap GU_2,B)\\
         \KK_*^G(GU_1,A)\oplus \KK_*^G(GU_2,A) & \KK_*^G(GU_1,B)\oplus \KK_*^G(GU_2,B)\\
       \KK^G_{*}(X,A) & \KK^G_{*}(X,B)\\
       \vdots&\vdots \\
       };
     \path[->,font=\scriptsize]
     %horizontal arrows
     %(m-1-1) edge node[auto] {$  $} (m-1-2)
     (m-2-1) edge node[auto] {$   $} (m-2-2)
     (m-3-1) edge node[auto] { $  $} (m-3-2)
     (m-4-1) edge node[auto] { $ $} (m-4-2)
     (m-5-1) edge node[auto] { $ $} (m-5-2)
     % vertical arrows
     (m-1-1) edge node[auto] {$  $} (m-2-1)
     (m-2-1) edge node[auto] {$  $} (m-3-1)
     (m-3-1) edge node[auto] {$  $} (m-4-1)
    (m-4-1) edge node[auto] {$  $} (m-5-1)
    (m-5-1) edge node[auto] {$  $} (m-6-1)
    
     	  (m-1-2) edge node[auto] {$  $} (m-2-2)
     	  (m-2-2) edge node[auto] {$  $} (m-3-2)
          (m-3-2) edge node[auto] {$  $} (m-4-2)
          (m-4-2) edge node[auto] {$  $} (m-5-2)
          (m-5-2) edge node[auto] {$  $} (m-6-2)
     ;
     \end{tikzpicture}
  \end{center}
	The diagram commutes, since all maps in the Mayer-Vietoris sequence are induced by equivariant $^*$-homomorphisms (see the proof of \cite[Theorem~21.2.2]{MR1656031}) and hence commutativity follows from the associativity of the Kasparov product.
  Using the first step of this proof we already know, that the second horizontal map is an isomorphism. Consider the set $V=U_1\cap GU_2$. It is clearly open and using properness of the action one easily verifies that is is also closed (apply Proposition \ref{Prop:CharacterizationsProperAction} (4)). Since $V\subseteq U_1$ we have that $p_{\mid V}$ is also a homeomorphism onto its image. One easily checks that $GV=GU_1\cap GU_2$. Thus, the third horizontal map is also an isomorphism. Hence the result follows by an application of the Five-Lemma.
  
If $n>2$ is arbitrary, use induction and the above Mayer-Vietoris argument on the decomposition $X=GU_1 \cup \bigcup\limits_{i=2}^n GU_i$ to complete the proof.
\end{proof}
We are now ready for the proof of Theorem \ref{MainTheorem}:
\begin{proof}[Proof of Theorem \ref{MainTheorem}] 
Our proof consists of a two step reduction, each of which tells us that we only need to prove that
\begin{equation}\label{Equation:Iso}
\cdot\otimes x: \KK^G(C_0(X),A)\rightarrow \KK^G(C_0(X),B)
\end{equation}
is an isomorphism for more and more special $G$-spaces $X$.

In the first step we will reduce the problem to showing that (\ref{Equation:Iso}) is an isomorphism for the spaces $P_K(G)$ from the beginning of this section. So let us assume (\ref{Equation:Iso}) holds for all the $P_K(G)$, where $K\subseteq G$ is a compact open subset, and explain how to deduce the conclusion of Theorem \ref{MainTheorem} from there.
Let $X_1$ be any $G$-compact subspace of $\mathcal{E}(G)$. Then $X_1$ is a proper $G$-space itself and thus we can find a compact open subset $K_1\subseteq G$ and a $G$-equivariant map $\varphi_1:X_1\rightarrow P_{K_1}(G)$ by the discussion in the beginning of this section.
Using the universal property of $\mathcal{E}(G)$ there is also a $G$-equivariant map $\psi_1:P_{K_1}(G)\rightarrow \mathcal{E}(G)$.  Let $X_2:=\psi_1(P_{K_1}(G))$. Then $X_2$ is a $G$-compact subspace of $\mathcal{E}(G)$.
By the universal property of $\mathcal{E}(G)$ the composition $\psi_1\circ\varphi_1$ is $G$-homotopic to the canonical inclusion map. So up to replacing $X_2$ by a larger space if necessary, we may assume that the map $\psi_1\circ\varphi_1:X_1\rightarrow X_2$ is $G$-homotopic to the inclusion map. 
Now proceed as above to find a sequence $(X_n)_n$ of $G$-compact subspaces of $\mathcal{E}(G)$ together with maps $\varphi_n$, and $\psi_n$ such that $\psi_n\circ \varphi_n$ is $G$-homotopic to the inclusion $X_n\hookrightarrow X_{n+1}$.
Since the Kasparov-product is natural, we get a commutative diagram, where all the horizontal arrows are given by taking Kasparov-product with $x$ and the vertical arrows are the maps found by the procedure just described.
\begin{center}
  \begin{tikzpicture} %[description/.style={fill=white,inner sep=2pt}]
     \matrix (m) [matrix of math nodes, row sep=3em,
     column sep=1em, text height=1em, text depth=0.25em]
     { \KK_*^G(C_0(X_1),A) & \KK_*^G(C_0(X_1),B)\\
       \KK_*^G(C_0(P_{K_1}(G)),A) & \KK_*^G(C_0(P_{K_1}(G)),B)\\
       \KK_*^G(C_0(X_2),A) & \KK_*^G(C_0(X_2),B)\\
       \KK_*^G(C_0(P_{K_2}(G)),A) & \KK_*^G(C_0(P_{K_2}(G)),B)\\
       \KK_*^G(C_0(X_3),A)  & \KK_*^G(C_0(X_3),B)\\
       \vdots & \vdots
           \\ };
     \path[->,font=\scriptsize]
     %horizontal arrows
     (m-1-1) edge node[auto] {$  $} (m-1-2)
     (m-2-1) edge node[auto] {$ \cong  $} (m-2-2)
     (m-3-1) edge node[auto] { $ $} (m-3-2)
     (m-4-1) edge node[auto] { $\cong $} (m-4-2)
     (m-5-1) edge node[auto] { $ $} (m-5-2)
     %(m-6-1) edge node[auto] { $\cong $} (m-6-2)
     %(m-7-1) edge node[auto] { $ $} (m-7-2)
     %(m-8-1) edge node[auto] { $\cong $} (m-8-2)
     % vertical arrows
     (m-1-1) edge node[auto] {$ (\varphi_1)_* $} (m-2-1)
     (m-2-1) edge node[auto] {$ (\psi_1)_* $} (m-3-1)
     (m-3-1) edge node[auto] {$ (\varphi_2)_* $} (m-4-1)
     (m-4-1) edge node[auto] {$ (\psi_2)_* $} (m-5-1)
     (m-5-1) edge node[auto] {$  $} (m-6-1)
     %(m-6-1) edge node[auto] {$  $} (m-7-1)
     %(m-7-1) edge node[auto] {$  $} (m-8-1)
     	  (m-1-2) edge node[auto] {$ (\varphi_1)_* $} (m-2-2)
     	  (m-2-2) edge node[auto] {$ (\psi_1)_* $} (m-3-2)
          (m-3-2) edge node[auto] {$ (\varphi_2)_* $} (m-4-2)
          (m-4-2) edge node[auto] {$ (\psi_2)_* $} (m-5-2)
          (m-5-2) edge node[auto] {$  $} (m-6-2)
          %(m-6-2) edge node[auto] {$  $} (m-7-2)
          %(m-7-2) edge node[auto] {$  $} (m-8-2)
     ;
     \end{tikzpicture}
  \end{center}
  
By going 'zick-zack' in this diagram we get the following diagram:
\begin{center}
   \begin{tikzpicture}[description/.style={fill=white,inner sep=2pt}]
     \matrix (m) [matrix of math nodes, row sep=3em,
     column sep=2.5em, text height=1.5ex, text depth=0.25ex]
     { \KK_*^G(C_0(X_1),A) & \KK_*^G(C_0(X_1),B)\\
       \vdots & \KK_*^G(C_0(X_2),B)\\
       \KK_*^G(C_0(X_3),A) & \vdots\\
       \vdots & \KK_*^G(C_0(X_4),B)\\
       \vdots & \vdots\\
       \K_*^{\mathrm{top}}(G;A)  & \K_*^{\mathrm{top}}(G;B)
           \\ };
     \path[->,font=\scriptsize]
     %diagonal arrows
     (m-1-1) edge node[auto] {$ \alpha_1  $} (m-2-2)
     (m-2-2) edge node[auto] {$ \alpha_2 $} (m-3-1)
     (m-3-1) edge node[auto] {$ \alpha_3 $} (m-4-2)
     (m-4-2) edge node[auto] {$  $} (m-5-1)
     %vertical arrows
     (m-1-1) edge node[auto] {$  $} (m-2-1)
     (m-2-1) edge node[auto] {$  $} (m-3-1)
     (m-3-1) edge node[auto] {$  $} (m-4-1)
     (m-4-1) edge node[auto] {$  $} (m-5-1)
     (m-5-1) edge node[auto] {$  $} (m-6-1)
     	  (m-1-2) edge node[auto] {$  $} (m-2-2)
     	  (m-2-2) edge node[auto] {$  $} (m-3-2)
          (m-3-2) edge node[auto] {$  $} (m-4-2)
          (m-4-2) edge node[auto] {$  $} (m-5-2)
          (m-5-2) edge node[auto] {$  $} (m-6-2)
     ;
     \end{tikzpicture}
  \end{center}

Each of the columns of this diagram is an inductive sequence of abelian groups, and in fact a subsequence of the inductive system defining topological K-theory with coefficients. Hence, by an elementary argument, the limits of these two sequences are the groups  $\K_*^{\mathrm{top}}(G;A)$  and $\K_*^{\mathrm{top}}(G;B)$ respectively, as indicated.
Whenever we have such a diagram, the inductive limits must be isomorphic, such that the isomorphism commutes with the diagram (i.e. it is exactly the morphism induced by taking Kasparov-product in each step). This completes the first step of the proof.

Let us now proceed with the second reduction step. The main advantage of working with the spaces $P_K(G)$ is that it is (the geometric realization of) a proper, $G$-compact finite dimensional $G$-simplicial complex, and its barycentric subdivision is $G$-equivariantly homeomorphic to it.
In fact we don't need the special model $P_K(G)$ in what follows and will prove that
$$\cdot\otimes x:\KK_*^G(C_0(X),A)\rightarrow \KK_*^G(C_0(X),B)$$
is an isomorphism for every typed, proper, $G$-compact $G$-simplicial complex $X$ of finite dimension.

To do so we will use an induction argument on the dimension $n$ of $X$ to reduce the problem to the zero dimensional case. If $X$ is (the geometric realization) of a $0$-dimensional complex it follows from the discussion after Remark \ref{ImagesOfSections}, that the anchor map $X\rightarrow G^{(0)}$ has the property, that every point in $X$ has a compact open neighbourhood, such that the anchor map restricts to a homeomorphism onto its image on that neighbourhood. Consequently, Proposition \ref{ZeroDimensionalCase} tells us that $\cdot\otimes x:\KK_*^G(C_0(X),A)\rightarrow \KK_*^G(C_0(X),B)$ is an isomorphism.

Now let $X$ be a $G$-simplicial complex of dimension $n>0$,  $Y$ be its $n-1$-skeleton, and $U=X\setminus Y$ the union of all open $n$-simplices.
Then we get a $G$-equivariant exact sequence
$$0\longrightarrow C_0(U)\longrightarrow C_0(X)\longrightarrow C_0(Y)\longrightarrow 0.$$
As $Y$ is clearly $G$-invariant, \cite[Lemma~3.9]{Tu12} yields the following commutative diagram with exact columns:
\begin{center}
  \begin{tikzpicture} %[description/.style={fill=white,inner sep=2pt}]
     \matrix (m) [matrix of math nodes, row sep=3em,
     column sep=1em, text height=1em, text depth=0.25em]
     { \vdots & \vdots\\
       \KK_*^G(C_0(Y),A) & \KK_*^G(C_0(Y),B)\\
       \KK_*^G(C_0(X),A) & \KK_*^G(C_0(X),B)\\
       \KK_*^G(C_0(U),A) & \KK_*^G(C_0(U),B)\\
       \vdots & \vdots
           \\ };
     \path[->,font=\scriptsize]
     %horizontal arrows
     (m-2-1) edge node[auto] {$ \cdot\otimes x  $} (m-2-2)
     (m-3-1) edge node[auto] { $\cdot\otimes x $} (m-3-2)
     (m-4-1) edge node[auto] { $\cdot\otimes x $} (m-4-2)
     %(m-6-1) edge node[auto] { $\cong $} (m-6-2)
     %(m-7-1) edge node[auto] { $ $} (m-7-2)
     %(m-8-1) edge node[auto] { $\cong $} (m-8-2)
     % vertical arrows
     (m-1-1) edge node[auto] {$  $} (m-2-1)
     (m-2-1) edge node[auto] {$  $} (m-3-1)
     (m-3-1) edge node[auto] {$  $} (m-4-1)
     (m-4-1) edge node[auto] {$  $} (m-5-1)
     %(m-5-1) edge node[auto] {$  $} (m-6-1)
     %(m-6-1) edge node[auto] {$  $} (m-7-1)
     %(m-7-1) edge node[auto] {$  $} (m-8-1)
     	  (m-1-2) edge node[auto] {$  $} (m-2-2)
     	  (m-2-2) edge node[auto] {$  $} (m-3-2)
          (m-3-2) edge node[auto] {$  $} (m-4-2)
          (m-4-2) edge node[auto] {$  $} (m-5-2)
          %(m-5-2) edge node[auto] {$  $} (m-6-2)
          %(m-6-2) edge node[auto] {$  $} (m-7-2)
          %(m-7-2) edge node[auto] {$  $} (m-8-2)
     ;
     \end{tikzpicture}
  \end{center}
If we assume inductively that the upper horizontal map is an isomorphism, we only need to show that the lower map is also an isomorphism to invoke the Five-Lemma and conclude the result.
But $U$ is equivariantly homeomorphic to $X'\times \RR^n$, where $X'$ denotes the barycenters of $n$-dimensional simplices. Thus, we have $\KK^G_*(C_0(U),A)\cong \KK^G_{*+n}(C_0(X'),A)$. Since taking suspension is compatible with the Kasparov product, it is enough to show that $$\cdot\otimes x:\KK^G_*(C_0(X'),A)\rightarrow \KK^G_*(C_0(X'),B)$$ is an isomorphism. But $X'$ is a $G$-compact, proper $G$-space whose anchor map is a local homeomorphism. Hence the result follows from Proposition~\ref{ZeroDimensionalCase}.
\end{proof}

%In the following we briefly discuss the difficulties that arise when one tries to prove an analogue of Theorem \ref{MainTheorem} for general étale groupoids. The main difficulties arise from basic point-set topology facts: If $G$ is no longer totally disconnected, then we can still do most of the reduction steps using the simplicial complexes $P_K(G)$. Note however that in this context $K$ cannot be chosen to be open. This leads to the fact, that in the zero-dimensional case (compare Proposition \ref{ZeroDimensionalCase}) we may only assume, that the anchor map is locally injective. When defining $H$ as in the proof of Proposition \ref{ZeroDimensionalCase} it is still a subgroupoid, but has relatively bad topological properties as a subset of $G$. It is neither open nor closed in $G$, two features we used in the proof Theorem \ref{CompressionIsomorphism}. Even in this general situation one can still show, that $H$ is a proper groupoid, which is open in $G_{H^{(0)}}$ and closed in $G_{H^{(0)}}^{H^{(0)}}$. Hence the compression homomorphism still makes sense in this setting. It is just our method to prove that $\mathsf{comp}_H^G$ is an isomorphism, which fails in this generality.

\section{Amenability at Infinity and the Baum-Connes Conjecture}
\label{Section:AmenabilityAtInfinity}
As an application of Theorem \ref{MainTheorem} we will show that for ample groupoids, which are strongly amenable at infinity, the Baum-Connes assembly map is split-injective. Let us first recall the definitions:

\begin{defi}[\cite{Lassagne},\cite{Delaroche}]
A locally compact Hausdorff groupoid $G$ is called \textit{amenable at infinity}, if there exists a $G$-space $Y$ such that the anchor map $p:Y\rightarrow G^{(0)}$ is proper and $G\ltimes Y$ is amenable (i.e. $G$ acts amenably on $Y$).

We call $G$ \textit{strongly amenable at infinity}, if in addition the anchor map $p:Y\rightarrow G^{(0)}$ admits a continuous (not necessarily equivariant) section.
\end{defi}
Note, that every amenable groupoid is strongly amenable at infinity by taking $Y=G^{(0)}$ with the canonical $G$-action.
Furthermore, by results of \cite{Lassagne}, if $Y$ is a $G$-space witnessing amenability at infinity of $G$, such that the anchor map $p$ is also open, then $G$ is strongly amenable at infinity.

Now if $G$ is (strongly) amenable at infinity and $Y$ is a $G$-space witnessing this, the properness of $p:Y\rightarrow G^{(0)}$ implies that we get an induced map
$$p^*:C_0(G^{(0)})\rightarrow C_0(Y)$$
and consequently, for every $G$-algebra $A$, we get a $G$-equivariant $*$-homo\-morphism
$$id_A\otimes p^*:A\cong A\otimes_{G^{(0)}}C_0(G^{(0)})\rightarrow A\otimes_{G^{(0)}}C_0(Y).$$
This homomorphism in turn induces a map on the level of topological K-theory, which we - by slight abuse of notation - also denote by
$$p_*:\K_*^{\mathrm{top}}(G;A)\rightarrow \K_*^{\mathrm{top}}(G;A\otimes_{G^{(0)}}C_0(Y))$$

By results in \cite{Delaroche} and \cite{Lassagne} we can always find $Y$ with the following additional properties:
\begin{itemize}
	\item $Y$ is second countable.
	\item Each $Y_u$ is a convex space and $G$ acts by affine transformations on $Y$.
\end{itemize}
If we fix $Y$ with these properties we can show:
\begin{prop}\label{contractible}
Let $Y$ be a $G$-space with the properties listed above. If $K\subseteq G$ is a proper, open subgroupoid, then $Y_K=p^{-1}(K)\subseteq Y$ is $K$-equivariantly homotopy-equivalent to $K^{(0)}$.
\end{prop}
\begin{proof}
We will construct a $K$-equivariant continuous section $\tilde{s}:K^{(0)}\rightarrow Y_K$ as follows:
Let $c:K^{(0)}\rightarrow [0,1]$ be a cut-off function for $K$, i.e.
\begin{enumerate}
	\item $\sum\limits_{k\in K^u}c(d(k))=1$ for all $u\in K^{(0)}$, and
	\item $r:supp(c\circ d)\rightarrow K^{(0)}$ is proper.
\end{enumerate}
We define
$$\tilde{s}(u):=\sum\limits_{k\in K^u} c(d(k)) k\cdot s(d(k)),$$
where $s:G^{(0)}\rightarrow Y$ is the continuous section from above. Note that by $(2)$ the sum in the definition is finite for each fixed $u\in K^{(0)}$, and hence $(1)$ and the convexity of $Y_u$ imply that $\tilde{s}(u)\in Y_u$. Thus $\tilde{s}$ is a well-defined section.

The following calculation shows that $\tilde{s}$ is $K$-equivariant:
\begin{align*}
\tilde{s}(k'\cdot u) & = \tilde{s}(r(k'))\\
& = \sum\limits_{k\in K^{r(k')}}c(d(k))k\cdot s(d(k))\\
& = \sum\limits_{k\in K^u} c(d(k'k))k'k\cdot s(d(k'k))\\
& = k'\cdot \left(\sum\limits_{k\in K^u} c(d(k))k\cdot s(d(k))\right)\\
& = k'\tilde{s}(u)
\end{align*}
It remains to show that $\tilde{s}$ is continuous. We prove this along the lines of Lemma \ref{Lem:proper support}:
Fix a $u\in K^{(0)}$ and let $V$ be an open neighbourhood of $u$ such that $\overline{V}$ is compact. Let $\psi\in C_c(K^{(0)})$ be a positive function with $\psi\equiv 1 $ on $\overline{V}$. Then $f(k):=c(d(k))\psi(r(k))$ has compact support and for all $v\in V$ we still have $\sum\limits_{k\in K^v}f(k)=1$ and hence
$\tilde{s}(v)=\sum\limits_{k\in K^v} f(k) k\cdot s(d(k))\in Y_v$.
Now we use compactness of $supp(f)$ to cover it with a finite number of open bisections $(U_i)_i$ and use a partition of unity subordinate to this covering to write $f$ as a finite sum $f=\sum f_i$. Then we get
$$\tilde{s}(v)=\sum_i\sum_{k\in K^v} f_i(k) k\cdot s(d(k))= \sum_i f_i(r_{\mid U_i}^{-1}(v)) r_{\mid U_i}^{-1}(v)\cdot s(d(r_{\mid U_i}^{-1}(v))).$$
The latter expression in this equation is obviously continuous in $v$ since all the functions and operations used are continuous. Hence $\tilde{s}$ must be continuous.

Now by construction we have $p\circ \tilde{s}=id_{K^{(0)}}$ and by convexity the linear homotopy gives $\tilde{s}\circ p\simeq id_{Y_K}$. This homotopy is equivariant since the action of $K$ on $Y_K$ is affine.
\end{proof}
We can now prove the following extention of results from \cite{Higson} and \cite{CEO} to ample groupoids:
\begin{satz}\label{Theorem:BC Injectivity  for amenable at infty grpds}
Let $G$ be a second countable ample groupoid which is strongly amenable at infinity. Then, for any separable $G$-algebra $A$ the Baum-Connes assembly map
$$\mu_A:\K_*^{\mathrm{top}}(G;A)\rightarrow \K_*(A\rtimes_r G)$$
is split injective.
\end{satz}
\begin{proof}
Consider the homomorphism 
$$p_*:\K_*^{\mathrm{top}}(G;A)\rightarrow \K_*^{\mathrm{top}}(G;A\otimes_{G^{(0)}}C_0(Y))$$
induced by the anchor map $p:Y\rightarrow G^{(0)}$ as explained prior to Proposition \ref{contractible}. As explained there, we can also assume that $Y$ is second countable, each fibre $Y_u$ is convex and $G$ acts by affine transformations. Furthermore we may assume that $p$ admits a continuous section. Thus, for every proper, open subgroupoid $K\subseteq G$ we can apply Proposition \ref{contractible} to see that the restriction of $p_K:Y_K\rightarrow K^{(0)}$  of $p$ induces an isomorphism
$$\mathrm{KK}^K(C_0(K^{(0)}),A_K)\rightarrow \mathrm{KK}^K(C_0(K^{(0)}),A_K\otimes_{K^{(0)}} C_0(Y_K)).$$
Thus, we have checked the conditions of Theorem \ref{MainTheorem} and can deduce that $p_*$ is an isomorphism. By naturality of the assembly map $p_*$ fits into the following commutative diagram:
\begin{center}
	\begin{tikzpicture}[description/.style={fill=white,inner sep=2pt}]
	\matrix (m) [matrix of math nodes, row sep=3em,
	column sep=2.5em, text height=1.5ex, text depth=0.25ex]
	{ \K_*^{\mathrm{top}}(G;A) & \K_*(A\rtimes_ r G) \\
		\K_*^{\mathrm{top}}(G;A\otimes_{G^{(0)}}C_0(Y)) & \K_*((A\otimes_{G^{(0)}} C_0(Y))\rtimes_r G) \\
	};
	\path[->,font=\scriptsize]
	(m-1-1) edge node[auto] {$ \mu_A $} (m-1-2)
	(m-1-2) edge node[auto] {$ (p\rtimes G)_* $} (m-2-2)
	(m-2-1) edge node[auto] { $ \mu_{A\otimes C_0(Y)} $ } (m-2-2)
	(m-1-1) edge node[auto] { $ p_* $ } (m-2-1)
	;
	\end{tikzpicture}
\end{center}
By \cite[Lemma~4.1]{STY02} the Baum-Connes assembly map for $G$ with coefficients in $A\otimes_{G^{(0)}} C_0(Y)$ is an isomorphism if and only if the assembly map for $G\ltimes Y$ with coefficients in $A\otimes_{G^{(0)}} C_0(Y)$ is. Since $G\ltimes Y$ is amenable by assumption, we can apply the results in \cite{Tu98} to conclude, that the lower horizontal map in the above diagram is an isomorphism. Thus, $\mu_A$ is injective with splitting homomorphism $\sigma_A:=p_*^{-1}\circ \mu_{A\otimes C_0(Y)}\circ (p\rtimes_rG)_*$.
\end{proof}

We will now apply Theorem \ref{Theorem:BC Injectivity  for amenable at infty grpds} to relate the Baum-Connes conjecture for an ample, strongly amenable at infinity groupoid group bundle to the Baum-Connes conjecture for each of its isotropy groups. This generalizes part (b) of \cite[Proposition~3.1]{MR2010742}, which treats the case of a trivial group bundle (i.e. $G=\Gamma\times X$ for some discrete group $\Gamma$ and a totally disconnected space $X$). We also make use of ideas from the recent paper \cite{ELN17} to avoid $\gamma$-elements.

We shall need the notion of an exact groupoid:
\begin{defi}
	A locally compact groupoid $G$ with Haar system is called \textit{exact} (in the sense of Kirchberg and Wassermann), if for every $G$-equivariant exact sequence 
	$$0\rightarrow I\rightarrow A\rightarrow B\rightarrow 0$$
	of $G$-algebras, the corresponding sequence
	$$0\rightarrow I\rtimes_r G\rightarrow A\rtimes_r G\rightarrow B\rtimes_r G\rightarrow 0$$
	of reduced crossed products is exact.
\end{defi}
The following result is a part of \cite[Proposition~6.7]{Delaroche}:
\begin{prop}
	Let $G$ be an étale groupoid. If $G$ is amenable at infinity, then $G$ is exact.
\end{prop}

Let us now focus on group bundles:
For a start let us observe that if $G$ is an étale groupoid group bundle and $(A,G,\alpha)$ is a groupoid dynamical system, then $(A_u,G_u^u,\alpha_u)$ is a (group) dynamical system for every $u\in G^{(0)}$.
The following proposition describes the relation of the crossed product $A\rtimes_r G$ with the crossed products corresponding to the fibres:
\begin{prop}\label{Prop:CrossedProductsByGroupBundles}
Let $G$ be an étale groupoid group bundle and $A$ be a $G$-algebra. Then the following hold:
\begin{enumerate}
	\item The reduced crossed product $A\rtimes_r G$ is a $C_0(G^{(0)})$-algebra.
	\item If $G$ is exact, then the fibres are given by $(A\rtimes_r G)_u=A_u\rtimes_r G_u^u$.
	\item If in addition the $C^*$-bundle $\mathcal{A}$ associated to $A$ is continuous, then so is the $C^*$-bundle associated to $A\rtimes_r G$.	
\end{enumerate}
\end{prop}
\begin{proof}
For $\varphi\in C_0(G^{(0)})$ and $f\in \Gamma_c(G,r^*\mathcal{A})$ define a linear map $\Phi(\varphi):\Gamma_c(G,r^*\mathcal{A})\rightarrow \Gamma_c(G,r^*\mathcal{A})$ by
$$(\Phi(\varphi) f)(g):=\varphi(r(g))f(g)$$
We want to show, that $\Phi(\varphi)$ extends to an element of the multiplier algebra of $A\rtimes_r G$. To this end let $u\in G^{(0)}$. Then, for $\varphi\in C_0(G^{(0)}), f\in \Gamma_c(G,r^*\mathcal{A})$ and $\xi\in C_c(G^u_u,A_u)$, we compute
\begin{align*}
(\pi_u(\Phi(\varphi)f)\xi)(g) & = \sum\limits_{h\in G^u_u} \alpha_g^{-1}((\Phi(\varphi)f)(g^{-1}h))\xi(h)\\
& = \sum\limits_{h\in G^u_u} \varphi(u)\alpha_g(f(g^{-1}h))\xi(h)\\
& = (\varphi(u)\pi_u(f)\xi)(g)
\end{align*}
Hence we have $\pi_u(\Phi(\varphi(f)))=\varphi(u)\pi_u(f)$ and applying this equality we obtain
$$\norm{\Phi(\varphi)f}_r=\sup_{u\in G^{(0)}}\norm{\pi_u(\Phi(\varphi)f)}=\sup_{u\in G^{(0)}}\betrag{\varphi(u)}\norm{\pi_u(f)}\leq \norm{\varphi}_\infty \norm{f}_r$$
Thus, $\Phi(\varphi)$ extends to a bounded linear map $\Phi(\varphi):A\rtimes_r G\rightarrow A\rtimes_r G$. One easily computes on the dense subalgebra $\Gamma_c(G,r^*\mathcal{A})$, that $\Phi(\varphi)$ is adjointable with $\Phi(\varphi)^*=\Phi(\overline{\varphi})$.
We have thus defined a $*$-homo\-morphism $\Phi:C_0(G^{(0)})\rightarrow M(A\rtimes_r G)$. Next, we would like to show that $\Phi$ takes its image in the centre of the multiplier algebra. By \cite[Lemma~8.3]{Williams} it is enough to show, that $\Phi(\varphi)(f_1\ast f_2)=f_1\ast \Phi(\varphi)f_2$ for all $f_1,f_2\in \Gamma_c(G,r^*\mathcal{A})$ and $\varphi\in C_0(G^{(0)})$. For $g\in G$ and $u:=r(g)=d(g)$ we compute
\begin{align*}
(\Phi(\varphi)(f_1\ast f_2)(g)&=\varphi(u)(f_1\ast f_2)(g)\\
& = \sum\limits_{h\in G^u_u}\varphi(u)f_1(h)\alpha_h(f_2(h^{-1}g))\\
& = \sum\limits_{h\in G^u_u}f_1(h)\alpha_h(\varphi(u)f_2(h^{-1}g))\\
& = \sum\limits_{h\in G^u_u} f_1(h)\alpha_h((\Phi(\varphi)f_2)(h^{-1}g))\\
& = (f_1\ast \Phi(\varphi)f_2)(g)
\end{align*}
It remains to show that $\Phi$ is non-degenerate. Given $x\in A\rtimes_r G$ and $\varepsilon>0$, find $f\in \Gamma_c(G,r^*\mathcal{A})$ such that $\norm{x-f}_r<\varepsilon$. Choose a function $\varphi\in C_c(G^{(0)})$, $0\leq\varphi\leq 1$ with $\varphi= 1$ on $r(supp(f))$. Then $\Phi(\varphi)f=f$ and hence $x\in \overline{C_0(G^{(0)})A\rtimes_r G}$. We have thus established the first part of the proposition, namely that $A\rtimes_r G$ is a $C_0(G^{(0)})$-algebra.

For the second part, we want to analyse the fibres: We always have a canonical family of surjective $*$-homo\-morphisms defined as follows:
For each $u\in G^{(0)}$, there is a canonical map $q_u:\Gamma_c(G,r^*\mathcal{A})\rightarrow C_c(G^u_u,A_u)$ given by restriction. This map extends to a surjective $*$-homo\-morphism $A\rtimes_r G\rightarrow A_u\rtimes_r G_u^u$, still denoted by $q_u$.
Let $J_u$ denote the ideal $\overline{C_0(G^{(0)}\setminus\lbrace u\rbrace)A\rtimes_r G}$ of $A\rtimes_r G$. We clearly have $J_u= A_{\mid G^{(0)}\setminus\lbrace u\rbrace}\rtimes_r G_{\mid G^{(0)}\setminus\lbrace u\rbrace}$.
Now if $G$ is exact, the sequence
$$0\rightarrow A_{\mid G^{(0)}\setminus\lbrace u\rbrace}\rtimes_r G_{\mid G^{(0)}\setminus\lbrace u\rbrace}\rightarrow A\rtimes_r G\stackrel{q_u}{\rightarrow} A_u\rtimes_r G_u^u\rightarrow 0$$
is exact for every $u\in G^{(0)}$. Hence $ker(q_u)=J_u$. It follows that $(A\rtimes_r G)_u=A_u\rtimes_r G_u^u$.

Finally, for part $(3)$, we have to show continuity of the $C^*$-bundle associated to the $C_0(G^{(0)})$-algebra $A\rtimes_r G$, provided the continuity of $\mathcal{A}$. For this we have to prove, that $u\mapsto \norm{q_u(x)}$ is lower semicontinuous for every $x\in A\rtimes_r G$.
Recall that we have a representation $\pi:\Gamma_c(G,r^*\mathcal{A})\rightarrow \mathcal{L}_A(L^2(G,A))$.
We can compute
\begin{align*}\norm{q_u(f)}_r&=\norm{\pi_u(f)}\\
& =\sup \lbrace \norm{\lk \pi(f)\xi,\eta\rk_A(u)}\mid \xi,\eta\in \Gamma_c(G,r^*\mathcal{A}), \norm{\xi},\norm{\eta}\leq 1\rbrace.
\end{align*} The latter expression however is lower semicontinuous as a function in $u$, since it is the supremum of the continuous functions $$u\mapsto \norm{\lk \pi(f)\xi,\eta\rk_A(u)}.$$
\end{proof}

\begin{lemma}
	Let $G$ be an étale groupoid group bundle. If $G$ is amenable at infinity, then so is $G_u^u$ for each $u\in G^{(0)}$.
\end{lemma}
\begin{proof}
	By assumption there exists a locally compact space $X$ and an action of $G$ on $X$ with proper anchor map $p:X\rightarrow G^{(0)}$, such that $G\ltimes X$ is amenable. Then $X_u:=p^{-1}(\lbrace u \rbrace)$ is a compact subspace of $X$ and the action of $G$ on $X$ restricts to an action of $G_u^u$ on $X_u$. In particular $G_u^u\ltimes X_u$ is a closed subgroupoid of $G\ltimes X$. Hence it is amenable by \cite[Proposition~5.1.1]{ADR}.
\end{proof}

Next, we turn to $\mathrm{KK}$-theory. Recall, that for $C_0(X)$-algebras $A$ and $B$, the group $\mathcal{R}\KK(X;A,B)$ is built from Kasparov triples $(E,\Phi,T)$, as in the construction of $\KK(A,B)$, satisfying the additional requirement that
$$(fa)\cdot (eb)=(ae)\cdot (bf) \text{ for all }f\in C_0(X), a\in A, e\in E, \text{ and }b\in B.$$

 We will start with the following observation:
\begin{lemma}
	If $G$ is a second countable étale groupoid group bundle and $(A,G,\alpha)$ and $(B,G,\beta)$ are separable groupoid dynamical systems, then the descent map $j_{G}$ actually takes values in the group $\RKK(G^{(0)};A\rtimes_r G,B\rtimes_r G)$.
\end{lemma}
\begin{proof} Let $(E,\Phi,T)\in \mathbb{E}^G(A,B)$.
	It is enough to show, that for all $\varphi\in C_0(G^{(0)})$, $f\in \Gamma_c(G,r^*\mathcal{A})$, $f'\in \Gamma_c(G,r^*\mathcal{B})$ and $\xi\in \Gamma_c(G,r^*\mathcal{E})$ we have
	$$(\varphi f)\xi f'=f\xi (\varphi f').$$
	Hence we compute for all $g\in G$:
	\begin{align*}
	((\varphi f)\xi f')(g) & = \sum\limits_{h\in G^{r(g)}} ((\varphi f)\xi)(h) \beta_h(f'(h^{-1}g))\\
	& = \sum\limits_{h\in G^{r(g)}} \sum\limits_{s\in G^{r(h)}} \varphi(r(s))f(s) V_s(\xi(s^{-1}h))\beta_h(f'(h^{-1}g))\\
	& = \sum\limits_{h\in G^{r(g)}} \sum\limits_{s\in G^{r(h)}} f(s)V_s(\xi(s^{-1}h))\beta_h((\varphi f')(h^{-1}g))\\
	& =(f\xi(\varphi f'))(g).
	\end{align*}
\end{proof}

\begin{lemma}\label{Lemma:AssemblyMapForGroupoidCommutesWithAssemblyMapForIsotropyGroups}
Let $G$ be a second countable exact étale groupoid group bundle and $A$ be a separable $G$-algebra. For each $u\in G^{(0)}$ the inclusion map $i_u:G_u^u\rightarrow G$ induces a group homomorphism $i_u^*:\K_*^{\mathrm{top}}(G;A)\rightarrow \K_*^{\mathrm{top}}(G_u^u;A_u)$, such that the following diagram commutes:
\begin{center}
	\begin{tikzpicture}[description/.style={fill=white,inner sep=2pt}]
	\matrix (m) [matrix of math nodes, row sep=3em,
	column sep=2.5em, text height=1.5ex, text depth=0.25ex]
	{ \K_*^{\mathrm{top}}(G;A) & \K_*(A\rtimes_ r G) \\
		\K_*^{\mathrm{top}}(G_u^u;A_u) & \K_*(A_u\rtimes_r G_u^u) \\
	};
	\path[->,font=\scriptsize]
	(m-1-1) edge node[auto] {$ \mu_A $} (m-1-2)
	(m-1-2) edge node[auto] {$ q_{u,*} $} (m-2-2)
	(m-2-1) edge node[auto] { $ \mu_{A_u} $ } (m-2-2)
	(m-1-1) edge node[auto] { $ i_u^* $ } (m-2-1)
	;
	\end{tikzpicture}
\end{center}
\end{lemma}
\begin{proof}
It follows from \cite[Propositions~7.1 and 7.2]{LeGall99}, that the inclusion map $i_u$ induces group homomorphisms
$$i_{X,u}^*:\mathrm{KK}^G(C_0(X),A)\rightarrow \mathrm{KK}^{G_u^u}(C_0(X_u),A_u)$$ for every locally compact $G$-space $X$. If $X$ is proper and cocompact, then $X_u$ is a proper and cocompact $G_u^u$-space. Hence we obtain maps $\mathrm{KK}^G(C_0(X),A)\rightarrow \K_*^{\mathrm{top}}(G_u^u;A_u)$. One easily checks, that these commute with the connecting maps coming from continuous $G$-maps $X\rightarrow Y$ for two proper $G$-compact $G$-spaces $X$ and $Y$. Consequently, taking the limit over all proper and $G$-compact subspaces $X\subseteq \mathcal{E}(G)$, we obtain the desired homomorphism $i_u^*:\K_*^{\mathrm{top}}(G;A)\rightarrow \K_*^{\mathrm{top}}(G_u^u;A_u)$.
In order to obtain commutativity of the diagram in the proposition, it is enough to observe that the following diagram commutes:
\begin{center}
\begin{tikzpicture}[description/.style={fill=white,inner sep=2pt}]
	\matrix (m) [matrix of math nodes, row sep=3em,
	column sep=1.5em, text height=1.5ex, text depth=0.25ex]
	{ \mathrm{KK}^G(C_0(X),A) & \mathrm{KK}^{G_u^u}(C_0(X_u),A_u)\\ \mathcal{R}\mathrm{KK}(G^{(0)};C_0(X)\rtimes_r G, A\rtimes_r G) & \mathrm{KK}(C_0(X_u)\rtimes_r G_u^u, A_u\rtimes_r G_u^u)\\
	 \K_0(A\rtimes_r G) &
	   \K_0(A_u\rtimes_r G_u^u)\\
	};
	\path[->,font=\scriptsize]
	(m-1-1) edge node[auto] {$ j_{G} $} (m-2-1)
	(m-3-1) edge node[auto] {$ q_{u,*} $} (m-3-2)
	(m-1-2) edge node[auto] { $ j_{G_u^u} $ } (m-2-2)
	(m-1-1) edge node[auto] { $ i_{X,u}^* $ } (m-1-2)
	(m-2-1) edge node[auto] {$ p_{G\ltimes X}\otimes\cdot $} (m-3-1)
	(m-2-2) edge node[auto] {$ p_{G_u^u\ltimes X_u}\otimes\cdot $} (m-3-2)
	(m-2-1) edge node[auto] {$ (i_u^{(0)})^* $} (m-2-2)
	;
	\end{tikzpicture}
\end{center}
The middle vertical map is induced by the inclusion map $i_u^{(0)}:\lbrace u\rbrace \hookrightarrow G^{(0)}$. Let us deal with the upper square first: Let $(E,\Phi,T)$ be a Kasparov triple in $\mathbb{E}^G(C_0(X),A)$. Recall, that $j_{G,r}$ sends the class of $(E,\Phi,T)$ to the class represented by $(E\rtimes_r G,\widetilde{\Phi},\widetilde{T})$.
Applying Proposition \ref{Prop:CrossedProductsByGroupBundles} and Proposition \ref{Prop:Bundles and Pullbacks} we obtain a canonical isomorphism
\begin{align*}
(E\rtimes_r G)_u&=(E\otimes_A (A\rtimes_r G))_u\\
&\cong E_u\otimes_{A_u} (A\rtimes_r G)_u\\
&\cong E_u\otimes_{A_u} (A_u\rtimes_r G_u^u)\\
&=E_u\rtimes_r G_u^u,\end{align*}
which intertwines the representations $(\widetilde{\Phi})_u$ and $\widetilde{\Phi_u}$ and the operators $(\widetilde{T})_u$ and $\widetilde{T_u}$.

In order to prove commutativity of the lower square we first fix a cut-off function $c$ for $G\ltimes X$. Then its restriction to the subspace $X_u$ is easily checked to be a cut-off function for $G_u^u\ltimes X_u$. It follows, that if $p:=p_{G\ltimes X}$ is the canonical projection associated to $c$, then $p(u)\in C_0(X_u)\rtimes_r G_u^u$ is the projection associated to the restriction of $c$ to $X_u$.
Now let $(E,\Phi,T)$ be the representative of an element $x\in\mathcal{R}\mathrm{KK}(G^{(0)},C_0(X)\rtimes_r G, A\rtimes_r G)$.
Recall, that under the identification $\K_0(C_0(X)\rtimes_r G)\cong \mathrm{KK}(\CC,C_0(X)\rtimes_r G)$ the class of $p$ is represented by the Kasparov tripel $(C_0(X)\rtimes_r G,\Phi_p,0)$, where $\Phi_p:\CC\rightarrow C_0(X)\rtimes_r G$ is given by $\Phi_p(1)=p$. Then the Kasparov product $p\otimes x\in \mathrm{KK}(\CC,A\rtimes_r G)$ can be represented by the tripel $(E\otimes_{q_u} (A_u\rtimes_r G_u^u),(\Phi\circ \Phi_p) \otimes 1,T\otimes 1)$.
On the other hand $(i_u^{(0)})^*(x)$ is represented by the tripel $(E_u,\Phi_u,T_u)$ and hence the product $p(u)\otimes(i_u^{(0)})^*(x)$ is represented by the triple $(E_u,\Phi_u\circ \Phi_{p(u)},T_u)$, where $\Phi_{p(u)}:\CC\rightarrow C_0(X_u)\rtimes_r G_u^u$ is again given by $1\mapsto p(u)$.
But by Remark \ref{Remark:Fibre of Hilbert module can be defined as tensor product} there is a canonical isomorphism
$E\otimes_{q_u} (A_u\rtimes_r G_u^u)\rightarrow E_u$ and one easily checks on elementary tensors, that this isomorphism intertwines $(\Phi\circ \Phi_p) \otimes 1$ with $\Phi_u\circ \Phi_{p(u)}$ and $T\otimes 1$ with $T_u$.
\end{proof}

Let $G$ be an ample groupoid group bundle, which is strongly amena\-ble at infinity and let $A$ be a $G$-algebra.
Let $\sigma_A:\K_*(A\rtimes_r G)\rightarrow \K_*^{\mathrm{top}}(G;A)$ be the splitting homomorphism provided by Theorem \ref{Theorem:BC Injectivity  for amenable at infty grpds}. Then $\gamma_A:=\mu_A\circ \sigma_A$ is an idempotent endomorphism of $\K_*(A\rtimes_r G)$ such that $im(\gamma_A)=im(\mu_A)$. In particular, it follows that $G$ satisfies the Baum-Connes conjecture for $A$ if and only if $(1-\gamma_A)\K_*(A\rtimes_r G)=\lbrace 0\rbrace$.

Since $G$ is strongly amenable at infinity, it is exact. Hence the reduced crossed product $A\rtimes _r G$ is the algebra of $C_0$-sections of a continuous bundle of $C^*$-algebras over $G^{(0)}$ with fibres $(A\rtimes_r G)_u=A_u\rtimes_r G_u^u$. Let $q_u:A\rtimes_r G\rightarrow A_u\rtimes_r G_u^u$ be the corresponding quotient map.
Likewise, every group $G_u^u$ of the bundle $G$ is amenable at infinity. Hence by the same reasoning, we obtain idempotents $\gamma_{A_u}\in End(\K_*(A_u\rtimes_r G_u^u))$.
We shall need the observation, that the elements $\gamma_A$ and $\gamma_{A_u}$ are compatible:
%\begin{lemma}
%Let $G$ be an ample groupoid, which is strongly amenable at infinity. Given a short exact sequence $0\rightarrow I\rightarrow A\rightarrow A/I\rightarrow0$ of $G$-algebras, there exists an exact sequence
%\begin{center}
%	\begin{tikzpicture}[description/.style={fill=white,inner sep=2pt}]
%	\matrix (m) [matrix of math nodes, row sep=3em,
%	column sep=2.2em, text height=1.5ex, text depth=0.25ex]
%	{ (1-\gamma_I)K_0(I\rtimes_r G) & (1-\gamma_A)K_0(A\rtimes_r G) & (1-\gamma_{A/I})K_0(A/I \rtimes_r G) \\
%		(1-\gamma_{A/I})K_1(A/I\rtimes_r G) & (1-\gamma_A)K_1(A\rtimes_r G) & (1-\gamma_{I})K_1(I \rtimes_r G) \\
%	};
%	\path[->,font=\scriptsize]
%	(m-1-1) edge node[auto] {$  $} (m-1-2)
%	(m-1-2) edge node[auto] {$  $} (m-1-3)
%	(m-1-3) edge node[auto] {$  $} (m-2-3)
%	(m-2-2) edge node[auto] { $  $ } (m-2-1)
%	(m-2-3) edge node[auto] { $  $ } (m-2-2)
%	(m-2-1) edge node[auto] { $ $ } (m-1-1)
%	;
%	\end{tikzpicture}
%\end{center}
%\end{lemma}
\begin{lemma}\label{Lem:GammaElements are natural}
Let $G$ be a second countable ample groupoid group bundle, which is strongly amen\-able at infinity. If $A$ is a separable $G$-algebra and $q_u:A\rtimes_r G\rightarrow A_u\rtimes_r G_u^u$ denotes the canonical quotient map, then $q_{u,*}\circ \gamma_A=\gamma_{A_u} \circ q_{u,*}$.
\end{lemma}
\begin{proof}
Let $\pi_u:(A\otimes_{G^{(0)}}C_0(Y))\rtimes_r G\rightarrow (A_u\otimes C(Y_u))\rtimes_r G_u^u$ be the canonical quotient map. Then we have a commutative diagram:
\begin{center}
	\begin{tikzpicture}[description/.style={fill=white,inner sep=2pt}]
	\matrix (m) [matrix of math nodes, row sep=3em,
	column sep=3em, text height=1.5ex, text depth=0.25ex]
	{ \K_*(A\rtimes_r G) & \K_*((A\otimes_{G^{(0)}}C_0(Y))\rtimes_rG) &  \K_*^{\mathrm{top}}(G;A\otimes_{G^{(0)}} C_0(Y))\\		
		\K_*(A_u\rtimes_r G_u^u) &  \K_*((A_u\otimes C(Y_u))\rtimes_r G_u^u) & \K_*^{\mathrm{top}}(G_u^u;A_u\otimes C(Y_u)) \\
	};
	\path[->,font=\scriptsize]
	(m-1-1) edge node[auto] {$ q_{u,*} $} (m-2-1)
	(m-2-1) edge node[auto] {$ ((id_{A_u}\otimes 1)\rtimes_r G)_*  $} (m-2-2)
	(m-1-2) edge node[auto] { $ \pi_{u,\ast} $ } (m-2-2)
	(m-1-2) edge node[auto] { $ (\mu_{A\otimes C_0(Y)})^{-1} $ } (m-1-3)
	(m-2-2) edge node[auto] { $ (\mu_{A_u\otimes C(Y_u)})^{-1} $ } (m-2-3)
	(m-1-1) edge node[auto] { $ (p_A\rtimes_r G)_* $ } (m-1-2)
		(m-1-3) edge node[auto] { $ i_u^* $ } (m-2-3)
	;
	\end{tikzpicture}
\end{center}
Here, the first square commutes already at the level of the $*$-homo\-morphisms, since $p_A\rtimes_r G$ is a $C_0(G^{(0)})$-linear map with $(p_A\rtimes_r G)_u=(id_{A_u}\otimes 1_{C(Y_u)})\rtimes_r G_u^u$. The second square commutes by Lemma \ref{Lemma:AssemblyMapForGroupoidCommutesWithAssemblyMapForIsotropyGroups} applied to the $G$-algebra $A\otimes_{C_0(G^{(0)})}C_0(Y)$. For similar reasons, each square in the following diagram commutes:
\begin{center}
	\begin{tikzpicture}[description/.style={fill=white,inner sep=2pt}]
	\matrix (m) [matrix of math nodes, row sep=3em,
	column sep=3em, text height=1.5ex, text depth=0.25ex]
	{  \K_*^{\mathrm{top}}(G;A\otimes_{G^{(0)}} C_0(Y)) & \K_*^{\mathrm{top}}(G;A) & \K_*(A\rtimes_r G) \\
		 \K_*^{\mathrm{top}}(G_u^u;A_u\otimes C(Y_u)) & \K_*^{\mathrm{top}}(G_u^u;A_u) & \K_*(A_u\rtimes_r G_u^u) \\
	};
	\path[->,font=\scriptsize]
	(m-1-1) edge node[auto] {$ (p_A)_*^{-1} $} (m-1-2)
	(m-1-2) edge node[auto] {$ \mu_A  $} (m-1-3)
	(m-1-3) edge node[auto] {$ q_{u,*} $} (m-2-3)
	(m-2-1) edge node[auto] { $ (p_{A_u})_*^{-1} $ } (m-2-2)
	(m-2-2) edge node[auto] { $ \mu_{A_u} $ } (m-2-3)
	(m-1-1) edge node[auto] { $ i_u^* $ } (m-2-1)
	(m-1-2) edge node[auto] { $ i_u^* $ } (m-2-2)
	;
	\end{tikzpicture}
\end{center}
Since the composition of the upper (respective lower) rows of these diagrams is by definition $\gamma_A$ (respective $\gamma_{A_u}$), the result follows.
\end{proof}

\begin{satz}\label{Thm:BC for group bundles}
Let $G$ be a second countable ample group bundle, which is strongly amenable at infinity. Suppose $A$ is a separable $G$-algebra such that the associated $C^*$-bundle is continuous, and $G_u^u$ satisfies the Baum-Connes conjecture with coefficients in $A_u$ for all $u\in G^{(0)}$. Then $G$ satisfies the Baum-Connes conjecture with coefficients in $A$.

In particular, $G$ satisfies the Baum-Connes conjecture with trivial coefficients, whenever each of its isotropy groups does.
\end{satz}
\begin{proof}
By the above considerations, it is enough to show, that $(1-\gamma_A)\K_*(A\rtimes_r G)=\lbrace 0\rbrace$. To this end, let $x\in (1-\gamma_A)\K_*(A\rtimes_rG)$. By Lemma \ref{Lem:GammaElements are natural} we have $q_{u,*}(x) =q_{u,*}(1-\gamma_A)(x)=(1-\gamma_{A_u})(q_{u,*}(x))\in (1-\gamma_{A_u})\K_*(A_u\rtimes_r G_u^u)$. But the latter group is zero by our assumption, hence $q_{u,*}(x)=0$ for all $u\in G^{(0)}$.
By \cite[Lemma~3.4]{MR2010742} every $u\in G^{(0)}$ admits a compact neighbourhood $C$ of $u$, such that $q_{C,*}(x)=0$, where $q_C:A\rtimes_r G\rightarrow A_{\mid C}\rtimes_r G_{\mid C}$ denotes the map induced by restriction. Since $G^{(0)}$ is assumed to be totally disconnected, we can find a partition $G^{(0)}=\coprod_{i\in I} C_i$ into compact open sets $C_i$ such that $q_{C_i,*}(x)=0$ for all $i\in I$. As the cover is disjoint, we obtain a decompositon $A\rtimes_r G=\bigoplus_{i\in I} A_{\mid C_i}\rtimes_r G_{\mid C_i}$. Using the additivity of $\K$-theory, we see that the maps $q_{C_i}$ induce an isomorphism
$\K_*(A\rtimes_r G)\cong \bigoplus_{i\in I} \K_*(A_{\mid C_i}\rtimes_r G_{\mid C_i})$. Since $q_{C_i,*}(x)=0$ for all $i\in I$, we conclude $x=0$ as desired.
\end{proof}
Combining Theorem \ref{Thm:BC for group bundles} above with \cite[Theorem~3.10]{Tu12} we get the following Corollary:
\begin{kor}
	Let $G$ be a second countable ample group bundle, which is strongly amenable at infinity. Then $G$ satisfies the Baum-Connes conjecture with coefficients in all $G$-algebras $A$ whose associated bundle of $\mathrm{C}^*$-algebras is continuous if and only if $G_u^u$ satisfies the Baum-Connes conjecture with coefficients for all $u\in G^{(0)}$.
\end{kor}

\ \newline
{\bf Acknowledgments}. The content of this paper covers part of the main results of the authors doctoral thesis \cite{Bonicke}. He would like to thank his supervisor Siegfried Echterhoff for his support and advice and the anonymous referee for a careful reading and several suggestions for improvement.

\bibliographystyle{plain}
\bibliography{Literatur}

\end{document}